\documentclass{amsart}
\usepackage{amsmath, amssymb, amscd, amsthm, graphicx, color, comment}
\newtheorem{theorem}{Theorem}
\newtheorem*{theorem*}{Theorem}
\newtheorem{lemma}[theorem]{Lemma}
\newtheorem{corollary}[theorem]{Corollary}
\newtheorem{proposition}[theorem]{Proposition}
\newtheorem{conjecture}[theorem]{Conjecture}
\theoremstyle{definition}
\newtheorem{definition}{Definition}

\theoremstyle{remark}

\newtheorem*{remark}{Remark}

\DeclareMathOperator{\Exp}{Exp}

\DeclareMathOperator{\sech}{sech}
\DeclareMathOperator{\cosech}{cosech}
\DeclareMathOperator{\arctanh}{arctanh}
\DeclareMathOperator{\mf}{MFFF}

\newcommand{\cadlag}{c\`{a}dl\`{a}g }

\begin{document}
\begin{abstract}

We introduce a random finite rooted tree $\mathcal{C}$ which we call the steady state cluster. It is characterised by a recursive description: $\mathcal{C}$ is a singleton with probability $1/2$ and otherwise is obtained by joining by an edge the roots of two independent trees $\mathcal{C}'$ and $\mathcal{C}''$, each having the law of $\mathcal{C}$, then re-rooting the resulting tree at a uniform random vertex. 

We construct and study a stationary regenerative stochastic process $\mathcal{C}(t)$ on the space of finite rooted trees, which we call the steady state cluster growth process. It is characterized by a simple fixed-point property. Its stationary distribution is the law of the steady state cluster $\mathcal{C}$. We conjecture that $\mathcal{C}(t)$ is the local limit of the evolution of the cluster of a tagged vertex in the stationary state of the mean field forest fire model of R\'ath and T\'oth. In particular we conjecture that the size-biased cluster distribution in the stationary state of the R\'ath-T\'oth model converges to the law of the steady state cluster as the size of the model tends to infinity.

We couple the steady state cluster growth process to a growth process of genealogical trees, whose stationary measure is the law of the critical binary Galton-Watson tree.  We show how to compute many statistics of the steady state cluster growth process. We describe its explosions in terms of a L\'evy subordinator, using a state-dependent time change. 

 We give an alternative description of the steady state cluster as a multitype Galton-Watson tree with a continuum of types. We exhibit the steady state cluster conditioned on its size as a random weighted spanning tree of the complete graph equipped random edge weights with a simple explicit joint distribution. We describe the dynamics of the steady state cluster growth process in terms of coupled multitype Galton-Watson trees. The time-reversal of the steady state cluster growth process is realised as the component of a `uniform' vertex in a logging process of a critical multitype Galton-Watson tree conditioned to be infinite.

Finally we construct a stationary forest fire model on the infinite rooted tree $\mathbb{Z}^*$ with the property that the evolution of the cluster of the root is a version of the steady state cluster growth process. This infinite forest fire model is similar in spirit to Aldous' frozen percolation model on the rooted infinite binary tree. We conjecture that it is the local weak limit of the stationary R\'ath-T\'oth model.



\end{abstract}


\title[Steady state clusters]{Steady state clusters and the R\'ath-T\'oth forest fire model}


\author{Edward Crane}
\date{\today}
\address{Heilbronn Institute for Mathematical Research\\
University of Bristol \\
BS8 1TW\\
United Kingdom
}

\email{edward.crane@bristol.ac.uk }

 \maketitle  

 
\section{Introduction}
In this paper we introduce and study a random finite tree $\mathcal{C}$ which we call the \emph{steady state cluster}. Its law is characterized by the following simple recursive distributional equation (RDE).
With probability $1/2$, $\mathcal{C}$ is a singleton. Otherwise, $\mathcal{C}$ consists of two independent samples from the law of $\mathcal{C}$, joined by an edge that connects two vertices chosen  independently uniformly at random from their vertex sets.  We will see that this RDE has a unique solution, which is a critical random tree: it is almost surely finite but its expected size is infinite. The steady state cluster is intimately related to the critical binary Galton-Watson tree, which satisfies a similar RDE. 

A substantial part of the paper is devoted to an alternative description of the steady state cluster as a critical multitype Galton-Watson tree. When conditioned on its size it may also be seen as a random weighted spanning tree of the complete graph equipped with certain random edge weights. A key to the connection between these descriptions is to consider the stationary distribution of the local dynamics of a forest fire model. In particular we will show that the steady state cluster is the distribution of the connected component of the root vertex in a certain stationary forest fire model on the infinite tree $\mathbb{Z}^*$. 

Our motivation for studying the steady state cluster was to approximate the distribution of clusters in the stationary state of the \emph{mean field forest fire model} of R\'ath and T\'oth \cite{RT}.  This model is a  homogeneous \cadlag Markov process on the space of graphs on the finite vertex set $\{1, \dots, n\}$. We denote it by $\mf(n)$. The generator has two terms. The first term is the generator of the dynamical Erd\H{o}s-R\'enyi random graph process. Each possible edge that is not already present appears at rate $1/n$, independently. The second term is a Poisson rain of \emph{lightning strikes} on the vertices.  Each vertex is struck by lightning at times of a Poisson process of rate $\lambda_n$. The lightning strike processes of the vertices are independent and also independent of the edge arrival process. When lightning strikes a vertex, all the edges in the connected cluster of that vertex are instantaneously removed: we say that the cluster is \emph{burned}. The vertices survive the fire and continue to acquire edges following the dynamical Erd\H{o}s-R\'enyi transition rates. Edges that have been burned may appear again at later times.

 R\'ath and T\'oth studied the limit as $n \to \infty$ of the empirical size-biased cluster size distribution in $\mf(n)$, considered as a stochastic process. They concentrated on the asymptotic regime in which the total lightning rate $n\lambda_n$ tends to infinity but the lightning rate per site $\lambda_n$ tends to zero as $n \to \infty$. For suitable initial conditions, they showed that in this regime the limit of the cluster size distribution exists as a process as $n \to \infty$. The limit does not depend on the lightning rate except through the above conditions. The limit process is deterministic and exhibits the phenomenon of \emph{self-organized criticality} (SOC), in the following sense. There is a positive \emph{gelation} time, before which the cluster size distribution has an exponential tail and the fires have a negligible global effect. At the gelation time the cluster size distribution has a polynomially decaying tail, as it does at criticality in the Erd\H{o}s-R\'enyi model. But after the gelation time, the system remains critical: the cluster size distribution continues to have a polynomial tail and finite cluster sizes account for all the mass. In particular there is no giant component.

In \cite{CFT} we described an integer-valued continuous-time process that is the limit of the evolution of the \emph{size} of the cluster of a tagged vertex in $\mf(n)$, as $n \to \infty$. This local limit is an explosive continuous time Markov branching process in an environment that varies deterministically in time. It is regenerative, returning instantaneously to state $1$ at each explosion time. Its distribution at each fixed time agrees with the size-biased global cluster size distribution that defines the environment. This environment is the global limit of the size-biased cluster size distribution in $\mf(n)$  and was shown in \cite{RT} to be the unique solution of an infinite system of coupled ODEs subject to a conservativity condition, which they called the \emph{critical forest fire equations}. This system has a unique fixed point, or steady state. In this paper we aim to understand the \emph{graph structure} of the clusters in $\mf(n)$ for large $n$ when the system is close to this steady state.

Our main object of study, the steady state cluster growth process, is constructed in \S\ref{S: geometric cluster growth}. 
The construction does not rely on any results about $\mf(n)$
 but is motivated by the local limit described in \cite{CFT}. 
We start by defining a  regenerative cluster growth process of finite rooted trees growing in a constant and deterministic environment of rooted trees. The environment is described by a law $\mathcal{W}$ on the set of isomorphism classes of rooted finite trees.
 The cluster grows by coalescing with trees sampled from $\mathcal{W}$, at rate proportional to its own size, and if it explodes then it instantaneously returns to the singleton state. We show that there is a unique law $\mathcal{W}_0$ with the property that the stationary distribution of the cluster growth process in the environment $\mathcal{W}_0$ is equal to $\mathcal{W}_0$. This turns out to be
the law of the steady state cluster, as described above via an RDE. The steady state cluster growth process is defined to be the cluster growth process in the environment $\mathcal{W}_0$. The size distribution corresponding to $\mathcal{W}_0$ is the fixed point $(w_k)_{k=1}^\infty$ of the critical forest fire equations.  

 In~\S\ref{S: geometric cluster growth} we also describe an enriched process in which each cluster is decorated with a genealogical tree,  which is a rooted binary tree whose leaves are identified with the vertices of the cluster. The motivation from thinking about the local limit of $\mf(n)$ as $n \to \infty$ is that the genealogical tree records information about the  coalescences which formed the cluster. The genealogical tree of the steady state cluster also satisfies a fixed point equation, which identifies it as the critical binary Galton-Watson tree. We give a combinatorial description of the joint distribution of the steady state cluster and its genealogical tree. As a consequence we prove re-root invariance of the steady state cluster, and show that the steady state cluster is the unique solution of the RDE given in the first paragraph.

We \emph{expect} that the steady state cluster growth process arises as the local limit of the stationary states of the mean field forest fire model in the SOC regime, but we are not yet able to prove this.  $\mf(n)$ is aperiodic and irreducible, so it has a unique stationary law $\mathbb{P}_{stat}^n$ for each model size $n$. We expect that the steady state cluster growth process should be the local weak limit of $\mathbb{P}_{stat}^n$ as $n \to \infty$. However it is not currently known whether the empirical cluster size distribution under $\mathbb{P}_{stat}^n$ converges to a deterministic limit as $n \to \infty$. The convergence of solutions of the critical forest fire equations as $t \to \infty$ to the fixed point $(w_k)_{k=1}^\infty$ is also not known; as far as we know it may be the case that there are random non-constant solutions of the critical forest fire equations that have stationary law. Note that the fixed point $\left(w_k\right)_{k=1}^\infty$ does not satisfy the third moment condition in the main theorem of R\'ath and T\'oth \cite{RT}, so we cannot apply their results to a sequence of initial conditions that converge to the fixed point. Nevertheless, we conjecture that the stationary states do converge to the fixed point solution of the critical forest fire equations in the following sense. Let $v_k^n(t)$ be the random proportion of vertices at time $t$ that belong to clusters of size $k$, under the stationary law $\mathbb{P}^n_{stat}$.

\begin{conjecture}\label{C: Stationary limit conjecture} For every $\epsilon > 0$ and $T > 0$,
\begin{equation} \mathbb{P}_{stat}^n\left(\sup_{t \in [0,T]} \sup_{k \ge 1} \left|v_k^n(t) - w_k\right| > \epsilon\right) \to 0 \qquad \text{as $n \to \infty$.}\end{equation}
\end{conjecture}

In \S\ref{S: properties of steady state cluster growth} we derive some basic properties of the steady state cluster growth process.  We compute the distribution of the excess time to explosion for a cluster given its size. We obtain closed-form generating functions for various joint distributions of the cluster size, the remaining survival time of the cluster, the age of the root vertex and the degree of the root vertex. We find the joint distribution of the cluster size and the number of jumps in cluster size seen by the root vertex since its last fire. We find a scaling limit for the number of jumps needed to exceed a given large size. We examine the nature of the explosions, proving a limit theorem that describes a time-change of the size process in terms of an exponential functional of a L\'evy process. This result is analogous to the behaviour of the tagged fragment in a self-similar fragmentation process. 

In~\S\ref{S: GW} we give an alternative description of the law of the steady state cluster, as a multitype Galton-Watson tree.  The type of a vertex in this tree corresponds to the age of a vertex in $\mf(n)$. The age is defined to be the length of the interval during which a vertex has survived unburned either since time $0$ or since its most recent fire. In $\mf(n)$, if we condition on the ages of all the vertices, then the set of edges present is the following inhomogeneous random graph. The probability that an edge is present between two vertices of ages $a$ and $b$ is $1 - \exp(-(a \wedge b)/n)$ and the states of all possible edges are independent. 
We expect that the empirical age distribution in the stationary state $\mathbb{P}^n_{stat}$ should converge to a certain deterministic distribution $\pi$. We therefore expect that the exploration process of the cluster of a tagged vertex in the stationary state of the mean field forest fire model when $n$ is large should be approximated by a critical multitype Galton-Watson tree. For the purposes of the present paper this argument, due to Bal\'azs R\'ath and Dominic Yeo, is just motivation.  We define a multitype Galton-Watson tree $H$ \emph{without reference to $\mf(n)$}, and show using the characterizing RDE that $H$ has the law of the steady state cluster. We also define a stationary process $\mathcal{H}(t)$, marginally distributed like $H$, and show that it is a version of the steady state cluster growth process. We describe $\mathcal{H}(t)$ conditioned on its next explosion time as another multitype Galton-Watson tree size-biased by its total progeny. We give an alternative structural description of this distribution in terms of trees hanging from a finite spine.  We derive as a limit the distribution of the fires in the steady state cluster growth process. The fires are distributed as a random one-ended infinite rooted tree $\hat{H}^{(0)}$ that is built up from an infinite spine, along which the vertex ages comprise a Markov chain, with a copy of $H$ conditioned on its root age hanging off each spinal vertex. This random infinite tree is also the local limit of the steady state cluster conditioned to have size $k$, as $k \to \infty$. 

The culmination of the present paper, in \S\ref{S: stationary forest fire}, is the realization of the steady state cluster growth process as the cluster of the root vertex in a stationary forest fire model on an infinite tree. An informal description of this model is quite simple but the rigorous construction is complicated. The difficulty is similar to that encountered by Aldous in the construction of his frozen percolation model on the infinite rooted binary tree \cite{AldousFroz}.
We construct an essentially stationary forest fire model on the infinite rooted plane tree $\mathbb{Z}^*$. This has the property that the cluster of the root is a version of the steady-state cluster growth process. The future edge arrivals are described by Aldous' Poisson-weighted infinite tree. The model can be viewed as a candidate for the local weak limit of the R\'ath-T\'oth model in its steady state, \emph{as a process}. It satisfies an important necessary condition to be a local weak limit identified by Aldous and Steele \cite{AldousSteele}, namely that it is \emph{involution-invariant}, or equivalently \emph{unimodular}. In particular the steady state cluster itself is a re-root invariant random tree and the distribution of the fires as rooted one-ended infinite trees is unimodular. An interesting and probably difficult question is whether the infinite forest fire model is measurable on the sigma algebra generated by its edge arrival times. This is the \emph{endogeny problem}, which has not yet been resolved in the simpler case of Aldous' frozen percolation model on the infinite rooted binary tree \cite{AldousFroz}.

In a companion paper, we study the scaling limit of the steady state cluster conditioned to have $k$ vertices, as $k \to \infty$. We treat the cluster as a measured metric space equipped with the graph metric and the normalized counting measure on the vertices. We show there that when the metric is rescaled so that each edge has length $3/(2\sqrt{2k})$, the limit exists and is the Brownian continuum random tree (BCRT). The convergence is in distribution with respect to the Gromov-Prokhorov topology. 

\subsection{Acknowledgements}

The author would like to acknowledge helpful conversations on the material in this paper and the companion paper on the scaling limit of the steady state cluster with many people including Louigi Addario-Berry, Omer Angel, Nicolas Curien, Nic Freeman, Christina Goldschmidt, James Martin, Gr\'egory Miermont, Bal\'azs R\'ath, Oliver Riordan, B\'alint T\'oth and Dominic Yeo. The author would like to thank the Isaac Newton Institute for Mathematical Sciences, Cambridge, for its hospitality during the 2015 programme on Random Geometry, where some of this research was carried out.


\section{The geometric cluster growth process}\label{S: geometric cluster growth}

\subsection{Growing graphs in a constant environment} \label{SS: growing graphs}
 Fix a probability distribution $\mathcal{W}$ on the set $Graph_*$ of isomorphism classes of rooted unlabelled finite simple undirected graphs. 
  For a graph $G$ with root vertex $r$ we will write $[(G,r)]$ for the isomorphism class of $(G,r)$ as a rooted graph. We will treat the distribution $\mathcal{W}$ as an \emph{environment} in which to grow a random time-dependent rooted graph $\mathcal{C}^\mathcal{W}_1(t)$. The subscript $1$ indicates that the process is a graph of size $1$ at time $0$. To be precise, $\mathcal{C}^\mathcal{W}_1(\cdot)$ is a continuous-time  $Graph_*$-valued Markov process, defined on a random time interval $[0,t_\infty)$, where the distribution of $t_\infty$ depends on $\mathcal{W}$. When $\mathbb{E}(t_\infty) < \infty$ we can subsequently construct a stationary process $\mathcal{C}^\mathcal{W}(\cdot)$.
 
 We begin by constructing the process $\mathcal{C}^\mathcal{W}_1(\cdot)$.  Let $(C_i, r_i)_{i=1}^\infty$ be an i.i.d. sequence of rooted graphs such that the isomorphism class $[(C_1,r_1)]$ of $(C_1,r_1)$ has law $\mathcal{W}$. Let $\gamma_1, \gamma_2, \dots$ be an i.i.d. sequence of exponential random variables with mean $1$, independent of the sequence $(C_i,r_i)$. 
 Let $t_0 = 0$ and let $\mathcal{C}^\mathcal{W}_1(0)$ be a graph consisting of a single vertex $\rho$, the \emph{root}.
Now we define $t_i$ and $\mathcal{C}^\mathcal{W}_1(t_i)$ inductively for $i \in \mathbb{N}$. Given  $t_i$ and $\mathcal{C}^\mathcal{W}_1(t_i)$, we let $$t_{i+1} = t_i
 + \frac{\gamma_i}{|\mathcal{C}^\mathcal{W}_1(t_i)|}\,.$$
Define $\mathcal{C}^\mathcal{W}_1(t) =  \mathcal{C}^\mathcal{W}_1(t_i)$ for all $t \in [t_i, t_{i+1})$. 
Construct $\mathcal{C}^\mathcal{W}_1(t_{i+1})$ by joining $C_i$ to $\mathcal{C}^\mathcal{W}_1(t_i)$ by adding an edge from the root $r_i$ to a random vertex $x_i$, chosen uniformly from the vertices of $\mathcal{C}^\mathcal{W}_1(t_i)$, independently of the sequences $(\gamma_i)$ and $(G_i, r_i)$ and of all $x_j$ for $j < i$. Let $\rho$ be the root of $\mathcal{C}^\mathcal{W}_1(t_{i+1})$. We denote this joining construction as follows: \[ \mathcal{C}^\mathcal{W}_1\left(t_{i+1}\right) = \mathcal{C}^\mathcal{W}_1\left(t_i\right) {\coprod\atop{(x_i,r_i)}} C_i\,.\]

Let $t_\infty = \lim_{i \to \infty} t_i$. The process $\mathcal{C}^\mathcal{W}_1(\cdot)$ is defined only on $[0,t_\infty)$.  For times $s < t$, $ \mathcal{C}^\mathcal{W}_1(s)$ is naturally a subgraph of $\mathcal{C}^\mathcal{W}_1(t)$. If $t_{\infty} < \infty$ then we define $\mathcal{C}^\mathcal{W}_1(t_\infty^{-})$ to be the nested union of $\mathcal{C}^\mathcal{W}_1(t)$ over all $t < t_{\infty}$. Note that this is an infinite rooted graph.

The size process $|\mathcal{C}^\mathcal{W}_1(t)|$ is a time-homogeneous continuous time Markov branching process. There is a well-known necessary and sufficient condition for explosivity of such a process (see  \cite[Chapter 5, Thm 9.1]{Harris}). Let $$W(z) = \sum_{k=1}^\infty \mathbb{P}(|G_1| = k)$$ be the probability generating function for the sizes of graphs distributed according to the law $\mathcal{W}$. If the integral $\int_0^1 \frac{dz}{1-W(z)}$ diverges then $t_\infty = \infty$ a.s. Otherwise $t_\infty$ has an exponential tail, so $\mathbb{E}(t_\infty) < \infty$.
\begin{remark} The size process $|\mathcal{C}_1^\mathcal{W}(t)|$ can also be thought of as a time-change of a compound Poisson process whose jump distribution is the marginal size distribution corresponding to $\mathcal{W}$. We will exploit this connection, along with another time-change, in \S\ref{S: properties of steady state cluster growth}.
\end{remark}

\subsection{Construction of a stationary process}\label{SS: stationarization}
 Suppose that $\mathbb{E}(t_\infty) < \infty$. Then there exists a unique stationary regenerative Markov process $\mathcal{C}^\mathcal{W}(t)$ for $t \in \mathbb{R}$, whose state space and evolution are identical to those of $\mathcal{C}^\mathcal{W}_1$ except that at each explosion time the process instantaneously returns to the singleton graph. The sequence of explosion times is a stationary renewal process.  
 We construct $\mathcal{C}^\mathcal{W}(\cdot)$ by
 concatenating independent instances of $\mathcal{C}^\mathcal{W}_1(\cdot)$, using a standard procedure called \emph{stationarization}.   In order to start the process $\mathcal{C}^\mathcal{W}(t)$ in its stationary state at $t=0$ we first take a sample of $\mathcal{C}^\mathcal{W}_1(\cdot)$ that is size-biased in
 proportion to $t_{\infty}$. This is possible exactly when $\mathbb{E}(t_\infty) < \infty$. Let $s$ be a $U([0,1])$ random variable independent of $\mathcal{C}^\mathcal{W}_1(\cdot)$ and define $C(t) = \mathcal{C}^\mathcal{W}_1(t+st_{\infty})$ for $t \in
 [-st_\infty,(1-s)t_\infty)$. Finally, concatenate two independent i.i.d. sequences of instances of $\mathcal{C}^\mathcal{W}_1(t)$ before and after the interval $[-s,t_\infty
 + s)$, translated in time so that their domains of definition form a
 partition of $\mathbb{R}$.

\subsection{The unique fixed point}
\begin{lemma}\label{L: unique fixed point}
 There exists a unique probability distribution $\mathcal{W}_0$
 for which $\mathcal{C}^\mathcal{W}_1(t)$ almost surely explodes and the distribution of the isomorphism class of $\mathcal{C}^\mathcal{W}(0)$ is precisely
 $\mathcal{W}_0$. For $(G,r)$ distributed according to $\mathcal{W}_0$, the distribution of $|G|$ is given by 
\[ \mathbb{P}(|G| = k) \; =
\frac{2}{k}\binom{2k-2}{k-1}\,4^{-k}\,.\]
$\mathcal{W}_0$ is supported on the set $\mathcal{T}_*$ of isomorphism classes of
rooted trees.
\end{lemma}

\begin{proof}
Let $\mathcal{W}$ be a law on $Graph_*$, and let $(G,r)$ be a random finite rooted graph such that the isomorphism class $[(G,r)]$ is distributed according to $\mathcal{W}$. Write $w_i = \mathcal{W}(|G| = i)$. First we
will show that there is a unique size distribution $(w_i)_{i=1}^\infty$ such that  
$\mathcal{C}^\mathcal{W}_1(t)$ almost surely explodes, $\mathbb{E}(t_\infty) < \infty$ and the
resulting stationary distribution of $|\mathcal{C}^\mathcal{W}(0)|$ equals $(w_i)_{i=1}^\infty$.

Define \[p_r := \mathbb{P}(\exists t \in (t_0,t_\infty) \,:\, |\mathcal{C}^\mathcal{W}_1(t)| = r)\,.\] We have $p_1 = 1$. Considering the size of the last jump taken to reach size $r$ we have, for $r \ge 2$,
\begin{equation} \label{E: convolution for P and W} p_r = \sum_{i=1}^{r-1} p_{r-i} \, w_i\,.\end{equation}
Define the ordinary generating functions \[P(z) = \sum_{i=1}^\infty p_i z^i
\quad \text{and} \quad W(z) = \sum_{i=1}^\infty w_i z^i\,.\]
Both converge on $|z| < 1$ and (\ref{E: convolution for P and W}) gives, for $|z| < 1$,
\[P(z) = z + P(z) W(z)\,.\]
Conditioned on $\{|\mathcal{C}^\mathcal{W}(t)|: t \in (t_0,t_{\infty})\}$, the holding times form a sequence of independent exponential random variables, with the mean holding time at size $k$ being $1/k$. Therefore
\[\mathbb{E}(t_\infty) \;=\; \sum_{r=1}^\infty \frac{1}{r} p_r \;=\;
\lim_{z \nearrow 1} \int_{0}^z \frac{P(s)}{s}\, ds\;=\; \int_0^1 \frac{1}{1-W(z)} \,dz\]
Suppose that
the condition $\mathbb{E}(t_\infty) < \infty$ holds. Then the stationary process $\mathcal{C}^\mathcal{W}(t)$ can be constructed as described above so that it is defined for all $t \in \mathbb{R}$ and is
right-continuous. Let $\pi$ denote the stationary size distribution $$\pi_k :=  \mathbb{P}(|\mathcal{C}^\mathcal{W}(0)| = k)\,.$$
Then for $k \ge 2$ we have
\begin{equation}\label{E: pi equation} k\pi_k = \sum_{i=1}^{k-1} (k-i)\pi_{k-i}\,w_i\,.\end{equation}
Define $\Pi(z) = \sum_{i=1}^{\infty} \pi_{i}z^i$ for $|z| < 1$. Then~\eqref{E: pi equation} implies
\[ z\Pi'(z) - z\pi_1  = z\Pi'(z) W(z)\,,\]
for $|z| < 1$, and $\Pi(0) = 0$. Thus
\[\Pi(z) = \pi_1 \int_0^z \frac{1}{1-W(s)}\,ds\,.\]
Since $\lim_{z \nearrow 1} \Pi(z) = 1$, and $\int_0^1
(1-W(s))^{-1}\,ds < \infty$, we find $\pi_1 = \mathbb{E}(t_\infty)^{-1}$, and then \eqref{E: pi equation} determines $\pi$ inductively. 
 We have shown that for each distribution $\mathcal{W}$ for which
$\mathbb{E}(t_\infty) < \infty$ there exists a unique stationary distribution
$(\pi_i)$ for $|\mathcal{C}^\mathcal{W}(\cdot)|$.

The fixed-point condition on the size distributions is that $\pi_i = w_i$ for all $i$, which
leads to \begin{equation}\label{E: size distribution gf} W'(z) = \frac{w_1}{1 - W(z)}\,.\end{equation}
Using $W(0) = 0$ we integrate to get 
$W(z) - W(z)^2/2 = zw_1$. The condition $\lim_{z \nearrow 1} W(z) = 1$ fixes $w_1 = 1/2$.
Hence  $$W(z)^2 -2W(z) + z = 0\,,$$ which has the unique analytic solution
$W(z) = 1 - \sqrt{1-z}$ on the unit disc satisfying the condition $W(0) = 0$.
For this solution of \eqref{E: size distribution gf} the condition $\mathbb{E}(t_\infty) < \infty$ is indeed satisfied, since
\[ \mathbb{E}(t_\infty) \;=\;  \int_0^1 \frac{1}{1 - (1 - \sqrt{1-z})}\, dz\; = \; 2\; =\; \frac{1}{\pi_1} \;=\; \frac{1}{w_1}.\] 
We recover the fixed point size distribution $\left(w_k\right)_{k=1}^\infty$ from the Taylor series:
$$ 1 - \sqrt{1-z} \;=\; \sum_{k=1}^\infty w_k z^k \;=\;\sum_{k=1}^\infty \frac{2}{k}\binom{2k-2}{k-1}4^{-k} z^k\,.$$

To complete the proof of the lemma, we have to show that there is a unique probability distribution
$\mathcal{W}_0$ on $Graph_*$
such that if the environment is $\mathcal{W}_0$ then the distribution
of $\mathcal{C}^{\mathcal{W}_0}(0)$ is also $\mathcal{W}_0$. We have already shown that for such a
$\mathcal{W}_0$ the corresponding size distribution $(w_i)$ must be given by the generating
function $W(z) = 1 - \sqrt{1-z}$. In particular the probability that $\mathcal{C}^\mathcal{W}(0)$ is a singleton must be $\tfrac{1}{2}$.  It is easy to see that any such distribution $\mathcal{W}_0$ must be supported on the set of $\mathcal{T}_*$ of isomorphism classes of rooted trees, by considering a non-tree of minimal size in the support of $\mathcal{W}_0$: there is no way for such a cluster to be formed by joining two trees.

Considering now an arbitrary law $\mathcal{W}$ on $\mathcal{T}_*$, we next explain how for any rooted tree $(T,r) \in \mathcal{T}_*$ with $|T| \ge 2$
vertices the probability that $\mathcal{C}^\mathcal{W}(0)$ is isomorphic to $(T,r)$ can be computed from the
corresponding probabilities for isomorphism classes of rooted trees with
fewer than $|T|$ vertices.  The rate at which the stationary process $\mathcal{C}^\mathcal{W}(\cdot)$ leaves the state $(T,r)$ is simply $|T|\,\mathbb{P}(\mathcal{C}^\mathcal{W}(0) \cong (T,r))$, and this must equal the rate at which it enters the state $(T,r)$. By summing rates of possible jumps to  $(T,r)$, we obtain 
\begin{multline}\label{E: stationary tree distribution} |T|\,\mathbb{P}(\mathcal{C}^\mathcal{W}(0) \cong (T,r)) = \\ \sum_{[(A,a)] \in \mathcal{T}_*}\sum_{[(B,b)] \in \mathcal{T}_* } \mathbb{P}(\mathcal{C}^\mathcal{W}(0) \cong (A,a))\, \mathcal{W}([(B,b)])\, \frac{|\{ v \in A : (A {\coprod\atop{(v,b)}} B , a) \cong (T,r)\}|}{|A|}\,.\end{multline}
 There are only finitely many non-zero terms in this double sum, and they occur in cases where $|A| < |T|$. Therefore all of these probabilities may be computed inductively with respect to $|T|$.  
 
 Now we can solve the equation \begin{equation} \label{E: fp equation} \mathbb{P}(\mathcal{C}^{\mathcal{W}_0}(0) \cong (T,r)) = \mathcal{W}_0([T,r]) \end{equation} for all finite rooted  trees $(T,r)$ simultaneously as follows. Assign $\mathcal{W}_0$-mass $1/2$ to the class of the singleton graph. Let $\mathcal{T}_{n}$ be the set of isomorphism classes of rooted trees with exactly $n$ vertices. Then we make the following inductive definition, where the induction is with respect to $k = |T|$.
 \begin{multline}\label{E: tree distribution fp} \mathcal{W}_0 ([(T,r)]) = \\ \frac{1}{k} \sum_{i=1}^{k-1} \sum_{[(A,a)] \in \mathcal{T}_{i}}\sum_{[(B,b)] \in \mathcal{T}_{k-i}} \mathcal{W}_0([(A,a)])\, \mathcal{W}_0([(B,b)])\, \frac{|\{ v \in A : (A {\coprod\atop{(v,b)}} B , a) \cong (T,r)\}|}{|A|}\,.\end{multline}
 
Define $$\tilde{w}_k = \sum_{(T,r): |T| = k} \mathcal{W}_0([(T,r)])\,.$$ Then we have $\tilde{w}_1 = 1/2$ and for $k \ge 2$ we sum \eqref{E: tree distribution fp} over isomorphism classes of $(T,r)$ with $|T| = k$ to find
$$ \tilde{w}_k = \sum_{i=1}^{k-1} i \tilde{w}_i \tilde{w}_{k-i}\,.$$
As we have seen, the unique solution of these equations is $\tilde{w}_k = \frac{2}{k}\binom{2k-2}{k-1}4^{-k}$, and as this sequence sums to $1$, the inductive definition \eqref{E: tree distribution fp} defines a probability distribution on $\mathcal{T}_*$ as required. The rate at which $\mathcal{C}^{\mathcal{W}_0}$ enters the singleton state is $1/\mathbb{E}(t_\infty) = 1/2$. Equating this with the rate at which $\mathcal{C}^{\mathcal{W}_0}$ leaves the singleton state, we find $\mathbb{P}(\left|\mathcal{C}^{\mathcal{W}_0}\right| = 1) = 1/2$, so \eqref{E: fp equation} is satisfied for the singleton tree. Taking \eqref{E: stationary tree distribution} and \eqref{E: tree distribution fp} together, one can show by induction that \eqref{E: fp equation} is satisfied for all finite rooted trees.  
 
\end{proof}

\subsection{The steady state cluster: definitions and notation}
\begin{definition} We call the stationary process $\mathcal{C}^{\mathcal{W}_0}(t)$ the \emph{steady state cluster growth process}, and from now on we denote it simply as $\mathcal{C}(t)$. We call the random rooted tree $\mathcal{C}$ whose law is $\mathcal{W}_0$ the \emph{steady state cluster}. We denote by $\mathcal{C}_1(t)$ the process  $\mathcal{C}_1^{\mathcal{W}_0}(t)$ that starts as a singleton and is defined on the random time interval $[0,t_\infty)$. We denote by $\rho$ the root vertex of either of these rooted tree valued processes. 
  
  Let $\left(\theta_i\right)_{i = -\infty}^\infty$ be the doubly infinite sequence of explosion times of the stationary process $\mathcal{C}(t)$, where almost surely \[\dots < \theta_{-2} <  \theta_{-1} < \theta_0 < 0 < \theta_1 < \theta_2 < \dots\,.\]
We say that the \emph{age} of the root $\rho$ at time $t$ is \[a_\rho(t) := \min(\{t - \theta_i: t \ge \theta_i\})\,.\]   
 Note that $t_\infty$ has the law of $\theta_1$ conditioned on the event $|\mathcal{C}(0)| = 1$, and $\mathcal{C}_1$ has the law of the restriction of $\mathcal{C}(\cdot)$ to the time interval $[0,\theta_1)$ conditioned on $|\mathcal{C}(0)| = 1$. Extending this notation we write $\mathcal{C}_\ell(t)$ for the process that has the law of the restriction of $\mathcal{C}$ to the time interval $[0,\theta_1)$ conditioned on the event $|\mathcal{C}(0)| = \ell$. 
 Conditional on $|\mathcal{C}(0)|$ the distribution of $\theta_1$ is independent of the structure of $\mathcal{C}(0)$ as a rooted tree.  However, conditional on $|\mathcal{C}(0)|$ the structure of $\mathcal{C}(0)$ is not independent of $\theta_0$.
 
  For the rest of the paper, $\left(w_k\right)_{k=1}^\infty$ will denote the sequence $$w_k := \frac{2}{k}\binom{2k-2}{k-1}4^{-k}$$ and $W(z) := 1 - \sqrt{1-z}$ will denote its ordinary generating function.  From Stirling's approximation we find $$ w_k \sim \frac{k^{-3/2}}{2\sqrt{\pi}}\; \text{as $k \to \infty$.}$$ The law of $\mathcal{C}(\cdot)$ restricted to the time interval $[0,\theta_1)$ is the mixture of the laws of $\mathcal{C}_\ell(\cdot)$ for $\ell = 1, 2, 3, \dots$ with weights $w_\ell$.
 \end{definition}

Figure~\ref{F: example probabilities} shows the mass that the law of $\mathcal{C}$ assigns to each of the rooted trees with up to five vertices. Observe that for each unrooted tree the possible rooted versions of the tree occur in proportions consistent with choosing the root vertex uniformly at random from the vertices of the tree. This property is called re-root invariance and we will prove it in the next section.

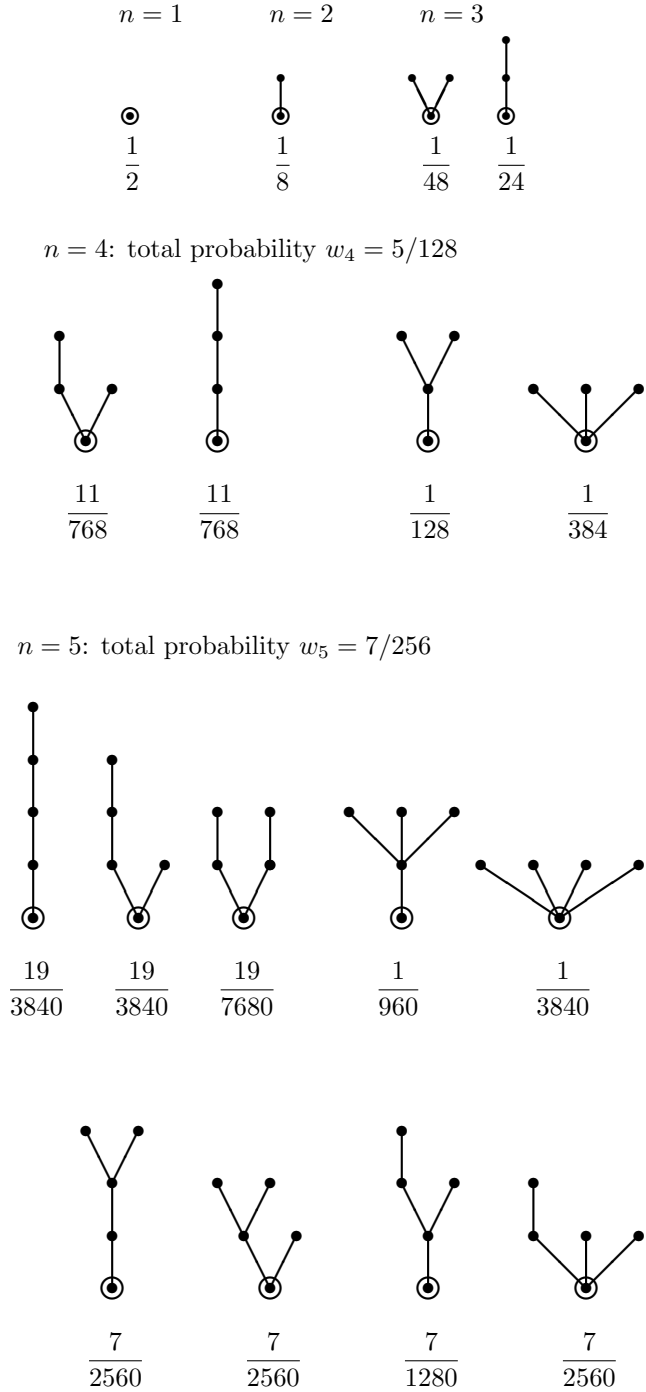
\begin{figure}\caption{The probabilities of clusters of up to $5$ vertices under $\mathcal{W}_0$, the law of $\mathcal{C}$.}\label{F: example probabilities}

\setlength{\unitlength}{0.5cm}
\begin{picture}(12,7)
\thicklines
\put(1,5.5){{$n=1$}}
\put(1.1,1.5){{$\displaystyle{\frac{1}{2}}$}}
\put(1.3,3){\circle*{0.2}}
\put(1.3,3){\circle{0.4}}

\put(5,5.5){{$n=2$}}
\put(5.1,1.5){{$\displaystyle{\frac{1}{8}}$}}
\put(5.3,3){\circle*{0.2}}
\put(5.3,3){\circle{0.4}}
\put(5.3,3){\line(0,1){1}}
\put(5.3,4){\circle*{0.2}}

\put(9,5.5){{$n=3$}}
\put(9,1.5){{$\displaystyle{\frac{1}{48}}$}}
\put(9.3,3){\circle*{0.2}}
\put(9.3,3){\circle{0.4}}
\put(9.3,3){\line(-1,2){0.5}}
\put(8.8,4){\circle*{0.2}}
\put(9.3,3){\line(1,2){0.5}}
\put(9.8,4){\circle*{0.2}}

\put(11,1.5){{$\displaystyle{\frac{1}{24}}$}}
\put(11.3,3){\circle*{0.2}}
\put(11.3,3){\circle{0.4}}
\put(11.3,3){\line(0,1){1}}
\put(11.3,4){\circle*{0.2}}
\put(11.3,4){\line(0,1){1}}
\put(11.3,5){\circle*{0.2}}

\end{picture}

\setlength{\unitlength}{0.7cm}
\begin{picture}(12,7)
\thicklines
\put(1,6.5){{$n=4$: total probability $w_4 = 5/128$}}

\put(1.4,1.5){{$\displaystyle{\frac{11}{768}}$}}
\put(1.8,3){\circle*{0.2}}
\put(1.8,3){\circle{0.4}}
\put(1.8,3){\line(-1,2){0.5}}
\put(1.3,4){\circle*{0.2}}
\put(1.3,4){\line(0,1){1}}
\put(1.3,5){\circle*{0.2}}
\put(1.8,3){\line(1,2){0.5}}
\put(2.3,4){\circle*{0.2}}

\put(3.9,1.5){{$\displaystyle{\frac{11}{768}}$}}
\put(4.3,3){\circle*{0.2}}
\put(4.3,3){\circle{0.4}}
\put(4.3,3){\line(0,1){1}}
\put(4.3,4){\circle*{0.2}}
\put(4.3,4){\line(0,1){1}}
\put(4.3,5){\circle*{0.2}}
\put(4.3,5){\line(0,1){1}}
\put(4.3,6){\circle*{0.2}}

\put(7.9,1.5){{$\displaystyle{\frac{1}{128}}$}}
\put(8.3,3){\circle*{0.2}}
\put(8.3,3){\circle{0.4}}
\put(8.3,4){\circle*{0.2}}
\put(8.3,3){\line(0,1){1}}
\put(8.3,4){\line(-1,2){0.5}}
\put(7.8,5){\circle*{0.2}}
\put(8.3,4){\line(1,2){0.5}}
\put(8.8,5){\circle*{0.2}}

\put(10.9,1.5){{$\displaystyle{\frac{1}{384}}$}}
\put(11.3,3){\circle*{0.2}}
\put(11.3,3){\circle{0.4}}
\put(11.3,3){\line(-1,1){1}}
\put(10.3,4){\circle*{0.2}}
\put(11.3,3){\line(0,1){1}}
\put(11.3,4){\circle*{0.2}}
\put(11.3,3){\line(1,1){1}}
\put(12.3,4){\circle*{0.2}}

\end{picture}

\setlength{\unitlength}{0.7cm}
\begin{picture}(13,9)
\thicklines
\put(1,8){{$n=5$: total probability $w_5 = 7/256$}}
\put(0.8,1.5){{$\displaystyle{\frac{19}{3840}}$}}
\put(1.3,3){\circle*{0.2}}
\put(1.3,3){\circle{0.4}}
\put(1.3,3){\line(0,1){1}}
\put(1.3,4){\circle*{0.2}}
\put(1.3,4){\line(0,1){1}}
\put(1.3,5){\circle*{0.2}}
\put(1.3,5){\line(0,1){1}}
\put(1.3,6){\circle*{0.2}}
\put(1.3,6){\line(0,1){1}}
\put(1.3,7){\circle*{0.2}}

\put(2.8,1.5){{$\displaystyle{\frac{19}{3840}}$}}
\put(3.3,3){\circle*{0.2}}
\put(3.3,3){\circle{0.4}}
\put(3.3,3){\line(-1,2){0.5}}
\put(2.8,4){\circle*{0.2}}
\put(2.8,4){\line(0,1){1}}
\put(2.8,5){\circle*{0.2}}
\put(2.8,5){\line(0,1){1}}
\put(2.8,6){\circle*{0.2}}
\put(3.3,3){\line(1,2){0.5}}
\put(3.8,4){\circle*{0.2}}

\put(4.8,1.5){{$\displaystyle{\frac{19}{7680}}$}}
\put(5.3,3){\circle*{0.2}}
\put(5.3,3){\circle{0.4}}
\put(5.3,3){\line(-1,2){0.5}}
\put(5.8,4){\circle*{0.2}}
\put(5.8,4){\line(0,1){1}}
\put(5.8,5){\circle*{0.2}}
\put(4.8,4){\circle*{0.2}}
\put(4.8,4){\line(0,1){1}}
\put(4.8,5){\circle*{0.2}}
\put(5.3,3){\line(1,2){0.5}}
\put(5.8,4){\circle*{0.2}}

\put(7.8,1.5){{$\displaystyle{\frac{1}{960}}$}}
\put(8.3,3){\circle*{0.2}}
\put(8.3,3){\circle{0.4}}
\put(8.3,4){\circle*{0.2}}
\put(8.3,3){\line(0,1){1}}
\put(8.3,4){\line(-1,1){1}}
\put(7.3,5){\circle*{0.2}}
\put(8.3,4){\line(0,1){1}}
\put(8.3,5){\circle*{0.2}}
\put(8.3,4){\line(1,1){1}}
\put(9.3,5){\circle*{0.2}}

\put(10.8,1.5){{$\displaystyle{\frac{1}{3840}}$}}
\put(11.3,3){\circle*{0.2}}
\put(11.3,3){\circle{0.4}}
\put(11.3,3){\line(-3,2){1.5}}
\put(9.8,4){\circle*{0.2}}
\put(11.3,3){\line(-1,2){0.5}}
\put(10.8,4){\circle*{0.2}}
\put(11.3,3){\line(1,2){0.5}}
\put(11.8,4){\circle*{0.2}}
\put(11.3,3){\line(3,2){1.5}}
\put(12.8,4){\circle*{0.2}}
\end{picture}

\setlength{\unitlength}{0.7cm}
\begin{picture}(12,7)
\thicklines

\put(1.8,1.5){{$\displaystyle{\frac{7}{2560}}$}}
\put(2.3,3){\circle*{0.2}}
\put(2.3,3){\circle{0.4}}
\put(2.3,3){\line(0,1){1}}
\put(2.3,4){\circle*{0.2}}
\put(2.3,4){\line(0,1){1}}
\put(2.3,5){\circle*{0.2}}
\put(2.3,5){\line(-1,2){0.5}}
\put(1.8,6){\circle*{0.2}}
\put(2.3,5){\line(1,2){0.5}}
\put(2.8,6){\circle*{0.2}}

\put(4.8,1.5){{$\displaystyle{\frac{7}{2560}}$}}
\put(5.3,3){\circle*{0.2}}
\put(5.3,3){\circle{0.4}}
\put(5.3,3){\line(-1,2){0.5}}
\put(4.8,4){\circle*{0.2}}
\put(4.8,4){\line(-1,2){0.5}}
\put(4.3,5){\circle*{0.2}}
\put(4.8,4){\line(1,2){0.5}}
\put(5.3,5){\circle*{0.2}}
\put(5.3,3){\line(1,2){0.5}}
\put(5.8,4){\circle*{0.2}}

\put(7.8,1.5){{$\displaystyle{\frac{7}{1280}}$}}
\put(8.3,3){\circle*{0.2}}
\put(8.3,3){\circle{0.4}}
\put(8.3,4){\circle*{0.2}}
\put(8.3,3){\line(0,1){1}}
\put(8.3,4){\line(-1,2){0.5}}
\put(7.8,5){\circle*{0.2}}
\put(8.3,4){\line(1,2){0.5}}
\put(8.8,5){\circle*{0.2}}
\put(7.8,5){\line(0,1){1}}
\put(7.8,6){\circle*{0.2}}

\put(10.8,1.5){{$\displaystyle{\frac{7}{2560}}$}}
\put(11.3,3){\circle*{0.2}}
\put(11.3,3){\circle{0.4}}
\put(11.3,3){\line(-1,1){1}}
\put(10.3,4){\circle*{0.2}}
\put(11.3,3){\line(0,1){1}}
\put(11.3,4){\circle*{0.2}}
\put(11.3,3){\line(1,1){1}}
\put(12.3,4){\circle*{0.2}}
\put(10.3,4){\line(0,1){1}}
\put(10.3,5){\circle*{0.2}}

\end{picture}

\end{figure}

\subsection{The genealogical tree of the steady state cluster}\label{SS: genealogical tree}
In this section we enrich the steady state cluster growth process by equipping it with a \emph{genealogical tree}. Recall that the steady state cluster is a candidate for the limit in distribution of the cluster of a tagged vertex in the stationary state of the mean field forest fire model on $n$ vertices, as $n \to \infty$. For a cluster $C$ in the stationary state of the mean field forest fire model, consider the collection of subsets of the vertex set of $C$ which at some time during the formation of $C$ formed a cluster. Call these \emph{subclusters}. The genealogical tree of the cluster is a rooted binary tree $G$ whose vertices correspond to the subclusters of $C$. The root vertex $r$ of $G$ corresponds to $C$ itself. Every singleton subset of the vertex set of $C$ is a subcluster and corresponds to a leaf of the binary tree $G$. Every non-singleton subcluster $S$ corresponds to a vertex of $G$ that has two child vertices, which correspond to the subclusters that were joined together by an edge arrival to form $S$. In this section we construct a candidate for the limit in distribution of the pair $(G,C)$  as $n \to \infty$.

 Instead of rooted binary trees it will be technically more convenient to use rooted plane binary trees. Let $\{0,1\}_*$ denote the infinite binary tree whose vertices are the finite strings over the alphabet $\{0,1\}$. For us, a rooted plane binary tree is a connected subgraph of $\{0,1\}_*$ that contains the root $\emptyset$ and in which every vertex $v \neq \emptyset$ has degree $1$ or $3$. That is, each vertex has either $0$ or $2$ children. The vertices of $\{0,1\}_*$ have a lexicographic order, which we will think of as left-to-right ordering.  In addition, we will label each vertex of the rooted plane binary tree with a non-negative real number, its \emph{spent time}. This label is intended to correspond to the age of the corresponding subcluster at the time of its coalescence with another subcluster, or, in the case of the root, the length of time since the whole cluster was formed.

In section \ref{SS: growing graphs} we considered an  environment of rooted graphs. We will now consider a richer environment of pairs $(G,C)$, where $G$ is a finite rooted plane binary tree with non-negative vertex labels and $C$ is a rooted tree whose vertex set is the set of leaves of $G$. We will refer to $C$ as a cluster and $G$ as the genealogical tree of the cluster. We write $\overline{G}$ for the rooted plane binary tree obtained from $G$ by forgetting the spent time labels. To relate this to our description of genealogical trees in $\mf(n)$, think of each vertex in $G$ corresponding to the set of leaves above it, which once formed a cluster. We will denote the root of $C$ by $\rho$ and the root of $G$ by $r$. Let $\mathcal{T}_{gen}$ be the space of such pairs $(G,C)$, with the obvious topology and Borel $\sigma$-algebra. Given an \emph{environment} that is a law $\mathcal{V}$ on $\mathcal{T}_{gen}$, we will now construct a continuous time jump process $(\mathcal{G}_1^\mathcal{V}(t),\mathcal{C}_1^\mathcal{V}(t))$ taking values in $\mathcal{T}_{gen}$. The construction is similar to the construction of $\mathcal{C}_1^\mathcal{W}(t)$ in section~\ref{SS: growing graphs}. 

Let $(G_i, C_i)_{i=1}^\infty$ be an i.i.d.~sequence drawn from the environment $\mathcal{V}$. Let $(\gamma_i)_{i=1}^\infty$ be an i.i.d.~sequence of exponential random variables with mean $1$ and let $(b_i)_{i=1}^\infty$ be an i.i.d.~sequence of Bernoulli(1/2) random variables. Let these three sequences be mutually independent. To start the enriched growth process at time $0$, let $\mathcal{G}_1^\mathcal{V}(0)$ be the rooted plane tree with one vertex labelled with spent time $0$, and $\mathcal{C}_1^\mathcal{V}(0)$ the singleton rooted tree consisting of this same vertex. The label of the root $r$ of $\mathcal{G}_1^\mathcal{V}(t)$ increases at rate $1$. When $\mathcal{G}_1^\mathcal{V}(t)$ has more than one vertex, it is only the root whose spent time label increases; the spent time labels of the other vertices are frozen. The rate at which $(\mathcal{G}_1^\mathcal{V}(\cdot),\mathcal{C}_1^\mathcal{V}(\cdot))$ jumps from any given state $(G,C)$ in $\mathcal{T}_{gen}$ is given by the number of leaves of $G$, which equals the number of vertices of $C$. To be precise, the $i^{th}$ jump occurs at time $t_i$, where $t_0 = 0$ and for $i \ge 1$,  $$t_{i} := t_{i-1} + \frac{\gamma_i}{|\{\text{leaves of } \mathcal{G}_1^\mathcal{V}(t_{i-1})\}|}\,=\,t_{i-1} + \frac{\gamma_i}{|\{| \mathcal{C}_1^\mathcal{V}(t_{i-1})\}|}.$$
The rooted plane tree  $\mathcal{G}_1^\mathcal{V}(t_i)$ is constructed as follows. Introduce a new root vertex $r_i$ with two children, and spent time label $0$. If $b_i = 0$ then the left child of $r_i$ is identified with the root of $\mathcal{G}_1^\mathcal{V}(t_{i-1})$, while the right child is identified with the root of $G_i$. If $b_i = 1$ the right child of $r_i$ is identified with the root of $\mathcal{G}_1^\mathcal{V}(t_{i-1})$, and the left child is identified with the root of $G_i$. Note that we are considering $G_1^\mathcal{V}(t_i)$ as a finite subtree of the infinite rooted plane binary tree $\{0,1\}_*$. To construct the tree $\mathcal{C}_1^\mathcal{V}(t_i)$, take the trees  $\mathcal{C}_1^\mathcal{V}(t_{i-1})$ and $C_i$ and join them by introducing an edge from the root of $C_i$ to a vertex of $\mathcal{C}_1^\mathcal{V}(t_{i-1})$ chosen uniformly at random, independently of all other variables. The root of $\mathcal{C}_1^\mathcal{V}(t_i)$ is the vertex that was the root of $\mathcal{C}_1^\mathcal{V}(t_{i-1})$.

\begin{figure}\caption{An example of a jump of $\left(\mathcal{G}_1^\mathcal{V}(\cdot),\mathcal{C}_1^\mathcal{V}(\cdot)\right)$
 }\label{F: example jump}
\setlength{\unitlength}{0.6cm}
\begin{picture}(14,11)
\thicklines

\put(0.6,6.5){{$\mathcal{C}_1(t_i^{-})$}}
\put(1.1,8){\circle{0.4}}
\put(1.1,8){\circle*{0.2}}
\put(1.1,8){\line(0,1){1}}
\put(1.1,9){\circle*{0.2}}
\put(1.1,9){\line(-1,2){0.5}}
\put(0.6,10){\circle*{0.2}}
\put(1.1,9){\line(1,2){0.5}}
\put(1.6,10){\circle*{0.2}}

\put(5.0,6.5){{$C_i$}}

\put(4.8,8){\circle*{0.2}}
\put(4.8,8){\line(1,0){1.0}}
\put(5.8,8){\circle*{0.2}}
\put(4.8,8){\circle{0.4}}

\put(0.6,0.5){{$\mathcal{G}_1(t_i^{-})$}}
\put(1.1,2){\circle*{0.4}}
\put(1.6,5){\circle{0.4}}
\put(1.1,4){\line(-1,2){0.5}}
\put(1.1,4){\line(1,2){0.5}}
\put(1.6,5){\circle*{0.2}}
\put(0.6,5){\circle*{0.2}}
\put(1.1,2){\line(-1,2){0.5}}
\put(0.6,3){\circle*{0.2}}
\put(0.6,3){\line(-1,2){0.5}}
\put(0.1,4){\circle*{0.2}}
\put(0.6,3){\line(1,2){0.5}}
\put(1.1,4){\circle*{0.2}}
\put(1.1,2){\line(1,2){0.5}}
\put(1.6,3){\circle*{0.2}}

\put(5.0,2.5){{$G_i$}}
\put(5.3,4){\circle*{0.4}}
\put(5.8,5){\circle{0.4}}
\put(5.3,4){\line(-1,2){0.5}}
\put(4.8,5){\circle*{0.2}}
\put(5.3,4){\line(1,2){0.5}}
\put(5.8,5){\circle*{0.2}}

\put(3.0,5.9){$+$}
\put(7.0,6){\vector(1,0){2.0}}

\put(10.8,6.5){{$\mathcal{C}_1(t_i)$}}
\put(11.3,8){\circle{0.4}}
\put(11.3,8){\circle*{0.2}}
\put(11.3,8){\line(0,1){1}}
\put(11.3,9){\circle*{0.2}}
\put(11.3,9){\line(-1,2){0.5}}
\put(10.8,10){\circle*{0.2}}
\put(11.3,9){\line(1,2){0.5}}
\put(11.8,10){\circle*{0.2}}
\put(11.8,10){\line(1,0){1.0}}
\put(12.8,10){\circle*{0.2}}
\put(12.8,10){\line(1,0){1.0}}
\put(13.8,10){\circle*{0.2}}

\put(11.8,0){{$\mathcal{G}_1(t_i)$}}
\put(11.3,2){\circle*{0.2}}
\put(11.8,5){\circle{0.4}}
\put(11.3,4){\line(-1,2){0.5}}
\put(11.3,4){\line(1,2){0.5}}
\put(11.8,5){\circle*{0.2}}
\put(10.8,5){\circle*{0.2}}
\put(11.3,2){\line(-1,2){0.5}}
\put(10.8,3){\circle*{0.2}}
\put(10.8,3){\line(-1,2){0.5}}
\put(10.3,4){\circle*{0.2}}
\put(10.8,3){\line(1,2){0.5}}
\put(11.3,4){\circle*{0.2}}
\put(11.3,2){\line(1,2){0.5}}
\put(11.8,3){\circle*{0.2}}
\put(12.3,1){\circle*{0.4}}
\put(13.3,2){\circle*{0.2}}
\put(12.3,1){\line(-1,1){1.0}}
\put(12.3,1){\line(1,1){1.0}}
\put(13.3,2){\line(1,2){0.5}}
\put(13.3,2){\line(-1,2){0.5}}
\put(13.8,3){\circle*{0.2}}
\put(12.8,3){\circle*{0.2}}
\end{picture}
\end{figure}
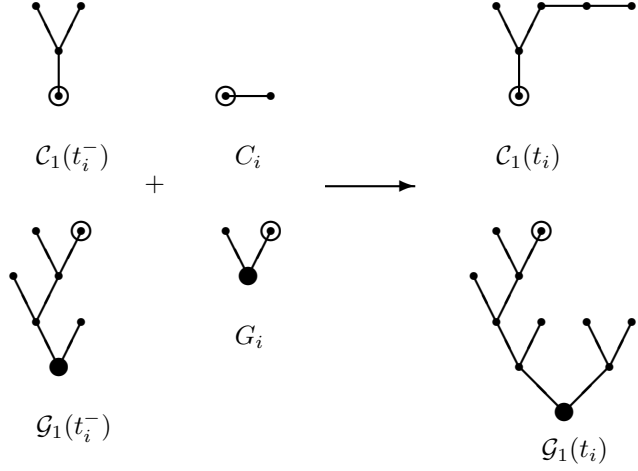

Notice that the sequence of times $(t_i)$ has a law that only depends on the common distribution of the size increments $|C_i|$. Likewise the growth process of rooted trees $\mathcal{C}_1^\mathcal{V}(t)$ on its own depends only on the marginal distribution $\mathcal{W}$ of $C_1$ as a rooted tree. The process $\mathcal{C}_1^\mathcal{V}(t)$ is therefore a version of the process $\mathcal{C}_1^\mathcal{W}(\cdot)$ constructed in \S~\ref{SS: growing graphs}, as is implied by the  notation.

 We let $t_\infty = \lim_{i \to \infty} t_i$.  If $\mathbb{E}(t_\infty) < \infty$ then we construct a corresponding stationary process $(\mathcal{G}^\mathcal{V}(t), \mathcal{C}^\mathcal{V}(t))$ by size-biased sampling and concatenation, exactly as we did to construct $\mathcal{C}^\mathcal{W}(t)$ in \S\ref{SS: stationarization}.

\begin{proposition}\label{P: unique joint law} There is a unique law $\mathcal{V}_0$ on $\mathcal{T}_{gen}$ such that $\mathbb{E}(t_\infty) < \infty$ and the law of $(\mathcal{G}^{\mathcal{V}_0}(0),\mathcal{C}^\mathcal{V}_0(0))$ is equal to $\mathcal{V}_0$. For $(G,C) \sim \mathcal{V}_0$, the genealogical tree $G$ has the law of the critical binary Galton-Watson tree with vertices labelled by spent times that are independent exponential random variables, conditional on the rooted plane tree structure of $G$. The mean of the spent time for a given vertex is the reciprocal of the number of leaves above that vertex. The cluster $C$ has the law of the steady state cluster. Given $G$, the root $\rho$ of $C$ is a uniform random leaf of $G$. For each non-leaf vertex $x$ of $G$, there is a unique edge $e_x$ of $C$ that joins a leaf of the left subtree above $x$ to a leaf of the right subtree above $x$; these two leaves are independent uniform random leaves from the these two subtrees, and the edges $e_x$ are mutually independent and independent of the choice of $\rho$. 
\end{proposition} 

\begin{proof}
The elements of $\mathcal{T}_{gen}$ carry enough information that for any fixed time $T$ the value of $\left(\mathcal{G}_1^{\mathcal{V}}(T), \mathcal{C}_1^{\mathcal{V}}(T)\right)$ determines the process $\left(\mathcal{G}_1^{\mathcal{V}}(\cdot), \mathcal{C}_1^{\mathcal{V}}(\cdot)\right)|_{[0,T]}$ up to time $T$. This was not the case for the process $\mathcal{C}_1^\mathcal{W}(\cdot)$ in \S\ref{SS: growing graphs}. This fact makes it easy to compute the probability densities recursively and to check that they agree with the claim.

Let $\mathcal{V}_0$ be any law for which $\mathbb{E}(t_\infty) < \infty$ and the law of $(\mathcal{G}^{\mathcal{V}_0}(0),\mathcal{C}^{\mathcal{V}_0}(0))$ is equal to $\mathcal{V}_0$. Let $(G,C) \sim \mathcal{V}_0$. Then the marginal law of $C$ must be equal to the law of the steady state cluster $\mathcal{C}$, by lemma~\ref{L: unique fixed point}. In particular the probability that $C$ is a singleton is $1/2$. Hence the probability that $G$ is a singleton is $1/2$. 
 We can now establish the marginal law of $(\overline{G},C)$ (forgetting the spent time labels) by induction on the size of $G$. We claim that the probability that $\overline{G}$ is a given finite binary subtree $g$ of $\{0,1\}_*$, rooted at $\emptyset$, is equal to $2^{-|g|}$, where $|g|$ is the number of vertices of $g$. This is true for $|g| = 1$. Suppose $|g| > 1$. Then the root of $g$ has a left child, with a subtree $g_{\text{left}}$ having $a$ leaves, and a right child, with subtree $g_{\text{right}}$ having $b$ leaves. There are two ways for $G^{\mathcal{V}_0}$ to enter the state $g$: either from $g_{\text{left}}$, at rate $a \mathcal{V}_0(G=g_{\text{left}})\mathcal{V}_0(G=g_{\text{right}})/2$, or from $g_{\text{right}}$, at rate $b \mathcal{V}_0(G=g_{\text{right}})\mathcal{V}_0(G=g_{\text{left}})/2$. The division by two comes from the fact that in each jump the previous value of the genealogical tree becomes the left subtree with probability $1/2$, and otherwise it becomes the right subtree.
On the other hand $G^{\mathcal{V}_0}$ exits state $g$ at rate $a+b$, which is the number of leaves of $G$. Hence by stationarity we obtain
$$ \mathcal{V}_0(G=g) = \mathcal{V}_0(G=g_{\text{left}})\mathcal{V}_0(G = g_{\text{right}})/2\,.$$
From this it follows immediately that $\mathcal{V}_0(G=g) = 2^{-|g|}$. Moreover, by comparing rates of entry to state $g$, we see that conditional on $G=g$, the probability that the root $\rho$ belongs to the subtree $g_{\text{left}}$ is $a/(a+b)$, and it follows by induction that conditional on $G$, the root $\rho$ is uniformly distributed among the leaves of $G$.  The complete description of $C$ now follows easily by induction over the size of $G$.

Finally we must establish that there is a unique conditional law of the spent time labels of $G$ given $\overline{G}$ for which the law of the spent time labels of $\mathcal{G}^\mathcal{V}_0(0)$ given $\overline{G}^\mathcal{V}_0(0)$ coincides with the environment.  The law of the spent time label of the root given $\overline{G}$ and all the other spent time labels only depends on the number of leaves $k$ of $G$. It is a product of a $U([0,1])$ random variable with a size-biased sample of an exponential holding time with mean $1/k$. This is simply an exponential random variable with mean $1/k$, and the claim now follows by induction over the size of $G$. 

\end{proof}

\begin{corollary}\label{Cor: re-root invariance} The distribution of $(\mathcal{C}, \rho)$ is re-root invariant. 
\end{corollary}

Re-root invariance is a property that we expected, given that we believe the process $\mathcal{C}_t$ is the local weak limit of the stationary states of $\mf(n)$, in which the vertices are exchangeable. But since we did not construct the steady state cluster growth process as a local weak limit of an exchangeable model, re-root invariance required proof. 

\begin{corollary}\label{C: RDE}
 The law $\mathcal{W}_0$ of $\mathcal{C}$ is the unique solution of the following RDE: $\mathcal{C}$ is a singleton with probability $1/2$ and otherwise is obtained by taking two independent rooted trees $\mathcal{C}'$ and $\mathcal{C''}$ distributed according to $\mathcal{W}_0$ and joining their roots by an edge, then selecting the root of $\mathcal{C}$ uniformly at random among the vertices resulting tree. 
\end{corollary}
\begin{proof}
 Let $(\mathcal{G},\mathcal{C})$ be the pair distributed as $\mathcal{V}_0$. To see that the law of $\mathcal{C}$ satisfies the RDE, look at the root $r$ of $\mathcal{G}$. It has no children with probability $1/2$, in which case $\mathcal{G}$ and $\mathcal{C}$ are singletons. Otherwise $r$ has two child subtrees $G_{\textup{left}}$ and $G_{\textup{right}}$ and $r$ corresponds to an edge $e$ in $\mathcal{C}$. The edge $e$ connects the subclusters $C_{\textup{left}}$ and $C_{\textup{right}}$ corresponding to the left and right child subtrees of $r$ by joining vertices $v_{\textup{left}}$ and $v_{\textup{right}}$ chosen independently uniformly at random from $C_{\textup{left}}$ and $C_{\textup{right}}$. By re-root invariance, we may take $v_{\textup{left}}$ to be the root of $C_{\textup{left}}$ etc. Then the pairs $(G_{\textup{left}}, C_{\textup{left}})$ and $(G_{\textup{right}}, C_{\textup{right}})$ are independent and each distributed according to $\mathcal{V}_0$.
 
 For uniqueness, note that for any solution $\mathcal{W}$ of the RDE, a sample may be taken from $\mathcal{W}$ by recursively applying the RDE. This recursion is indexed by a critical binary Galton-Watson tree. Since this is almost surely finite, it easily follows that $\mathcal{W} = \mathcal{W}_0$.

\end{proof}

The edge described in the RDE corresponds to the root of the genealogical tree. One can understand the RDE in the context of $\mf(n)$ by considering the \emph{youngest} edge in a size-biased cluster. 

\begin{definition}\label{D: age labels in C} Let $(\mathcal{G}_1(t), \mathcal{C}_1(t))$ and $(\mathcal{G}(t), \mathcal{C}(t))$ denote the processes described above in the environment $\mathcal{V}_0$, and write $(\mathcal{G},\mathcal{C})$ for the random variable whose law is $\mathcal{V}_0$. In particular $\mathcal{C}$ has the law of the steady-state cluster which we previously called $\mathcal{C}$. For each vertex $v$ in the cluster $\mathcal{C}$, the \emph{age} of $v$, denoted $a(v)$, is defined to be the sum of the spent times of the vertices along the unique path in $\mathcal{G}$ from $v$ to $r$. Likewise for each edge $e$ of $\mathcal{C}$, corresponding to a vertex $w$ of $\mathcal{G}$, the age of $e$, denoted $a(e)$, is defined to be the sum of the spent times of the vertices along the unique path in $\mathcal{G}$ from $w$ to $r$. (The endpoints are included in these summations.)
\end{definition} 
\begin{corollary}\label{Cor: sum of ages}
 Conditional on $|\mathcal{C}| = k$, the sum of the ages of the vertices in $\mathcal{C}$ is a $\Gamma(2k-1,1)$ random variable and the age of the youngest edge is an exponential random variable with mean $1/k$.  
\end{corollary}
\begin{proof} The spent time labels of the vertices of $\mathcal{G}(0)$ are independent exponential random variables.   For each vertex $s$ of $\mathcal{G}(0)$, that corresponds to a subcluster of size $m$, the spent time label of $s$ contributes to $m$ terms in the sum, and is an exponential random variable of mean $1/m$. Thus the sum of ages in the cluster is the sum of $2k-1$ independent exponential random variables of mean $1$. The age of the youngest edge is the spent time of the root, which is exponential with mean $1/k$ given $|\mathcal{C}| = k$.
\end{proof}

\section{Properties of the steady state cluster growth process}\label{S: properties of steady state cluster growth}

\subsection{The distribution of time to explosion}
Let $t_\infty$ be the explosion time of the process $\mathcal{C}_1(\cdot)$, and write $c_t := |\mathcal{C}_1(t)|$, defined for $0 \le t < t_\infty$. Let $p_r$ be the probability that there exists $t \in
(0, t_\infty)$ such that  $c_t
=r$. 
\begin{lemma}
 $p_r = \binom{2r-2}{r-1}4^{1-r}$ for each $r \ge 1$.
\end{lemma}
\begin{proof}
We have $p_1 = 1$. For $r \ge 2$, the probability that $c_t$ ever jumps from $r-k$ to $r$ is $p_{r-k} w_k$ and these jumps are mutually exclusive, so
\[p_r\;=\; \sum_{k=1}^{r-1} p_{r-k} w_k\,.\]
Let $P(z) = \sum_{r=1}^\infty p_r z^r$. Then   $P(z) = z + P(z) W(z)$, so $P(z) = z(1-z)^{-1/2}$ and the result follows. \end{proof}

By summing the expected holding time at size $r$  we recover
\begin{equation}\label{E: expected time to explosion} \mathbb{E}\left(t_\infty\right)\,=\,\sum_{r=1}^{\infty} \frac{1}{r}  \binom{2r-2}{r-1}4^{1-r} = 2\,.\end{equation}

A similar partition by the last jump to reach $r$ can be applied to compute the probability $p_{r,k}$ that
$c_t = r$ for some $s < t < t_\infty$, conditional on $t_\infty > s$ and $c_s = k$. For $r \ge k$ it is simply $p_{r-k+1}$.   Therefore
\begin{eqnarray*} \mathbb{E}(t_\infty - s\,|t_\infty > s\, \text{ and }\,c_s = k) & = & \sum_{r=k}^\infty \frac{1}{r} p_{r-k+1} \;=\; \int_0^1
\frac{z^{k-1}}{\sqrt{1-z}}
dz  \\ & = & \frac{4^k}{k \binom{2k}{k}}\, \sim \sqrt{\frac{\pi}{k}} \;
\text{as $k \to \infty$}.\end{eqnarray*}

Therefore have the following simple expression for the martingale
$\mathbb{E}(t_\infty | \mathcal{F}_t)$, where $\mathcal{F}_t$ is the filtration generated by the process $\mathcal{C}_1(\cdot)$.
\[\mathbb{E}(t_\infty | \mathcal{F}_t) \;=\; \begin{cases} t +
  \frac{4^{k}}{k \binom{2k}{k}}\,, \quad \text{if $t_{\infty} > t$ and
    $c_t = k$,
  }\\
  t_\infty\,, \quad\text{if $t_\infty \le t$}\,.\end{cases}\]

Note that the expected amount of time that $\mathcal{C}_1$ spends at size greater than $K$ between time $0$ and its explosion time $t_\infty$ is given by the tail of the series in equation (\ref{E: expected time to explosion}), which is
$O\left(K^{-1/2}\right)$.  

We now compute the distribution of $\theta_1$, the time of the first explosion after time $0$ of the stationary process $\mathcal{C}(t)$.

\begin{lemma}\label{L: time to next explosion}
 For $x \ge 0$ and $k \ge 1$,
\[\mathbb{P}(\theta_1 < x \,|\, |\mathcal{C}(0)| = k) \;=\; 1 - (\cosh\left(\tfrac{x}{2}\right))^{-2k}\,,\]
in particular $$\mathbb{P}(t_\infty < x) = \tanh^2\left(\tfrac{x}{2}\right)\,,$$
and
\[\mathbb{P}(\theta_1 < x) \;=\; \tanh\left(\tfrac{x}{2}\right)\,.\] 
Consequently $$\mathbb{E}(\theta_1) = 2 \log 2$$ and $$\mathbb{E}(\theta_1 \,|\, |\mathcal{C}(0)| = k) = \frac{4^k}{k\binom{2k}{k}}\,.$$
\end{lemma}
\begin{proof}
Let $$F(x):=\mathbb{P}(\theta_1 < x)$$ and for $k \ge 1$ let $$F_k(x) := \mathbb{P}(\theta_1 < x \,|\, |\mathcal{C}(0)| = k)\,.$$ In particular $F_1(x) = \mathbb{P}(t_\infty < x)$.  We will relate $F$ and $F_1$ in two ways:
\begin{equation}\label{E: claim 1} F(x) = \sqrt{F_1(x)}\,,\end{equation}
\begin{equation}\label{E: claim 2} 2 F'(x) = 1 - F_1(x)\,.
\end{equation}
To prove~\eqref{E: claim 1}, suppose $\mathcal{C}(0)$ has $k$ vertices. Then $\theta_1$ is the minimum of $k$ independent random variables, one for each vertex, each with cumulative distribution $F_1$. Hence
\[ F_k(x) \;=\; 1 - \left(1 -  F_1(t)\right)^k\,, \]
and
\begin{eqnarray*} F(x)  & =\ & \sum_{k=1}^\infty w_k \left(1 -
    \left(1-F_1(x)\right)^k\right) \\
 & = & \left(\sum_{k=1}^\infty w_k\right) \; - \; \sum_{k=1}^\infty w_k \left(1-
 F_1(x)\right)^k \\
 & = & 1 - \left(1 - \sqrt{1- \left(1-F_1(x)\right)}\right) \; = \; \sqrt{F_1(x)}\,.
\end{eqnarray*}

To prove~\eqref{E: claim 2}, note that $\Theta$ is a the jump process of a stationary renewal process whose inter-arrival times have the law of $t_\infty$. Therefore $\theta_1$ has the law of $U.S$ where $S$ is a size-biased sample from the law of $t_\infty$ and $U$ is a $U([0,1])$ random variable independent of $S$. Hence
\[ F(x) \;=\; \mathbb{P}(\theta_1 < x) \;=\; \frac{\int_{0}^x s\, dF_1(s)  +
\int_{x}^\infty \frac{t}{s} s\,dF_1(s)}{\int_0^\infty s\,
dF_1(s)}\,.\]
The denominator is
$\mathbb{E}(t_\infty) = 2$. Differentiating with respect to $x$ we obtain
\[2 F'(x) \; =\; x F_1'(x) - x F_1'(x) + \int_x^\infty F_1'(s) ds \;=\; 1 - F_1(x)\,. \] 
Combining~\eqref{E: claim 1}~and~\eqref{E: claim 2} we obtain the differential equation $$2F'(x) = 1 -
F(x)^2\,, \qquad F(0) = 0\,.$$ This has the unique solution
$F(x) = \tanh(x/2)$.  Hence $F_1(x) = \tanh^2(x/2)$ and 
\[ F_k(x) \;=\; 1 - \left(1-F_1(x)\right)^k \;=\; 1 - \cosh(x/2)^{-2k}\,. \]
Finally we integrate to obtain
\[ e_k := \mathbb{E}(\theta_1 | |\mathcal{C}(0)| = k) \;=\;\int_0^\infty \mathbb{P}(\theta_1 > x\, |\, |\mathcal{C}(0)| = k)\,dx\;=\;
\frac{4^k}{k\,\binom{2k}{k}}\,,\] \[ \mathbb{E}\left(\theta_1\, \mathbf{1}_{|\mathcal{C}(0)| = k}\right) \;=\; e_k w_k \;=\; \frac{1}{k(2k-1)}\,,\] and 
 \[ \mathbb{E}(\theta_1) = \sum_{k=1}^\infty e_k w_k = \int_0^\infty 1-F(x) = 2\log 2\,.\] 
\end{proof}
Straightforward Bayesian calculations, summations and integrations, and Stirling's approximation now yield the following; we omit the proofs.
\begin{corollary}\label{C: Expectations}
\begin{eqnarray*}\mathbb{E}(|\mathcal{C}(0)|\,\mathbf{1}(\theta_1 > t)) & = & \frac{1}{\sinh t}\,.\\
\mathbb{P}(|\mathcal{C}(0)| =  k\,|\, \theta_1 = t) & = & 2k\,w_k\, \sech^{2k-2)}\left(\tfrac{t}{2}\right)\,\tanh\left(\tfrac{t}{2}\right)\,.\\
\mathbb{E}(|\mathcal{C}(0)|\,|\, \theta_1 = t) & = & 1 + \frac{1}{2 \sinh^2 \tfrac{t}{2}} = \frac{\cosh t}{\cosh t - 1}\,.\\
\mathbb{E}(|\mathcal{C}(0)|^{-1} \,|\, \theta_1 = t) & = & 1 - e^{-t}\,.\\\mathbb{E}(1/\sinh(\theta_1/2)\,| \, |\mathcal{C}(0)| = k ) & = & k(2k-1)w_k \pi\,.
\end{eqnarray*}
We have the asymptotics \begin{eqnarray*}
\mathbb{E}(\theta_1 \,| \,|\mathcal{C}(0)| = k) & \sim & \sqrt{\pi/k}\\ \mathbb{E}(\theta_1^{-1} \,| \,|\mathcal{C}(0)| = k ) & \sim & \tfrac{1}{2}\sqrt{\pi k} \quad \text{ as } k \to \infty\,,\\ \mathbb{E}(|\mathcal{C}(0)|\,| \,\theta_1 = t) & \sim & 2 t^{-2} \\
\mathbb{E}(|\mathcal{C}(0)|^{-1} \,|\, \theta_1 = t) & \sim & t \quad \text{ as } t \to 0\,.
\end{eqnarray*}
The conditional distribution of $(\theta_1/2)^2\, |\mathcal{C}(0)|$ given $\mathcal{C}(0) = k$ converges to the law of a standard exponential random variable as $k \to \infty$: $$\mathbb{P}((\theta_1/2)^2 k > x \,|\, |\mathcal{C}(0)| = k)  = \sech^{2k}(\sqrt{x/k})\,\to\,e^{-x} \quad \text{as $k \to \infty$.}$$  
The conditional distribution of $\theta_1^2\,|\mathcal{C}(0)|/2$ given $\theta_1 = t$ converges to the law of the square of a standard normal random variable as $t \to 0$. 
\end{corollary}

\subsection{Counting jumps} Define $n(t)$ to be the number of jumps of the stationary process $\mathcal{C}(\cdot)$ in the time interval $(\theta_i, t]$ where $\theta_i \le t < \theta_{i+1}$.
\begin{lemma}
 \begin{equation}\label{E: number of jumps} \mathbb{E}\left(z^{|\mathcal{C}(0)|} x^{n(0)}\right)  \,=\, \frac{1}{x}\left(W(z) + \frac{1-x}{x}\,\log(1-xW(z))\right) \,.\end{equation}\end{lemma}
\begin{proof}
The pair $(|\mathcal{C}(t)|, n(t))$ is an $\mathbb{N} \times \mathbb{N}$-valued stationary Markov process. It enters the state $(1,0)$ at rate $1/2$. Equating rates of entry and exit from all the states we obtain
\[z\frac{\partial}{\partial z}\mathbb{E}\left(z^{|\mathcal{C}(0)|} x^{n(0)}\right) \,+\, \frac{z}{2} \,=\,0\,, \] with the boundary value $W(x)$ along $x=1$. The stated function is the unique solution.
\end{proof} 
\begin{corollary}
$$\mathbb{E}(n(0) \,|\,|\mathcal{C}(0)| = k) \,=\, \frac{1}{2kw_k} -1 \,\sim \frac{\sqrt{k}}{2\sqrt{\pi}}\;\text{as $k \to \infty$.}$$
 The stationary distribution of  $n(t)$ is the Yule-Simon distribution with parameter $1$:
$$\mathbb{P}(n(0) = n) = \frac{1}{(n+1)(n+2)}\,\quad \text{for $n \ge 0$.}$$
\end{corollary}
\begin{proof}
 Take the partial derivative of the expression \eqref{E: number of jumps} with respect to $x$ and evaluate at $x=1$, to get 
 $$\mathbb{E}(z^{|\mathcal{C}(0)|}\,n(0)) = -W(z) - \tfrac{1}{2}\log(1-z)\,.$$ Now extract coefficients and divide by $\mathbb{P}(|\mathcal{C}(0)| = k)$ to obtain the first equation. For the distribution of $n(0)$ we evaluate the expression \eqref{E: number of jumps} at $z=1$ and extract coefficients.
\end{proof}

\begin{lemma}\label{L: jumps to size n}
 Let $J_n$ be the number of jumps taken by $|\mathcal{C}_1(\cdot)|$ to exceed $n$. Then for each $\alpha > 0$  we have
 $$ \mathbb{P}(J_{n} \le \alpha \sqrt{n}) \to \rm{erf}(\alpha)\quad \text{ as $n \to \infty$.}$$
\end{lemma}
\begin{proof}
Let $X_1, X_2, \dots$ be the sizes of the sucessive jumps of $|\mathcal{C}_1(\cdot)|$. Then  $J_{n} > \alpha \sqrt{n}$ if and only if $X_1 + \dots + X_{\lfloor \alpha \sqrt{n}\rfloor} < n$. Now consider a simple symmetric random walk $Z_n$ on $\mathbb{Z}$ with $Z_0 = 0$. For each $k \ge 0$ let $T_k = \min\{n \ge 0 \,:\, Z_n = -k\}$ be the first hitting time of $-k$. Then $(T_1 - T_0, T_2-T_1, T_3-T_2, \dots)$ is identical in distribution to $(2X_1 - 1, 2 X_2 - 1, 2 X_3 - 1, \dots)$. Therefore
\begin{eqnarray*} \mathbb{P}(J_n > m) & = & \mathbb{P}(X_1 + \dots + X_m < n) \\ & = & \mathbb{P}(T_{m} \le 2n-m-1) \;=\; 2\, \mathbb{P}(Z_{2n-m-1} \le -m)\,,\end{eqnarray*} by the reflection principle. Setting $m = \lfloor \alpha \sqrt{n}\rfloor$, the normal approximation to the binomial distribution yields the result. 
\end{proof}

\subsection{Fluctuations on the way to explosion}
\begin{figure} 
\begin{center}
 \includegraphics[width=0.8\textwidth]{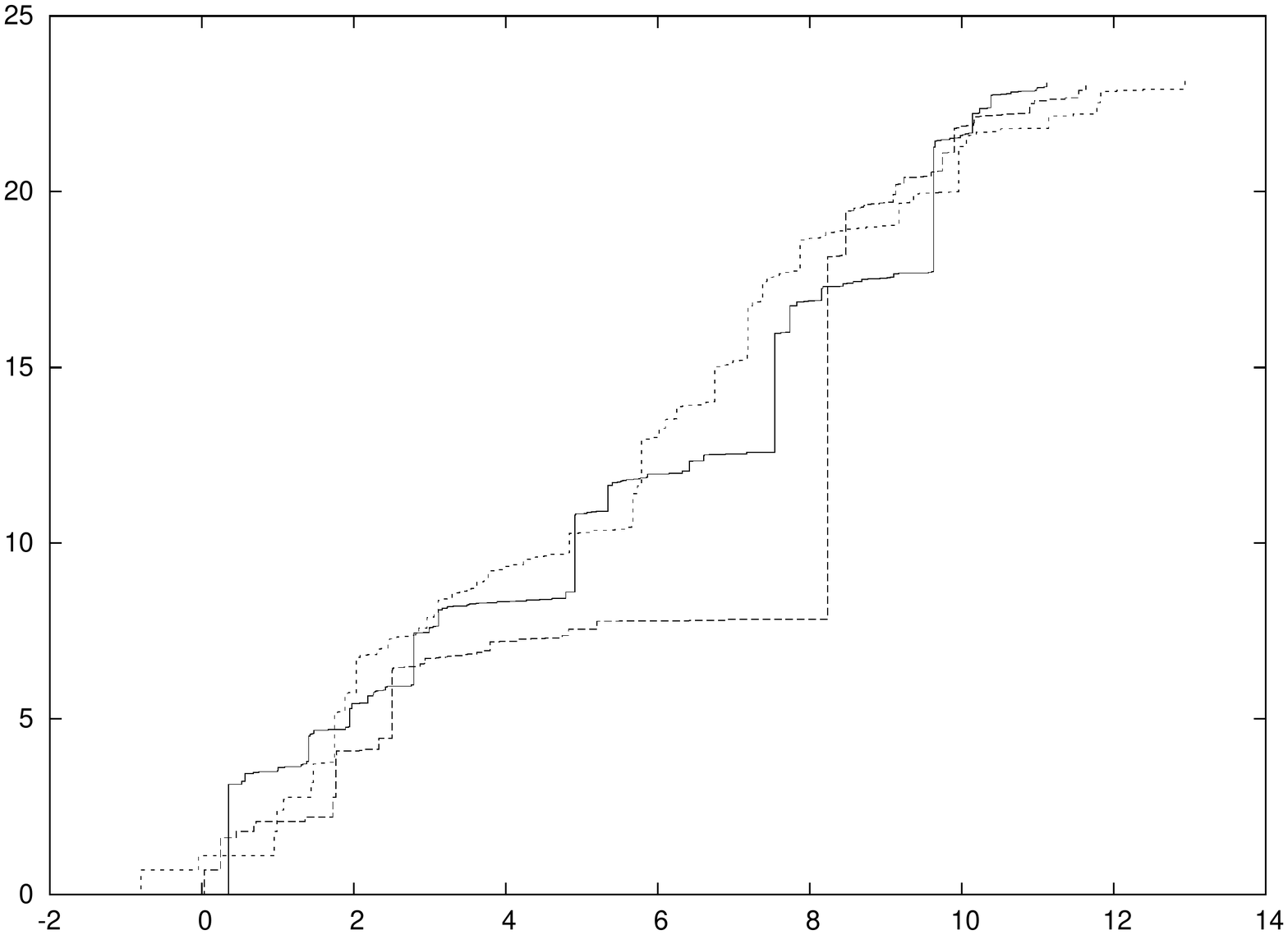}
\caption{Three simulated explosions: the $x$-axis is $-\log\left(t_\infty - t\right)$ and the
  $y$-axis is $\log\left(|\mathcal{C}_1(t)|\right)$.}\label{F: sample explosions}
 \end{center}
\end{figure}
Figure~\ref{F: sample explosions} is a plot of $\log |\mathcal{C}_1(t)|$ against
$-\log(t_\infty - t)$ for three simulated sample paths of $|\mathcal{C}_1(\cdot)|$. We will explain the features of this picture by proving a scaling limit theorem which approximates both $\log(t_\infty - t)$ and $\log |\mathcal{C}_1(\cdot)|$ by functionals of a certain L\'{e}vy subordinator. We begin by defining a random time change $\tau(t)$ for the process $\mathcal{C}_1(\cdot)$, which is closely related to $\log(t_\infty - t)$.
 $$\tau(t) := \int_0^t \sqrt{|\mathcal{C}_1(s)|}\,ds.$$ 
\begin{lemma} \label{L: tau diverges} Almost surely $\tau(t) \to \infty$ as $t \nearrow t_\infty$.
\end{lemma}
\begin{proof} As in the proof of Lemma~\ref{L: jumps to size n}, let $X_1, X_2, \dots$ be the random sequence of jumps of $|\mathcal{C}_1(\cdot)|$. This is an i.i.d. sequence with $\mathbb{P}(X_1 = k) = w_k$. Let $S_k = 1 + \sum_{i=1}^k X_i$ for $k \ge 0$. Then conditional on $(X_1, X_2, \dots)$, the limit as $t \to t_\infty$ of $\tau(t)$ is the sum of a sequence of independent exponential random variables with means $1/\sqrt{S_k}$ for $k = 0, 1, 2, \dots$. So it suffices to show that $\sum_{k=0}^\infty 1/\sqrt{S_k}$ is almost surely infinite. Let $E_n$ be the event that $|\mathcal{C}_1(\cdot)|$ hits $n$ and then takes at least $\sqrt{n}$ distinct values in the interval $[n+1,2n]$. From Lemma~\ref{L: jumps to size n} we see that there is a constant $c > 0$ such that for all $n \ge 1$, $\mathbb{P}(E_n \,|\, |\mathcal{C}_1(\cdot)|\; \text{ hits $n$}) > c$. Conditional on $|\mathcal{C}_1(\cdot)|$ hitting $n$, the event $E_n$ is independent of the jumps prior to hitting $n$. When $E_n$ occurs, the contribution to $\sum 1/\sqrt{S_k}$ from the passage through the interval $[n+1,2n]$ is at least $\sqrt{n}/\sqrt{2n}$. It is now straightforward to show that almost surely $E_n$ occurs for all $n$ in some random infinite sequence $n_1, n_2, \dots$ such that $n_{i+1} > 2 n_i$, and hence $\sum_{k=0}^\infty 1/\sqrt{S_k}$ is almost surely infinite.
\end{proof}
\begin{remark} The jump distribution is in the normal domain of attraction of the totally skewed stable law $S_{1/2}(1/4,1,0)$ of index $\tfrac{1}{2}$, also called a \emph{L\'evy distribution}. Thus $S_k / k^2$  has this law as its limiting distribution as $k \to \infty$, and the almost sure divergence of $\sum_{k=1}^\infty 1/\sqrt{S_k}$ may instead be proved by comparing the sequence $S_k$ to a stable subordinator of index $\tfrac{1}{2}$ and applying the law of the iterated logarithm for large times (see for example \cite[Ch.~3, Thm.~14]{Bertoin}).
\end{remark}
 Define $Z_\tau$ for every $\tau \in [0,\infty)$ by $Z_{\tau(t)} = \log |\mathcal{C}_1(t)|\,.$ According to Lemma~\ref{L: tau diverges}, this almost surely defines $Z_\tau$ for all $\tau \in [0,\infty)$.   Let $Y_\tau$ be the pure-jump subordinator given by the jump measure
$$\Pi(ds) = \frac{e^s\,ds}{2\sqrt{\pi}(e^{s}-1)^{3/2}}\,\mathbf{1}(s > 0)\,.$$ 
$Y_\tau$ is the \emph{Lamperti transformation} of the standard stable subordinator $S_t$ of index $1/2$. Lamperti's transformation is a correspondence between self-similar Markov processes and L\'evy processes, which in particular relates continuous state Markov branching processes with no killing to pure-jump subordinators. See for example \cite[Chap. 10]{Kyprianou}.
A calculation shows that $Y_\tau$ has Laplace exponent
\begin{equation}\label{E: Laplace exponent} \Phi(\lambda) = \frac{\Gamma(\lambda+\tfrac{1}{2})}{2\Gamma(\lambda)}\,.\end{equation}
\begin{proposition}\label{P: Levy} $Z_\tau$ and $Y_\tau$ may be coupled so that $Z_\tau - Y_\tau$ converges a.s. to a random finite value. The time remaining until explosion may be expressed as follows, where $\tau = \tau(t)$:
$$ t_\infty - t = \int_{\tau}^\infty e^{-Z_\upsilon/2} \,d\upsilon\, = e^{-Z_\tau/2} \int_{\tau}^\infty e^{-(Z_\upsilon - Z_\tau)/2}\,d\upsilon.$$ Hence \begin{equation}\label{E: asymptotic explosion process} |\mathcal{C}_1(t(\tau))|^{1/2}(t_\infty - t(\tau)) \; \sim \;\int_\tau^\infty e^{-(Y_\upsilon - Y_\tau)/2}  \,d\upsilon \end{equation} almost surely as $\tau \to \infty$. The integral in \eqref{E: asymptotic explosion process} is a stationary process with respect to the time $\tau$ and is independent of $(Y_s : s \in [0,\tau])$. \end{proposition}
\begin{remark}
 It is known how to compute moments of exponential functionals of L\'evy processes such as the integral on the right-hand side of \eqref{E: asymptotic explosion process}, by a theorem of Carmona et al \cite{CarmonaPetitYor}; see also the exposition in \cite{BertoinYor}. Since we have the explicit form \eqref{E: Laplace exponent} of the Laplace exponent of $Y_\tau$, one should in principle be able to compute that it has a (scaled) Rayleigh distribution, as implied by the final independence statement together with the penultimate part of Corollary 13. \end{remark}
\begin{proof}
When $e^{Z_\tau} = k$, the possible jumps of $Z_\tau$ are by $\log(1+\tfrac{n}{k})$ for $n \in \mathbb{N}$, and a jump by $\log(1+\tfrac{n}{k})$ occurs at rate $w_k \sqrt{k}$.
 We couple $Z_\tau$ and $Y_\tau$ in such a way that $Z_\tau$ jumps by $\log\left(1+n e^{-Z_\tau}\right)$ exactly when $Y_\tau$ jumps by an increment $\alpha$ where $\alpha \ge \log\left(1+\frac{1}{\pi e^{Z_\tau}}\right)$ and $n$ is the largest positive integer satisfying $$ e^{Z_{\tau}/2} \sum_{i=n}^\infty w_i \ge \frac{1}{\sqrt{\pi}\sqrt{e^\alpha-1}} = \int_\alpha^\infty \Pi(ds)\,.$$ 
This means that the total rate of jumps of $Z_\tau$ by increments of at least $\log\left(1+ n e^{-Z_\tau}\right)$ is exactly $e^{Z_{\tau}/2} \sum_{i=n}^\infty w_i$, as it should be. The remaining jumps of $Y_\tau$ are by increments smaller than $\log\left(1+\tfrac{1}{\pi e^{Z_\tau}}\right)$, and these are not coupled to any jumps of $Z_\tau$. We make two claims:
\begin{enumerate}\item There is a constant $c> 0$ such that for every possible simultaneous jump of $Y_\tau$ by $\alpha$ and of $Z_\tau$ by $\log\left(1+ n e^{-X_{\tau^-}}\right)$ we have $$|\alpha - \log\left(1+ n e^{-Z_{\tau^-}}\right)| \le \frac{c}{e^{Z_\tau}}\,.$$ \item Conditional on the path of $Z_\tau$, the expectation of the sum of the jumps of $Y_\tau$ that were not coupled to jumps of $Z_\tau$ is almost surely finite. \end{enumerate}
These two claims together imply that the total variation of $Z_\tau - Y_\tau$ is almost surely finite. The rest of the lemma follows easily.

\emph{Proof of claim 1:} We have $$\sum_{i=n}^\infty w_i = \binom{2n-2}{n-1} 4^{1-n} = \frac{1}{\sqrt{\pi n}}(1+O(1/n))\, \text{ as $n \to \infty$.}$$ Given simultaneous jumps of $Y_\tau$ by $\alpha$ and of $Z_\tau$ by $\log\left(1+ n e^{-Z_{\tau^-}}\right)$ we have
\begin{equation}\label{E: simultaneous jump bounds}\frac{1}{\sqrt{\pi n}} e^{-b/n} \le \frac{1}{\sqrt \pi (e^\alpha -1) e^{Z_{\tau^-}}} \le \frac{1}{\sqrt{\pi n}} e^{b/n}\,,\end{equation} for some constant $b > 0$.
In particular, $n \ge e^{-2b/n}(e^\alpha-1)e^{Z_{\tau^-}}$ 
so $$\frac{1-e^{-\alpha}}{n} \le e^{2b/n} e^{-\alpha}e^{-Z_{\tau^-}}  =O(e^{-Z_{\tau^-}})\,.$$
Rearranging~\eqref{E: simultaneous jump bounds} we find $$\log\left(1+\frac{n}{e^{Z_\tau}}\right) = \alpha + \log( 1 + O(1/n)(1-e^{-\alpha}))\, = \alpha + O\left(e^{-Z_\tau}\right)\,, $$ as required.

\emph{Proof of claim 2:} Conditional on $Z_\tau$, the component of the drift of $Y_\tau$ coming from jumps of size less than $u:= \log\left(1 + \frac{1}{\pi e^{Z_\tau}}\right)$ is 
$$\int_0^u \frac{s e^s \,ds}{2\sqrt{\pi}(e^s-1)^{3/2}} = -\log\left(1 + \frac{1}{\pi e^{Z_\tau}}\right) e^{Z_\tau/2} + \frac{2}{\sqrt{\pi}}\tan^{-1}\left(\frac{1}{\sqrt{\pi e^{Z_\tau}}}\right)\,.$$ As $Z_\tau \to \infty$, this is asymptotic to $\frac{1}{\pi} e^{-Z_\tau/2}$.  Given that $Z_\tau$ takes the value $\log k$ at some time,  the expected increment of $Y_\cdot$ due to the uncoupled jumps during the interval that $Z_\cdot$ spends at $\log k$ is therefore asymptotic to $1/(\pi k)$ as $k \to \infty$. It follows that the expected total increment from the uncoupled jumps of $Y_\tau$ is almost surely finite.

\end{proof}

\begin{corollary}\label{C: doublings}
 Let $1 < \lambda < \infty$. Almost surely there are infinitely many random jump times $t$ such that $|\mathcal{C}_1(t)| > \lambda |\mathcal{C}_1\left(t^{-}\right)|$. 
\end{corollary}

\begin{remark}
 Proposition~\ref{P: Levy} is strikingly similar to the description of the process describing the size of a tagged fragment in a self-similar fragmentation process (see Bertoin \cite[Theorem 3.2 and Cor. 3.1]{BertoinFrag}). For a binary fragmentation with dislocation measure $\sqrt{\tfrac{2}{\pi}}\,x^{-3/2} (1-x)^{-3/2}\,dx $ on $[\tfrac{1}{2},1]$, the process $\exp(-Y_\tau)$ describes the size of the fragment containing a tagged particle. See also Haas, \cite[\S4.3.1-2]{Haas} for a similar result for a fragmentation process which arises from a randomly cutting  a random tree. Later in \S\ref{SS: Dynamics conditioned on explosion time} we will see that one can view the time-reversal of $\mathcal{C}_1(\cdot)$ as a fragmentation process of a certain random infinite tree. 
\end{remark}

\subsection{Conditioning the size process on the next explosion time}
\begin{lemma}\label{L: Doob transform}
 The process $|\mathcal{C}(s)|_{s \in [0,t)}$ conditioned on $\theta_1 > t$, and the process $|\mathcal{C}_1(s)|_{s \in [0,t)}$ conditioned on $t_\infty > t$ are  time-inhomogeneous Markov jump processes whose jump rate from state $k$ to state $k+j$ at time $s$  is $$k\,w_j \,\left(\sech^{2}\left(\tfrac{t-s}{2}\right)\right)^j\,.$$  
 The process $|\mathcal{C}(s)|_{s \in [0,t)}$ conditioned on $\theta_1 = t$, and the process $c_s = |\mathcal{C}_1(s)|_{s \in [0,t)}$ conditioned on $t_\infty = t$ are  time-inhomogeneous Markov jump processes whose jump rate from state $k$ to state $k+j$ at time $s$ is $$(k+j)\,w_j \,\left(\sech^{2}\left(\tfrac{t-s}{2}\right)\right)^j\,.$$ 
\end{lemma}
\begin{proof} Immediate from Lemma~\ref{L: time to next explosion}. \end{proof}
Lemma~\ref{L: Doob transform} implies that the drift of $|\mathcal{C}(s)|$ conditioned on $\theta_1 > t$ is $$ \frac{|\mathcal{C}(s)|}{\sinh(t-s)},$$ and likewise for $|\mathcal{C}_1(s)|$ conditioned on $t_\infty > t$. Solving the resulting differential equation we can find the expected size of both processes conditioned on surviving to time $t$.
\begin{corollary} For $0 \le s < t$,
\begin{equation} \label{E: expected size given survival to time t} \mathbb{E}(|\mathcal{C}_1(s)| \,|\, t_\infty > t) \,=\, \tanh\left(\tfrac{t}{2}\right) \coth(\tfrac{t-s}{2})\,,
\end{equation}
\begin{equation} \label{E: expected stationary size given survival to time t} \mathbb{E}(|\mathcal{C}(s)| \,|\,\theta_1 > t) \,=\,\frac{\coth\left(\tfrac{t-s}{2}\right)}{1+e^{-t}}\,.
\end{equation}
\end{corollary}
Note that these expectations have a simple pole at $s=t$ even though the conditioned processes are almost surely finite at time $t$. Similarly, the drift of $|\mathcal{C}(s)|$ conditioned on $\theta_1 = t$ is $$ \frac{1+|\mathcal{C}(s)|}{\sinh(t-s)} \,+ \,\frac{1}{2\sinh(t-s)\,\sinh^2(\tfrac{t-s}{2})} \,.$$ 
\begin{corollary} \label{C: expected size given explosion time} For $0 \le s < t$,
\begin{equation} \label{E: expected size given explosion time} \mathbb{E}(|\mathcal{C}_1(s)| \,|\,t_\infty = t) \,=\, -1 + \coth(\tfrac{t-s}{2})\left(\coth(t-s) + 2\tanh \tfrac{t}{2} - \coth t\right)
\end{equation}
and
\begin{equation} \mathbb{E}(|\mathcal{C}(s)| \,|\,\theta_1 = t) \,=\, -1 + \coth(\tfrac{t-s}{2})\left(\coth(t-s) + \tanh\tfrac{t}{2}\right)\,.
\end{equation}
\end{corollary}
These expectations have a double pole at $t=s$. 

\subsection{The distribution of $|\mathcal{C}_1(s)|$ conditioned on $t_\infty$} \label{S: age and cluster size}
\begin{lemma}\label{L: size given survival to time t}
\begin{equation}\label{E: size given survival to time t}  \mathbb{E}\left(z^{|\mathcal{C}_1(t)|}\,|\, t_\infty > t\right)\,=\, \frac{z}{\left(1 + \tanh(\tfrac{t}{2})\sqrt{1-z}\right)^2}\,. \end{equation}
\end{lemma}
\begin{proof} Define
\[ F(z,t) := \mathbb{E}\left(z^{|\mathcal{C}_1(t)|} \mathbf{1}(t_\infty > t)\right)\,=\, \sum_{k=1}^\infty z^k \mathbb{P}( |\mathcal{C}_1(t)| = k \;\text{ and }\;
t_\infty > t)\,.\]
We consider the rates of entries and exits of $\mathcal{C}_1(\cdot)$ to and from the set of rooted trees of size $k$. Notice that the entry term is a convolution of the partial derivative $F_z$
with $W(z)$. This gives the PDE
\[ \frac{\partial F(z,t)}{\partial t} \;=\; -z \sqrt{1-z}\,\frac{\partial
  F(z,t)}{\partial z}\,.\]
The substitution $z = \sech^2(w/2)$, $G(w,t) = F(\sech^2(w/2),t)$ converts
this into the simple PDE 
\[ \frac{\partial G(w,t)}{\partial t} \; = \; \frac{\partial G(w,t)}{\partial w}\,.\]
This implies that $G(w,t)$ is constant along the lines $w+t = c$. The boundary conditions are \[ F(1,t) = \mathbb{P}(t_\infty > t) =
\sech^2(\tfrac{t}{2})\,, \quad F(z,0) = z\,.\]
 Converting these  into boundary conditions
for $G$, we find $G(0,t) = \sech^2(\tfrac{t}{2})$ and $G(w,0) = \sech^2(w/2)$, so  $G(w,t) \;=\;
\sech^2\left(\tfrac{w+t}{2}\right)$ and  
\[ F(z,t) \; = \; \frac{z}{\left(\cosh(\tfrac{t}{2}) +
    \sqrt{1-z}\sinh(\tfrac{t}{2})\right)^2}\,.\]
    \end{proof}
Applying a standard result in singularity analysis (see \cite[Cor. 2]{FO}) we deduce that for $t > 0$ we have
\begin{equation} \label{E: size asymptotics at time t} \mathbb{P}( t_\infty > t \;\text{ and }\;|\mathcal{C}_1(t)|= k ) \,\sim\, \frac{\sech^2 \tfrac{t}{2} \tanh \tfrac{t}{2}}{\sqrt{\pi}\, k^{3/2}}\; \text{as $k \to \infty$.}\end{equation}
and
$$ \mathbb{P}(|\mathcal{C}_1(t)|= k\,|\,t_\infty > t ) \,\sim\, \frac{\tanh \tfrac{t}{2}}{\sqrt{\pi}\, k^{3/2}}\; \text{as $k \to \infty$.} $$
Moreover, it follows from the singularity analysis that these asymptotics hold locally uniformly in $t > 0$.

We now compute the distribution of $|\mathcal{C}_1(s)|$ conditioned on $t_\infty = t$, for $0 \le s < t$. This coincides with the distribution of $|\mathcal{C}(0)|$ given both $\theta_0 = -s$ and $\theta_1 = t-s$. The density of $t_\infty$ given both $t_\infty > s$ and $\mathcal{C}_1(\cdot)|_{[0,s]}$ depends only on $|\mathcal{C}_1(s)|$, and from Lemma~\ref{L: time to next explosion} we can compute the density ratio
\begin{multline} \label{E: size at time s given explosion time} \mathbb{P}(|\mathcal{C}_1(s)| = k \,|\, t_\infty = t) = \\  \mathbb{P}(t_\infty > s \text{ and } |\mathcal{C}_1(s)| = k)\,\frac{\,k\, \sech^{2k}(\tfrac{t-s}{2}) \tanh(\tfrac{t-s}{2})}{\sech^2(\tfrac{t}{2})\tanh(\tfrac{t}{2})}\,.
\end{multline} 
Hence the probability generating function of $\mathcal{C}_1(s)$ conditioned on $t_\infty = t$ is 
\begin{multline*} \mathbb{E}(z^{|\mathcal{C}_1(s)|}\,|\,t_\infty = t) 
  =  \frac{\tanh(\tfrac{t-s}{2})}{\sech^2(\tfrac{t}{2})\tanh(\tfrac{t}{2})}\,\left(w \left.\frac{\partial}{\partial w}F(w,s)\right|_{\textstyle{w = z\sech^2(\tfrac{t-s}{2})}}\,\right). \end{multline*}  
By differentiating the latter expression with respect to $z$ it is possible to recover our earlier equation \eqref{E: expected size given explosion time} for $\mathbb{E}(|C_1(s)| \,|\, t_\infty = t)$. 

From \eqref{E: size asymptotics at time t} and \eqref{E: size at time s given explosion time} we obtain that for $0 < s < t$, as $k \to \infty$,
\begin{equation} \label{E: size tail given explosion time}\mathbb{P}(|\mathcal{C}_1(s)| = k\,|\, t_\infty = t) \,\sim\, \frac{\sech^2(\tfrac{s}{2}) \tanh(\tfrac{s}{2})\sech^{2k}(\tfrac{t-s}{2})\tanh(\tfrac{t-s}{2})}{\sqrt{\pi}\, k^{1/2} \sech^2(\tfrac{t}{2})\tanh(\tfrac{t}{2})}\,.\end{equation}
This asymptotic holds locally uniformly in $s$ and uniformly in $t-s$. 
Similarly, for $0 \le s \le t$,  
\begin{equation}\label{E: size given survival to a later time} \mathbb{P}(|\mathcal{C}_1(s)| = k \,|\, t_\infty > t) = \frac{\mathbb{P}(t_\infty > s \text{ and } |\mathcal{C}_1(s)| = k)\, \sech^{2k}\!\left(\tfrac{t-s}{2}\right)}{\sech^2{\tfrac{t}{2}}}\end{equation}
and hence 
\[\mathbb{E}\left(z^{|\mathcal{C}_1(s)|} \,| \, t_\infty > t\right) \,=\, \frac{F(z\,\sech^2\!\left(\tfrac{t-s}{2}\right),s)}{\sech^2 \tfrac{t}{2}}\,.\]
From~\eqref{E: size asymptotics at time t} and~\eqref{E: size given survival to a later time} we obtain
\[\mathbb{P}(|\mathcal{C}_1(s)| = k \,|\, t_\infty > t) \,\sim\, \frac{\sech^2(\tfrac{s}{2})\tanh(\tfrac{s}{2})\sech^{2k}\left(\tfrac{t-s}{2}\right)}{\sqrt{\pi}\,k^{3/2}\,\sech^2(\tfrac{t}{2})}\,\text{ as $k \to \infty$.}\]
Notice also that the conditional distribution of $|\mathcal{C}_1(s)|$ given $t_\infty = t$ is the size-biased version of the conditional distribution of $|\mathcal{C}_1(s)|$ given $t_\infty > t$. Later, in section~\ref{SS: GW tree given explosion time} we will describe explicitly the conditional distribution of $\mathcal{C}_1(s)$ as a rooted tree given $t_\infty > t$ or given $t_\infty = t$.

\subsection{The distribution of $|\mathcal{C}(s)|$ conditioned on $\theta_1$}
\begin{lemma} \label{L: stationary size conditioned on having survived for at least time t} For $t \ge 0$,
\begin{equation} \mathbb{E}\left(z^{|\mathcal{C}(t)|} \,|\, \theta_1 > t \right) = 
\frac{1 - \sqrt{1-z}}{1+\tanh(\tfrac{t}{2})\sqrt{1-z}}\,.
\end{equation}
\end{lemma}
\begin{proof}
Conditioning on $|\mathcal{C}(0)| = \ell$ instead of $|\mathcal{C}(0)| = 1$, we have
\begin{eqnarray*} F_\ell(z,t) &:=& \sum_{k=1}^\infty z^k \mathbb{P}( |\mathcal{C}(t)| = k \;\text{ and }\;
\theta_1 > t\,|\,|\mathcal{C}(0)| = \ell)\\ & = & F(z,t)^\ell\;=\;\frac{z^\ell}{\left(\cosh(\tfrac{t}{2}) +
    \sqrt{1-z}\sinh(\tfrac{t}{2})\right)^{2\ell}}\,.\end{eqnarray*}
Now we can compute the generating function analogous to $F(z,t)$ for the stationary process $\mathcal{C}(t)$:  
\begin{eqnarray*} H(z,t) & := & \mathbb{E}\left(z^{|\mathcal{C}(t)|}\,\mathbf{1}(\theta_1 > t)\right) \;=\;  \sum_{k=1}^\infty z^k\,\mathbb{P}( \theta_1 > t \, \text{ and }\, |\mathcal{C}(t)| = k) \\ & = & \sum_{k=1}^\infty \sum_{r=1}^\infty \mathbb{P}(|\mathcal{C}(0)| = r) \mathbb{P}(t_\infty > t\, \text{ and }\, |\mathcal{C}(t)| = k \,|\, |\mathcal{C}(0)| = r)\, z^k\\
 & = & \sum_{r=1}^\infty \mathbb{P}(|\mathcal{C}(0)| = r) \prod_{i=1}^r \sum_{k_i = 1}^\infty z^{k_i} \mathbb{P}(t_\infty > t\, \text{ and }\, c_t = k_i )\\
 & = & W(F(z,t))\,.
 \end{eqnarray*}
Continuing with the substitution $z = \sech^2(w/2)$, we have 
\[W(F(z,t)) \;=\; 1 - \sqrt{1 - \sech^2\left(\frac{w+t}{2}\right)} \;=\; 1 - \tanh\left(\frac{w+t}{2}\right)\]
Using the addition formula for $\tanh$, and noting that $\tanh(w/2) = \sqrt{1-z}$, we find
\[ H(z,t) \;=\; W(F(z,t)) \; = \;  \frac{(1-\sqrt{1-z})(1 - \tanh(\tfrac{t}{2}))}{1 + \tanh(\tfrac{t}{2})\sqrt{1-z}}\,. \]
\end{proof}
Similar calculations to those in the previous section yield
\begin{corollary}\label{C: size distribution asymptotics conditioned on next explosion}
 For $t \ge 0$, $$ \mathbb{P}(|\mathcal{C}(t)| = k \,|\, \theta_1 > t) \sim \frac{1+\tanh(\tfrac{t}{2})}{2\sqrt{\pi}\,k^{3/2}}\, \text{ as $k \to \infty$.}$$
 For $0 \le s \le t$,
 $$ \mathbb{P}(|\mathcal{C}(s)| = k \,|\, \theta_1 > t) \sim \frac{\sech^2(\tfrac{s}{2}) \sech^{2k}(\tfrac{t-s}{2})}{(1 - \tanh \tfrac{t}{2})\,.\,2\sqrt{\pi}\,k^{3/2}}\, \text{ as $k \to \infty$.}$$
 and
 $$ \mathbb{P}(|\mathcal{C}(s)| = k \,|\, \theta_1 = t) \sim  \frac{\sech^2(\tfrac{s}{2}) \sech^{2k}(\tfrac{t-s}{2}) \tanh(\tfrac{t-s}{2})}{\sech^2 \tfrac{t}{2}\,.\,\sqrt{\pi}\,k^{1/2}}\, \text{ as $k \to \infty$.} $$
\end{corollary}

\begin{lemma}\label{L: cluster size conditioned on age}
\[ \mathbb{E}(z^{|\mathcal{C}(0)|} \,| \,\theta_0 = -t)  = \mathbb{E}(z^{|\mathcal{C}(0)|}\,|\, \theta_0 \le -t \text{ \textup{and} } |\mathcal{C}(-t)| = 1)  =  \mathbb{E}(z^{|\mathcal{C}_1(t)|} \,|\, t_\infty > t) \,.\]
\end{lemma}
\begin{proof}
The first equality holds because $\mathcal{C}(\cdot)$ is a Markov process. For the second equality, note that  conditional on $\theta_0 = -t$, the process $\mathcal{C}(\cdot)$ restricted to the random interval $[-t,\theta_1)$ has the same law as $\mathcal{C}_1(\cdot+t)$ conditioned on $t_\infty > t$. 
\end{proof}
An interesting consequence of Lemmas~\ref{L: size given survival to time t},~\ref{L: stationary size conditioned on having survived for at least time t} and \ref{L: cluster size conditioned on age} is the limiting distribution of the cluster size given that the age of the root is large: 
$$ \lim_{t \to \infty} \mathbb{P}(|\mathcal{C}(0)| = k\,|\, \theta_0 = -t)
\;=\;\lim_{t \to \infty} \mathbb{P}(|\mathcal{C}(0)| = k\,|\, \theta_0 < -t)
 =\,  2 w_{k+1}\, .$$

\section{The steady state cluster as a multitype Galton-Watson tree} \label{S: GW}

\begin{definition}\label{D: H}
We define a multitype Galton-Watson tree $H$ whose type space is $[0,\infty)$. 
 The root $\rho$ of $H$ is a vertex whose type is distributed according the measure $\pi$, the steady state age distribution, with density $\tfrac{1}{2} \sech^2\left(\tfrac{x}{2}\right)$ on $[0,\infty)$. The multiset of types of the offspring of an individual of type $s$ is a Poisson random measure (PRM) of intensity $\lambda_s(dx) = (x \wedge s) \pi(dx)$. We consider $H$ as a rooted tree $H = (V,\rho,E)$ equipped with a function $a: V \to [0,\infty)$ giving the types of the vertices. We will refer to $a(v)$ as the age of $v$.  We also label each edge $e = (v,w)$ of $H$ by a random age $a(e)$ which, conditional on the ages $a(v), a(w)$ of the endpoints, is uniformly distributed in the interval $[0, a(v) \wedge a(w)]$. The edge ages are conditionally mutually independent given the tree and vertex ages. 

\end{definition}

Let us explain the meaning of the PRM in this definition. We consider the offspring of each vertex as an ordered sequence (not necessarily sorted by age). Then the density for the offspring sequence of a vertex of type $s$ to be $(x_1, x_2, \dots, x_k) \in [0,\infty)^k$ is
 \[ \frac{1}{k!} e^{-|\lambda_{s}|} \prod_{i=1}^k d\lambda_{s}(x_i)\,,\] where $|\lambda_s|$ denotes the total mass of $\lambda_s$. This density has support $[0,\infty)^k$. If we treat the offspring of a given vertex as a set rather than a sequence, then accounting for all the possible orderings, the density for the offspring set to be $\{y_1, \dots, y_k\}$, where $y_1 < y_2 < \dots < y_k$, is
\begin{equation}\label{E: PRM density} e^{-|\lambda_{s}|} \prod_{i=1}^k d\lambda_{s}(y_i)\,.\end{equation}    
The total mass of the measure $\lambda_s$ is $$ |\lambda_s| = \int_0^\infty (s \wedge x)  \frac{1}{2} \sech^2(x/2) \,dx \;=\; 2\log(1+\tanh(\tfrac{s}{2}))\,.$$ Thus the number of offspring of a vertex of type $s$ has the Poisson distribution with mean $2\log(1+\tanh(\tfrac{s}{2}))$, and the factor $\exp(-|\lambda_s|)$ appearing in the density \eqref{E: PRM density} is $(1+\tanh(\tfrac{s}{2}))^{-2}$.

\subsection{The distribution of the total progeny of $H$}

\begin{lemma}\label{L: total progeny}
 The total progeny $|H|$ has distribution
 \[\mathbb{P}(|H| = k) = w_k\,.\]
In particular $H$ is critical: it is almost surely finite but its expected size is infinite.
Conditioning on the age of the root, we have for $|z| < 1$
 \begin{equation}\label{E: cluster size given root age} \mathbb{E}\left(z^{|H|}\,| a(\rho) = x\right) \,=\, \frac{z}{\left(1 + \tanh(x/2)\sqrt{1-z}\right)^2}\,.\end{equation}
\end{lemma}
\begin{proof}
Let $\tilde{h}(x,z)$ denote the function on the right-hand side of~\eqref{E: cluster size given root age}, which could have been guessed from Lemma~\ref{L: cluster size conditioned on age}.  For $|z| < 1$, let $h(x,z) = \mathbb{E}\left(z^{|H|}\right| a(\rho) = x)$. Then for each value of $x \in [0,\infty)$ the function $h(x,\cdot)$ is analytic on the open unit disc. By 
conditioning on the offspring at the first generation we obtain the nonlinear recurrence relation
\begin{equation}\label{E: total progeny recurrence} h(x,z) \;=\; z \exp\left( - \int_0^\infty (1 - h(s,z)) (x \wedge s) d\pi(s)\right)\,.\end{equation}
To show that $\tilde{h}$ is a solution of~\eqref{E: total progeny recurrence} is a routine exercise in integration: split the integral at $x$, substitute $u = \tanh\left(\tfrac{s}{2}\right)$ and integrate by parts. 
 
Next, we will show that $\tilde{h}$ is the unique solution of \eqref{E: total progeny recurrence} when $z$ is real, positive, and sufficiently small. Then since $h(x,z)$ and $\tilde{h}(x, z)$ are analytic in $z$ for each fixed $x$, we deduce that $h = \tilde{h}$ for all $z$ in the disc $|z| < 1$.  Suppose $0 \le z < 1$. Then we have \emph{a priori} that $0 < h(x,z) \le z < 1$, so that \[ \int_0^\infty (1 - h(s,z))(x \wedge s) d\pi(s)\,> 0\,. \] The same statements hold with $\tilde{h}$ in place of $h$. Hence for every $x \in [0,\infty]$ we have
\begin{eqnarray*} |h(x,z)-\tilde{h}(x,z)| & \le & z \int_0^\infty |h(s,z) -\tilde{h}(s,z)| (x \wedge s)\,d\pi(s)
\\
& \le & z \| h(\cdot, z)  - \tilde{h}(\cdot,z)\|_\infty \int_0^\infty s \,d\pi(s) \\ & = &  2 z \log 2\,\| h(\cdot,z)  - \tilde{h}(\cdot,z)\|_\infty\,. \end{eqnarray*}
In particular for  $0 < z < 1/(2 \log 2)$ we must have $h(x,z) = \tilde{h}(x,z)$, as required.

Now that we have~\eqref{E: cluster size given root age}, we compute $$\mathbb{E}\left(z^{|H|}\right) \;=\; \int h(x,z) d\pi(x) \;=\; 1 - \sqrt{1-z}\,.$$ The integration here is straightforward using the substitution $u = \tanh(x/2)$. Hence $\mathbb{P}(|H| = k) = w_k$, as claimed.
\end{proof}

\subsection{The joint distribution of root age, root degree and total progeny}
\begin{lemma}
\begin{equation}\label{E: master generating function} \mathbb{E}(z^{|H|}s^{\deg(\rho)} | a(\rho) = x) = \frac{z}{(1+\tanh(x/2)\sqrt{1-z})^{2s}}.\left(\frac{2}{1+e^{-x}}\right)^{2(s-1)}\,.\end{equation}
\end{lemma}
\begin{proof}
$z^{|H|-1} s^{\deg(\rho)}$ is the product of $s z^{|H_i|}$ over the offspring of the root, where $H_i$ is the subtree rooted at the $i^{th}$ child of $\rho$. Since the types of the offspring of the root are a sample of the PRM with intensity $(x\wedge y) d\pi(y)$, we have
\[ \mathbb{E}(z^{|H|}s^{\deg(\rho)} | a(\rho) = x) = z \exp\left(\int_0^{\infty} (s\, h(y,z) - 1) (x \wedge y) \,d\pi(y)\right)\,,\]
where $$h(y,z) \,=\, \mathbb{E}\left(z^{|H|} | a(\rho) = y\right) \,=\, z\,\left(1+\tanh(y/2)\sqrt{1-z}\,\right)^{-2}$$ was computed in the proof of Lemma~\ref{L: total progeny}.
Using the recurrence \eqref{E: total progeny recurrence} that is satisfied by $h$, and computing 
\[\int_0^\infty (x \wedge y)\, d\pi(y) \,=\, x - 2 \log \cosh\left(\frac{x}{2}\right)\,=\, 2 \log\left(1+\tanh\left(\frac{x}{2}\right)\right) \,,\]
we obtain the result.
\end{proof}
Integrating equation \eqref{E: master generating function} against $d\pi(x)$, we find
\begin{equation}\label{E: root degree and cluster size}\mathbb{E}\left(z^{|H|} s^{\deg(\rho)}\right) = \frac{2}{2s-1} \left( \left(\frac{1+\sqrt{1-z}}{2}\right)^{2-2s} - \left(\frac{1+\sqrt{1-z}}{2}\right)\right)\,.\end{equation}
 Setting $s=1$ in~\eqref{E: master generating function} we recover Lemma~\ref{L: cluster size conditioned on age}. Setting $z=1$ in~\eqref{E: master generating function} we recover the fact (which follows from the description of the offspring by a PRM) that conditional on $a(\rho) = x$, the degree of the root has a Poisson distribution with mean $2\log\left(1 + \tanh\left(\tfrac{x}{2}\right)\right) = |\lambda_x|$. Thus the degree of the root of the steady state cluster has a compound Poisson distribution, where the mean of the random Poisson distribution is $2\log(1+U)$ where $U \sim U([0,1])$. Once we have shown that $H$ and $\mathcal{C}$ have the same law, it will also be possible to prove~\eqref{E: root degree and cluster size} directly by using the dynamics of $\mathcal{C}(t)$ to write down a recurrence relation for $w_{k,i} = \mathbb{P}(|\mathcal{C}| = k,\,\text{ and }\,\deg(\rho) = i)$, namely for $k \ge 2$ and $i \ge 1$,
\[
     0 \;=\; -k w_{k,i}\, +\,  \sum_{j=1}^{k-1} w_{j,i}\,  w_{k-j} (j-1) \,+\, \sum_{j=1}^{k-1} w_{j,i-1} w_{k-j}\,,
\]
 and converting this into a PDE for $\mathbb{E}\left(z^{|\mathcal{C}|} s^{\deg(\rho)}\right)$ whose unique solution is the right-hand side of~\eqref{E: root degree and cluster size}.

\subsection{The density of the distribution of $H$}
$H$ is a random finite rooted tree with age function $a: V(H) \to [0,\infty)$. Since the intensity measures $\lambda_s$ are non-atomic, the types in a tree are almost surely distinct. The density for $H$ to be any particular age-decorated rooted finite tree $(T,\rho)$ labelled with vertex ages $a$ is 

  \begin{multline}\label{E: density of vertex ages}
   \prod_{e=(v,w) \in E(T)} (a(v) \wedge a(w))\;
 \prod_{v \in V(T)} e^{-\lambda_{a(v)}([0,\infty))} d\pi(a(v))    \\ =      
 \prod_{(v,w) \in E(T)} (a(v) \wedge a(w))
 \prod_{v \in V(T)} \frac{e^{-a(v)}}{2} \,d(a(v)) \,.
\end{multline}

Including the edge age variables as well we get the joint density for all the age labels of a given rooted labelled tree $(T,\rho)$:
\begin{equation}\label{E: joint density of vertex and edge ages}   \prod_{e = (v,w) \in E(T)} d a(e) \mathbf{1}(a(e) \in [0,a(v) \wedge a(w)])\;
 \prod_{v \in V(T)} \frac{e^{-a(v)}}{2} \,d(a(v)) \,.\end{equation}  
  
  Notice that this simple expression does not depend on the choice of $\rho$, which is to say that the distribution of $H$ is re-root invariant. 
  
Note: here we are considering as isomorphic any two trees that are related by a graph isomorphism that preserves the root and preserves ages.

\subsection{The marginal distribution of the set of vertex ages in $H$}   
\begin{lemma}\label{L: joint age distribution}
The density for the sorted sequence of ages of the vertices in $H$ on the chamber $A_k : = \{(a_1, \dots, a_k) \in \mathbb{R}^k \,|\, 0 < a_1 < a_2 < \dots < a_k\}$ is
\begin{equation}\label{E: age set density} 2^{-k}\left(\prod_{j=1}^{k-1}\left((k-j+1)a_{j} + \sum_{m=1}^{j-1} a_m\right)\right)\prod_{i=1}^k e^{-a_i}\,da_i\,.\end{equation}
The sum of the ages is a $\Gamma(2k-1,1)$ random variable. Conditioned on the vertex ages, $H$ is a random weighted spanning tree of the complete graph with edge weights $w_{ij} = a_i \wedge a_j$.
\end{lemma}
\begin{proof}
 There are $n$ ways to assign the root to one of the ages, so the density with respect to the reference measure $2^{-k} \prod_{i=1}^k e^{-a_i} da_i$ on $A_k$ is $n$ times the sum over spanning trees $T$ of the complete graph on the set $\{1, \dots, k\}$ of the weight $\prod_{(i,j) \in E(T)} a_{i \wedge j}$. We can evaluate this sum using Kirchhoff's matrix-tree theorem: up to sign it reduces to the determinant of the $(k-1)$ by $(k-1)$ matrix $A$ defined by
 $$ A_{ij} = \begin{cases} a_{i \wedge j}, & \text{ if $i \neq j$,} \\ (i-k)a_i - \sum_{m=1}^{i-1} a_i & \text{ if $i =j$.}\end{cases}$$
 If we add columns $2, \dots, k-1$ of $A$ to column $1$ and then add column $1$ onto each of columns $2, \dots, k-1$, we obtain a lower triangular matrix with the same determinant, whose $j^{th}$ entry on the diagonal is $-a_1$ if $j=1$, and $(j-k-1)a_i - \sum_{m=1}^{j-1} a_i$ if $j>1$. The sum of weights of spanning trees is the absolute value of the product of the diagonal entries, yielding \eqref{E: age set density}.

Set $e_1 = k a_1$ and $e_i = (k+1-i)(a_i - a_{i-1})$ for $i > 1$. 
In terms of the variables $e_i$, the density \eqref{E: age set density} is given on $[0,\infty)^k$ by
\[ \frac{1}{2^k k!} \prod_{j=1}^{k-1} \left(\sum_{m=1}^j e_m \right) \,\prod_{i=1}^{k} e^{-e_i} \,de_i\,.\]
 Make the further change of variable $u_i = \sum_{m=1}^i e_m$, for $i=1, \dots, k$, to express this density as
 \[ \frac{e^{-u_k}\,du_k}{2^k k!} \,\prod_{i=1}^{k-1} u_i \,du_i\, \mathbf{1}(0 < u_1 < u_2 < \dots < u_k).\]  
It is now easy to integrate out the variables $u_1, \dots, u_{k-1}$, getting the marginal density for $u_k$:
\[\frac{1}{2^{2k-1} k!(k-1)!} u_k^{2(k-1)} e^{-u_k} \,du_k\,.\]  
 This is $w_k$ times the density of a $\Gamma(2k-1,1)$ random variable; this gives a different proof of the distribution of the total progeny $|H|$. Given $u_k$, the random variables $u_1, \dots, u_{k-1}$ are the sorted sequence of $k-1$ independent size-biased $U[0,u_k]$ random variables.
 Note that $$u_k = \sum_{m=1}^k e_m  = ka_k + \sum_{m=2}^k (k+1-m)(a_m-a_{m-1}) = \sum_{m=1}^k a_m\,,$$ which is the sum of ages of vertices. So we have shown that conditional on $|H| = k$, the sum of vertex ages is a $\Gamma(2k-1,1)$ random variable. 
 \end{proof}

\begin{remark} The proof showed how to sample the vertex ages efficiently conditioned on $|H| = k$. Since there are efficient algorithms for sampling random weighted spanning trees, such as Wilson's algorithm, one can efficiently sample $H$ conditioned on its size. However, we will see that $H$ and $\mathcal{C}$ are identically distributed, so it would be more efficient still to sample a critical binary Galton-Watson tree $G$ conditioned to have $k$ leaves using R\'emy's algorithm, label its vertices with independent $\textup{Exp}(1)$ random spent times, sample a steady-state cluster $C$ conditioned to have $G$ as its genealogical tree, and deduce the vertex and edge ages of $C$ from the spent times of $G$.\end{remark}

\subsection{The marginal joint density of edge ages} 

\begin{lemma}
 The density for $H$ with the vertex ages forgotten to be isomorphic to a particular finite rooted tree $(T,\rho)$ with edges labelled by ages $(a(e): e \in E(T))$ is
 \begin{equation}\label{E: density of edge ages} \prod_{v \in V(T)} \frac{1}{2}\exp(-\max\{a(e): \text{ $e$ incident on $v$}\})\,.\end{equation}
\end{lemma}
\begin{proof}
Integrate out the vertex ages from \eqref{E: joint density of vertex and edge ages}.
\end{proof}


\subsection{$H$ and $\mathcal{C}$ are identically distributed}
\begin{theorem}\label{T: H equals C}
Let $T$ be a finite unrooted tree with a marked oriented edge $\overrightarrow{e_0} = (v_1,v_2)$ (oriented from $v_1$ to $v_2$).   
Let $T_1$ and $T_2$ be the two trees obtained from $T$ by cutting the edge $\overrightarrow{e_0}$, rooted at $v_1$ and $v_2$ respectively. Let $p_1 =  \mathbb{P}(H \cong (T_1,v_1))$ be the probability that $H$ is isomorphic to $(T_1, v_1)$ as a rooted tree, and similarly $p_2 = \mathbb{P}(H \cong (T_2,v_2))$. Orient the youngest edge $e$ in $H$ uniformly at random. Then the probability that $(H, \overrightarrow{e})$ is isomorphic to $(T,\overrightarrow{e_0})$, as unrooted trees with a marked oriented edge, is $p_1p_2$. 

Hence $H$ satisfies the same RDE as $\mathcal{C}$ and consequently $H$ and $\mathcal{C}$ have the same law as rooted trees. In fact they have the same law as rooted trees with vertices and edges labelled by ages.
\end{theorem}
\begin{proof}
 Take two independent copies $H_1$ and $H_2$ of $H$, with roots $\rho_1$ and $\rho_2$. Let $Y$ be an $\Exp(1)$ random variable independent of $H_1$ and $H_2$ and set $X = Y/(|H_1| + |H_2|)$. Make an age-labelled tree $\tilde{H}$ by adding $X$ to all the vertex and edge labels in both $H_1$ and $H_2$ and joining the roots $\rho_1$ and $\rho_2$ by an edge labelled with age $X$. Choose a root $\tilde{\rho}$ uniformly at random from the vertices of $\tilde{H}$. This produces an age-labelled rooted tree with the youngest edge oriented (from $\rho_1$ to $\rho_2$). Conditional on $(\tilde{H},\tilde{\rho})$, the orientation of the youngest edge is uniform, by the symmetry of the construction. We claim $(\tilde{H},\tilde{\rho})$ is distributed as $(H,\rho)$ conditioned on $|H| \ge 2$. 
Because of re-root invariance of $|H|$ conditioned on its size, this reduces
to checking that their densities agree as unrooted age-labelled trees with oriented youngest edge. Let us say that an age labelling $a: V(T) \cup E(T) \to [0,\infty)$ is \emph{legal} if all the ages are in $[0,\infty)$ and for every vertex the age of each of its incident edges is less than its own age.
Then the density for either $(H,\overrightarrow{e})$ or $(\tilde{H}, (\rho_1,\rho_2))$ to be isomorphic to $(T,(v_1,v_2))$ with age labels $a: V(T) \cup E(T) \to [0,\infty)$, where $|V(T)| = k$, is
  \begin{multline*} \frac{k}{2^k}\prod_{e \in E(T)} \!da(e). \prod_{v \in V(T)} \!e^{-a(v) + a(v_1,v_2)}\,da(v)\; \mathbf{1}(\text{$a$ is legal}) \\ =\,\frac{k e^{-k e(v_1,v_2)}}{2^k}\prod_{i=1}^2\prod_{e \in E(T_i)}\!\! da(e) \!\!\!\prod_{v \in V(T_i)} \!\!e^{-a(v)}\,da(v) \,\mathbf{1}((a - a(v_1,v_2))|_{T_i} \text{ is legal})\,.
    \end{multline*}
In the first expression the factor $k$ comes from the choice of $k$ roots for $H$, all equally likely. In the second expression, $k\exp(-ke(v_1,v_2))$ corresponds to the density of $X$ given $|H_1| + |H_2| = k$.

We saw already that $\mathbb{P}(|H| = 1) = \tfrac{1}{2}$. Conditional on $|H| \ge 2$, there is almost surely a unique youngest edge in $H$, and the calculation above shows that uniformly orienting and then cutting this edge yields two independent copies of $H$, age-shifted by an exponential random variable with mean $1/|H|$. It follows that $H$ satisfies the RDE that characterizes $\mathcal{C}$. Therefore $H$ and $\mathcal{C}$ have the same law as rooted trees.

 In section \ref{SS: genealogical tree} we equipped the vertices of the genealogical tree $\mathcal{G}$ of $\mathcal{C}$ with spent time labels, allowing us to define the age of each vertex and each edge of $\mathcal{C}$ in Definition~\ref{D: age labels in C}. Using the first part of the theorem, we can now couple $H$ and $(\mathcal{G},\mathcal{C})$ so that that the vertex ages in $H$ correspond to the vertex ages in $\mathcal{C}$. To obtain a sample of the genealogical tree starting from $H$ proceed as follows. If $H$ is a singleton with age $x$ then so is $\mathcal{G}$, and we are done. Otherwise, we find the youngest edge of $H$, (which has age $x$, say); we orient it uniformly at random and cut it to obtain subtrees $H_1$ and $H_2$. Give the root $r$ of $\mathcal{G}$ spent time $x$ and let its left and right child subtrees be obtained from $H_1$ and $H_2$ respectively by the same procedure, recursively. Since $x$ has distribution $\textup{Exp(k)}$ given $|H| = k$, and $kx$ is independent of $H_1$ and $H_2$, we find that $\mathcal{G}$ is a critical binary Galton-Watson tree with independent spent time labels that are exponentially distributed with the correct means. Each edge of $H$ corresponds to a vertex of $\mathcal{G}$ and joins a pair of leaves of $\mathcal{G}$ chosen uniformly and independently from the subtrees of $\mathcal{G}$ above that vertex; these choices are independent for different edges of $H$. Keeping track of the age-shifting, we see that the edge and vertex ages as defined in Definition~\ref{D: age labels in C} agree with the original edge and vertex labels of $H$.  
\end{proof} 
\subsection{Dynamical version of the multitype Galton-Watson tree}
We have exhibited a measure-preserving map from pairs $(\mathcal{G},\mathcal{C})$ distributed according to $\mathcal{V}_0$ to age-labelled multitype Galton-Watson trees distributed according to the law of $H$. We will apply this map to the stationary process $(\mathcal{G}(t),\mathcal{C}(t))$ to obtain a stationary process $\mathcal{H}_t$ whose invariant measure is the law of $H$.
 
 We first consider the non-stationary process $\mathcal{H}_1(t)$  for $t \in [0, t_\infty)$, obtained by applying the bijection to the process $(\mathcal{G}_1(t), \mathcal{C}_1(t))$. The process $\mathcal{H}_1(t)$ takes values in the space of finite rooted trees with legal age labellings. It is increasing as a function of $t$  both in the sense of inclusion and in the sense that the labels all increase at rate 1. $\mathcal{H}_1(0)$ consists of a single vertex labelled with age $0$. The process $\mathcal{H}_1(t)$ evolves as a continuous-time Markov process.  Let $H_1, H_2, \dots$ be the i.~i.~d.~samples of the multitype Galton-Watson tree $H$ obtained by applying the bijection to the pairs $(G_i,C_i)$ in the construction of section~\ref{SS: genealogical tree}. Recall that $\gamma_1, \gamma_2, \dots$ are independent exponential random variables with mean $1$, and that they define a sequence of random times by $t_0 = 0$, and for $i \ge 1$,  
\[t_i \,=\, t_{i-1} + \frac{\gamma_i}{|\mathcal{C}_1(t_{i-1})|} \,=\, t_{i-1} + \frac{\gamma_i}{\sum_{j=1}^{i-1} \left|H_i\right|}\,.\]
 At time $t_i$, $\mathcal{H}_1(t)$ jumps by the addition of an edge that joins a uniform random vertex $v_i$ of $\mathcal{H}_1(t_i^-)$ to the root vertex of $H_i$. The vertex $v_i$ is chosen independently of the sequences $H_\cdot$, $\xi_\cdot$ and $v_1, \dots, v_{i-1}$. Recall that $t_\infty := \lim_{i\to\infty} t_i$. 
 
 The stationary process $\mathcal{H}(t)$ is now obtained by concatenating instances of $\mathcal{H}_1$, with the instance containing time $0$ size-biased by its lifetime, exactly as we did to create $\mathcal{C}(t)$ from $\mathcal{C}_1(t)$. This means that the age of the root of $\mathcal{H}(t)$ follows the steady state age distribution $\pi$, with density $(1/2)\sech^2(x/2)$ on $[0,\infty)$. 
 
 Note that the entire description of the process $\mathcal{H}(t)$ could be given in terms of the law of $H$ alone, without reference to $(\mathcal{G}, \mathcal{C})$. But it would not then be straightforward to show that the invariant measure of $\mathcal{H}(t)$ is the law of $H$. 
\begin{lemma}\label{L: H conditioned on root age} For each $t \ge 0$, the law of $\mathcal{H}_1(t)$ conditioned on $t_\infty > t$ is that of the multitype Galton-Watson tree $H$ conditioned on $a(\rho) = t$. \end{lemma}
\begin{proof}
 Recall that $\mathbb{P}(t_\infty > t) = \sech^2(\tfrac{t}{2})$ and $\mathbb{P}(\theta_1 > t) = 1- \tanh(\tfrac{t}{2})$.
  The law of $H$ conditioned on $a(\rho) = t$ is the law of $\mathcal{H}(0)$ conditioned on $\theta_{0} = -t$. The joint density of $(\theta_0,\theta_1)$ is $$\tfrac{1}{2} \tanh\left(\tfrac{\theta_1-\theta_0}{2}\right)\sech^2\left(\tfrac{\theta_1-\theta_0}{2}\right) \mathbf{1}(\theta_0 < 0 < \theta_1)\,.$$ Indeed, given $\theta_1-\theta_0 = \ell$, $\theta_0$ is uniformly distributed on $[-\ell, 0]$, where it has density $1/\ell$; this exactly cancels the size-biasing of the inter-explosion interval containing $0$.
The law of $\mathcal{H}(0)$ conditioned on $\theta_0 = -t$ and $\theta_1 = s$ is the law of $\mathcal{H}_1(t)$ conditioned on $t_\infty = s+t$. Now integrate over $s$.
\end{proof}

\subsection{The steady state cluster conditioned on its explosion time}\label{SS: GW tree given explosion time}
\begin{definition}For $x \ge 0$, let $H^{(x)}$ be the following multitype Galton-Watson tree with types (ages) in $[0,\infty)$. The root $\rho^{(x)}$ has a random age $a\left(\rho^{(x)}\right)$, where
\[\mathbb{P}\left(a\left(\rho^{(x)}\right) > y\right) = \frac{1 - \tanh\left(\tfrac{x+y}{2}\right)}{1-\tanh\left(\tfrac{x}{2}\right)}\,.\]   The offspring of any vertex of age $b$ are described by a PRM whose intensity is the measure $\lambda_b^{(x)}$ defined by the density $(b \wedge a)\, \frac{1}{2}\sech^2\!\left(\frac{a+x}{2}\right)\,da$.  The edges to the offspring vertices are labelled by random ages that are independent conditional on the offspring ages, where the edge to a vertex of age $a$ has the  uniform distribution over the interval $[0,b \wedge a]$. 
\end{definition}
In particular $H^{(0)}$ has the law of $H$. Note that $a\left(\rho^{(x)}\right)$ is stochastically decreasing in $x$. Moreover, the intensity of the offspring measure of a vertex of age $b$ in $H^{(x)}$  is increasing in $b$ and decreasing in $x$. Thus the offspring of the root $\rho^{(x)}$ is stochastically decreasing in $x$: for $x < y$ it is possible to couple $H^{(x)}$ and $H^{(y)}$ so that the offspring of $\rho^{(y)}$ form a subset of the offspring of $\rho^{(x)}$ with the same ages. It is therefore possible to couple $H^{(x)}$ and $H^{(y)}$ so that $H^{(y)}$ is isomorphic as a rooted tree to a subtree of $H^{(x)}$. If we condition $H^{(x)}$ and $H^{(y)}$ both to have root age $s$, then they may be coupled so that $H^{(y)}$ is isomorphic to a subtree of $H^{(x)}$ as rooted trees with identical age labels on all common edges and vertices.

\begin{lemma}\label{L: H conditioned on surviving for time t}
For $0 \le s \le t$, the law of $\mathcal{H}(s)$ conditioned on $\theta_1 > t$ is the law of $H^{(t-s)}$ conditioned to have root age $a(\rho) > s$. In particular, the law of  $\mathcal{H}(0)$ conditioned on $\theta_1 > t$ is the law of $H^{(t)}$.
\end{lemma}
\begin{proof}
By stationarity, sampling $\mathcal{H}(t)$ conditioned on $\theta_1 > t$ is the same as sampling $H$ conditioned on $a(\rho) > t$. We can then run the clock backwards from time $t$ to time $s$. This means that we delete every edge whose age at time $t$ is less than $t-s$, and keep only the connected component of the root. In particular we delete all vertices whose age at time $t-s$ is less than $t$, but possibly some older vertices as well. Consider a vertex that survives this pruning procedure and has age $b$ in $\mathcal{H}(s)$. It is a vertex of age $b+t-s$ in  $\mathcal{H}(t)$. Its offspring (vertex, edge) pairs in $\mathcal{H}(t)$ are described by a PRM of intensity $$\mathbf{1}(0 < a(e) < (b+t-s)\wedge a(v))\,da(e)\,.\,\frac{1}{2}\sech^2(\tfrac{a(v)}{2})\,da(v)\,.$$  
Keeping only the offspring whose edges have age at least $t-s$ at time $t$, and winding the clock back to time $s$, we get a PRM of intensity
$$\mathbf{1}(0 < a(e) < b\wedge a(v))\,da(e)\,.\,\frac{1}{2}\sech^2(\tfrac{a(v)+t-s}{2})\,da(v)\,.$$
The age of the root at time $s$ is $X-t+s$ where $X$ is distributed according to the steady state age distribution $\pi$ conditioned on $X > t$. This is the distribution of $a(\rho^{(t-s)})$ conditioned  to be at least $s$.
\end{proof}
One could also prove Lemma~\ref{L: H conditioned on surviving for time t} by computing the density of $\mathcal{H}^{(t)}$ explicitly, similarly to equation~\eqref{E: density of vertex ages}, and observing that it is the tilt of the density of $H$ by the Bayes factor $$(\sech^2(\tfrac{t}{2}))^{|V(T)|}/(1-\tanh(\tfrac{t}{2}))\, = \, \frac{\mathbb{P}(\theta_1 > t | \mathcal{H}(0) = T)}{\mathbb{P}(\theta_1 > t)}\,.$$

The next lemma describes the dynamics of the process $\mathcal{H}_1(\cdot)$ conditioned on the event $t_\infty > t$.
\begin{lemma}\label{L: unpruning}
 For $0 < s < t$ the law of $\mathcal{H}_1(s)$ conditioned on $t_\infty > t$ is the law of $\mathcal{H}(0)$ conditioned on $\theta_0 = -s$ and $\theta_1 > t-s$, which is the law of $H^{(t-s)}$ conditioned to have root age $s$.  Keeping $t$ fixed and letting $s$ increase, conditional on $t_\infty > t$ and on $\mathcal{H}_1(s)$, new offspring arrive at each vertex of $\mathcal{H}_1(s)$ independently as a Poisson rain of intensity $\frac{1}{2}\,\sech^2\!\left(\frac{a+t-s}{2}\right)\,da$ on the type space $[0,\infty)$. Each new offspring vertex of age $a$ comes with a subtree that has the law of the  Galton-Watson tree $H^{(t-s)}$ conditioned to have root age $a$.    
\end{lemma}
\begin{proof}
The first sentence is proved in the same way as Lemma~\ref{L: H conditioned on root age}. The rest follows from the proof of Lemma~\ref{L: H conditioned on surviving for time t} by considering how the pruned subtrees are added as we wind the clock forward again.
\end{proof}
\begin{remark}
Note that the arrivals process at each vertex depends only on $t$, and not on the structure or age-labelling of $\mathcal{H}_1(s)$, as we would expect from thinking about the process $\mathcal{C}_1(s)$ conditioned on $t_\infty > t$. \end{remark}

In Lemma~\ref{L: unpruning} we found the law of  $\mathcal{H}_1(s)$ conditioned on $t_\infty > t$. We now examine the law of $\mathcal{H}_1(s)$ conditioned on $t_\infty = t$, for $0 \le s < t$. We obtain this by taking the limit as $\epsilon \to 0$ of the law of $\mathcal{H}_1(s)$ conditioned on $t_\infty \in [t, t+\epsilon)$. 
From~\eqref{E: size at time s given explosion time} we have
\begin{multline}\label{E: tilt for conditioning on t infty}\mathbb{P}(|\mathcal{C}_1(s)| = k \,|\, t_\infty = t) = \\ \frac{k \sech^{2k} \frac{t-s}{2}\tanh\frac{t-s}{2} \sech^2\frac{s}{2}}{\sech^2\frac{t}{2}\tanh\frac{t}{2}}\;\mathbb{P}(|\mathcal{C}_1(s)| = k \,|\, t_\infty > s)\,.\end{multline}

Conditioning on $t_\infty = t$ tilts the density of $\mathcal{C}_1(s)$ by a factor that depends only on $|\mathcal{C}_1(s)|
$. We saw in Lemma~\ref{L: H conditioned on root age} that the distribution of $\mathcal{C}_1(s)$ given $t_\infty > s$ is the law of $H$ conditioned to have $a(\rho) = s$. 

We compute
\[\left|\lambda_{a}^{(x)}\right| = a - 2 \log \cosh\tfrac{a+x}{2} + 2\log \cosh \tfrac{x}{2}\,.\]
The marginal density for $H^{(t-s)}$ conditioned to have root age $s$ to be a given tree $(T,\rho)$ labelled with vertex ages $a: V(T) \to [0,\infty)$ with $a(\rho) = s$ is
\begin{multline}\label{E: density of H(t-s) with root age s} \prod_{(v,w)\in E(T)} \!\!(a(v) \wedge a(w))\, \prod_{v \in V(T)} e^{-\left|\lambda_{a(v)}^{(t-s)}\right|} \,\prod_{v \in V(T) \setminus\{\rho\}} \frac{1}{2}\sech^2\!\left(\tfrac{a(v) + t-s}{2}\right)\,da(v) \\
 =  \, \frac{2 \cosh^2\!\left(\tfrac{t}{2}\right)}{\cosh^{2|T|}\left(\tfrac{t-s}{2}\right)}\,.\,\prod_{(v,w) \in E(T)} \!\!(a(v) \wedge a(w))\,\prod_{v \in V(T)} \frac{e^{-a(v)}}{2}\,\prod_{v \in V(T) \setminus\{\rho\}} da(v)\,.
\end{multline} 
This is the tilt of the density of $H$ conditioned on $a(\rho) = s$ by the factor $ \sech^{2|T|}(\tfrac{t-s}{2}) \cosh^2( \tfrac{t}{2}) \sech^2(\tfrac{s}{2})$, which depends on $(T,\rho,a)$ only through the factor $\sech^{2|T|}\left(\tfrac{t-s}{2}\right)$.  Comparing~\eqref{E: tilt for conditioning on t infty} and ~\eqref{E: density of H(t-s) with root age s} we obtain
\begin{lemma}\label{L: conditioning on explosion time is size-biasing}
For $0 \le s < t$, the law of $\mathcal{H}_1(s)$ conditioned on $t_\infty = t$ is obtained by size-biasing the law of $H^{(t-s)}$ conditioned to have root age $s$. The size-biasing here means that we tilt the law by the total progeny $|H^{(t-s)}|$.
The normalizing factor for this tilt is given by \[\mathbb{E}\left(\left|H^{(t-s)}\right| | \,a(\rho) = s\right) = \frac{\tanh\tfrac{t}{2}}{\tanh\tfrac{t-s}{2}}\,.\]
\end{lemma}
\begin{lemma}\label{L: H conditioned on explosion time}
 For $0 \le s < t$ the law of $\mathcal{H}(s)$ conditioned on $\theta_1 = t$ is obtained by size-biasing the law of $H^{(s)}$ conditioned on $a(\rho) > s$, i.e. tilting that distribution by the total progeny.
\end{lemma}
\begin{proof}
The Bayesian calculations above show that for $0 \le s < t$ we have
$$\frac{\mathbb{P}(|\mathcal{C}(s)| = k \,|\,\theta_1 = t) }{\mathbb{P}(|\mathcal{C}(s)| = k \,|\,\theta_1 > t)} \,=\, k \tanh(\tfrac{t-s}{2})(1 + e^{-t}).$$
Thus the law of $\mathcal{H}(0)$ conditioned on $\theta_1 = t$ is obtained by size-biasing the law of $\mathcal{H}(0)$ conditioned on $\theta_1 > t$, which was shown in Lemma~\ref{L: H conditioned on surviving for time t} to be the law of $H^{(t)}$ conditioned on $a(\rho) > s$.
\end{proof}
\subsection{Spinal representation of the size-biased $H^{(x)}$}\label{SS: spinal representation} We define another multitype Galton-Watson tree $\hat{H}^{(x)}$. The root $\hat\rho^{(x)}$ has age $a(\hat{\rho}^{(x)})$ where $$\mathbb{P}( a(\hat{\rho}^{(x)}) > y ) = \frac{1 - \tanh^2(\tfrac{x+y}{2})}{1 - \tanh^2(\tfrac{x}{2})}\,=\, \frac{\cosh^2(\tfrac{x}{2})}{\cosh^2(\tfrac{x+y}{2})}\,.$$
The density of $a(\hat{\rho}^{(x)})$ is $\sech^2(\tfrac{a+x}{2})\tanh(\tfrac{a+x}{2}) \cosh^2(\tfrac{x}{2}) \,da $ on $[0,\infty)$.
The vertices of $\hat{H}^{(x)}$ fall into two classes: \emph{spinal} and \emph{non-spinal}. The root $\hat{\rho}^{(x)}$ is a spinal vertex. The offspring measure of a non-spinal vertex of age $b$ is a PRM of intensity $\lambda_b^{(x)}$ defined by the density $(b \wedge a) \tfrac{1}{2} \sech^2(\tfrac{a+x}{2})\,da$, (just as for the tree $H^{(x)}$). All offspring of a non-spinal vertex are non-spinal. For a spinal vertex of age $b$, the offspring is the union of a set of non-spinal vertices given by a PRM with intensity measure $\lambda_b^{(x)}$ with either zero or one spinal vertices. The spinal offspring are independent of the non-spinal offspring. There is no spinal child with probability $\tanh(\tfrac{x}{2})/\tanh(\tfrac{b+x}{2})$.
If there is a spinal child, its age is distributed according to the measure $\hat{\lambda}_b^{(x)}$ which is the measure $\lambda_b^{(x)}$ tilted by the factor $\tanh(\tfrac{a+x}{2})$. Explicitly, $\hat\lambda_b^{(x)}$ is the probability measure with density
$$ \frac{(b \wedge a) \sech^2(\tfrac{a+x}{2})\tanh(\tfrac{a+x}{2})}{2 (\tanh(\tfrac{b+x}{2}) - \tanh(\tfrac{x}{2}))} \,da.$$  
Thus the density for a spinal vertex of age $b$ to have a spinal child of age $a$ is
$$ \frac{(b \wedge a) \sech^2(\tfrac{a+x}{2})\tanh(\tfrac{a+x}{2})}{2 \tanh(\tfrac{b+x}{2})} \,da.$$
\begin{lemma}\label{L: size-biasing calculation}
Let $x > 0$. After forgetting the spine (but keeping the root), the tree $\hat{H}^{(x)}$ has the law of $H^{(x)}$ tilted in proportion to the total progeny, as in Lemmas~\ref{L: conditioning on explosion time is size-biasing} and~\ref{L: H conditioned on explosion time}. The same holds when we condition both distributions on the root age. Moreover, conditioned on $\hat{H}^{(x)}$ being isomorphic to $(T,\rho,a)$, the spine is the shortest path from the root to a uniform random vertex of $T$. Conditioned on $\hat{H}^{(x)}$ being isomorphic as an age-labelled unrooted tree to $(T,a)$, the spine is the oriented path between two independent uniform random vertices of $T$.
\end{lemma}
\begin{proof}
 For a spinal vertex of age $b$, the density for the offspring to consist of $k \ge 0$ non-spinal vertices with ages $a_1 <  \dots < a_k$ and no spinal vertex is $$\frac{\tanh(\tfrac{x}{2})}{\tanh(\tfrac{b+x}{2})}\,e^{-|\lambda_b^{(x)}|}\, \prod_{i=1}^k (b \wedge a_i).\tfrac{1}{2}\sech^2(\tfrac{a_i+x}{2})\,.$$
 For a spinal vertex of age $b$, the density for the offspring to consist of $k \ge 0$ non-spinal vertices with ages $a_1 <  \dots < a_k$ and one spinal vertex of age $b'$ is $$(b \wedge b') \tfrac{1}{2}\sech^2(\tfrac{b'+x}{2})\frac{\tanh(\tfrac{b'+x}{2})}{\tanh(\tfrac{b+x}{2})}\,e^{-|\lambda_b^{(x)}|}\, \prod_{i=1}^k (b \wedge a_i).\tfrac{1}{2}\sech^2(\tfrac{a_i+x}{2})\,.$$
Recall that $e^{-|\lambda_b^{(x)}|} = e^{-b} \cosh^2 (\tfrac{b+x}{2}) \sech^2(\tfrac{x}{2})$.

 Let $(T,\rho,a)$ be a rooted tree with vertices labelled by distinct ages. There are $|T|$ ways to choose a spine $S \subseteq V(T)$, being the shortest path from the root to a marked vertex $v_0$. For any choice of spine $S$, the density for $\hat{H}^{(x)}$ to be $(T,\rho,a)$ with spine $S$ is 
 $$ 2 \sech^{2|T|-2}(\tfrac{x}{2})\,\tanh(\tfrac{x}{2})\,\prod_{v \in V(T)} \frac{e^{-a(v)}}{2} \prod_{(v,w) \in E(T)} (a(v) \wedge a(w))\,.$$
 Note that this density does not involve the choice of root or the choice of $v_0$.
 In computing this density, for each vertex $v \in S$ there were cancelling factors of $\tanh(\tfrac{a(v)+x}{2})$ in the numerator and denominator. For comparison, the density for $H^{(x)}$ to be $(T,\rho,a)$ is 
 $$ \frac{\sech^{2|T|}(\tfrac{x}{2})}{1-\tanh(\tfrac{x}{2})}\,\prod_{v \in V(T)} \frac{e^{-a(v)}}{2} \prod_{(v,w) \in E(T)} (a(v) \wedge a(w)) \,.$$
 Since there are $|T|$ choices of spine, after forgetting the spine the density of $\hat{H}^{(x)}$ is the tilt of the density of $H^{(x)}$ by the factor
 $ |T|\,.2\sinh(\tfrac{x}{2})e^{-x/2}$. In light of Lemma~\ref{L: H conditioned on surviving for time t} and Lemma~\ref{L: H conditioned on explosion time}, this agrees with equation~\eqref{E: expected stationary size given survival to time t} which tells us that $1/\mathbb{E}(|\mathcal{H}(0)| \,|\, \theta_1 > x) = 2\sinh(\tfrac{x}{2})e^{-x/2}$.
 
 Repeating this calculation with the root conditioned to have age $s$ we find that $\hat{H}^{(x)}$ conditioned to have root age $s$ has the size-biased distribution of $H^{(x)}$ conditioned to have root age $s$. 
\end{proof}
The offspring measure of a spinal vertex of age $b$ in $\hat{H}^{(x)}$ is the tilt of the PRM with intensity $\lambda_b^{(x)}$ by the factor $\left(\tanh\tfrac{x}{2} + \sum_{i=1}^k \tanh(\tfrac{a_i+x}{2})\right)$, where $a_1 < \dots < a_k$ are all the offspring ages. The vertex with age $a_i$ is the spinal child with probability $\tanh(\tfrac{a_i + x}{2})/(\tanh\tfrac{x}{2} + \sum_{i=1}^k \tanh(\tfrac{a_i+x}{2}))$. There is no spinal child with probability $\tanh\tfrac{x}{2}/(\tanh\tfrac{x}{2} + \sum_{i=1}^k \tanh(\tfrac{a_i+x}{2}))$.


\subsection{Dynamics of the steady state cluster conditioned on its explosion time}\label{SS: Dynamics conditioned on explosion time}
In Lemma~\ref{L: Doob transform} we described the dynamics of the size processes $|\mathcal{C}(\cdot)|$ and $|\mathcal{C}_1(\cdot)|$ conditioned on $\theta_1$ and on $t_\infty$ respectively. 
In \S\ref{SS: GW tree given explosion time} we described the dynamics of $\mathcal{C}_1(\cdot)$ conditioned on $t_\infty > t$ and the dynamics of $\mathcal{C}(\cdot)$ conditioned on $\theta_1 > t$. We used this to understand the distribution of the rooted trees $\mathcal{C}(s)$ and $\mathcal{C}_1(s)$ conditioned on their next explosion time as size-biased multitype Galton-Watson trees. In \S\ref{SS: spinal representation} we gave an explicit spinal representation of these laws. Using the spinal representation we can now easily describe the dynamics of $\mathcal{C}(\cdot)$ and $\mathcal{C}_1(\cdot)$ conditioned on their next explosion time being $t$. The processes have the same generator on the interval $[0,t)$; the difference is in the initial law at time $0$. In this section we will only consider $\mathcal{C}(\cdot)$ conditioned on $\theta_1 = t$.

We  construct a monotone coupling of $\hat{H}^{(x)}$ conditioned on $a(\hat{\rho}^{(x)}) >  (t-x)$ over all $x \in [0,t]$. We start by noting that the tree $\hat{H}^{(0)}$ makes sense.   The age of the root $\hat{\rho}^{(0)}$ has density $\sech^2(\tfrac{a}{2})\tanh(\tfrac{a}{2})\,da\,$, so it has the same distribution as $t_\infty$.   In $\hat{H}^{(0)}$ the law of the offspring distribution of a spinal vertex of age $b$ is simply the usual (untilted) PRM with intensity $\lambda_b$ together with an extra, independent, point (the spinal offspring), which is distributed as the tilt of $\lambda_b$ by $\tanh\left(\tfrac{a}{2}\right)$. So $\hat{H}^{(0)}$ has an infinite spine starting at the root, and  the ages of the vertices up the spine form a Markov chain. We will study this Markov chain in~\S\ref{SS: transfer operator}, where we will see that its invariant measure has cumulative distribution function $\tanh^3\!\left(\tfrac{a}{2}\right)$. This means that the root age is out of equilibrium for the spinal Markov chain. Attached to the spine are independent multitype Galton-Watson trees, one rooted at each vertex of the spine. The tree rooted at spinal vertex $v$ has the law of $H$ conditioned on its root age being $a(v)$. Almost surely, all of these countably many attached trees are finite, so $\hat{H}^{(0)}$ is an infinite tree with one end. Each edge in the tree can be assigned a random age. As usual, we can assign edge ages independently, so that for an edge $e$ joining vertices $v$ and $w$ the age $a(e)$ is uniformly distributed in $[0, a(v) \wedge a(w)]$.

Let $T_t$ be a random infinite tree with root $\rho_t$, age labelling $a_t$ and spine $S_t$, distributed according to the law of $\hat{H}^{(0)}$ conditioned on $a_t(\rho_t) > t$. That is, the density of the root age is $$\cosh^2(\tfrac{t}{2})\sech^2(\tfrac{a}{2})\tanh(\tfrac{a}{2}) \mathbf{1}(a > t)\,da\,.$$  
We now wind the clock backwards from time $t$. As we rewind the clock we decrease all ages at rate $1$ and any edge whose age reaches $0$ is deleted. We keep only the connected component of the root, and call this tree $T_s$. We denote by $S_s$ the part of the spine $S_t$ that is contained in the subtree $T_s$, and by $a_s$ the age-labelling of the tree at time $s$. Thus we get an increasing family of rooted trees $(T_s,\rho,a_s, S_s)$, each decorated with a spine and with vertex and edge ages.
\begin{lemma}
 For each $s \in [0,t]$ the law of $(T_s,\rho,a_s, S_s)$ is the law of $\hat{H}^{(t-s)}$.
\end{lemma}
Now we can run the clock forwards from time $0$ and observe that $(T_s,\rho,a_s, S_s)$ is a Markov process with respect to the filtration that it generates. At time $0$ it is a sample of $\hat{H}^{(t)}$, consisting of a rooted tree $(T, \rho)$ equipped with age labels on the vertices and edges and a spine from the root to a marked vertex $v_0(0)$. Its dynamics are as follows:
\begin{itemize}
\item The age of each existing vertex and each existing edge increases at rate $1$.
\item For each vertex $v$ independently, at rate $1 - \tanh(\tfrac{t-s}{2})$ a new edge $e$ arrives joining $v$ to the root $w$ of an
 independent sample of $\mathcal{H}^{(t-s)}$.
  The age of $e$ is $0$. The new edge does not form part of the spine.
\item In addition, at the vertex $v_0(s)$ that is at the top of the spine at time $s$ there is an extra arrivals process, independent of the previous one: at rate $\cosech(\tfrac{t-s}{2})$ a new edge $e$ arrives joining $v_0(s)$ to the root of an independent sample of $\hat{H}^{(x)}$. The spine after the arrival  is the concatenation of the previous spine, $e$ and the spine of the new subtree.
\end{itemize}

A beautiful property of this Markov process is that if we watch the process $(T_\cdot, \rho, a_\cdot)$ from time $0$ up to some time $s<t$ but we are not shown the spine, then we cannot tell where the spine is, in the following sense:  conditioned on $(T_\cdot, \rho, a_\cdot)_{[0,s]}$, the spine at time $s$ is the path from the root to a \emph{uniform} random vertex of $T_s$.  Only once we reach time $t$, when the tree almost surely becomes infinite, are we certain where the spine is.

\begin{theorem}
 The law of $\mathcal{C}(\cdot)$ on $[0,t)$ conditioned on $\theta_1 = t$ is the law of the process $(T_\cdot, \rho, a_\cdot, S_\cdot)$ with the spine forgotten, when it is started at time $0$ as a sample of $\hat{H}^{(t)}$. Similarly the law of
  $\mathcal{C}_1(\cdot)$ on $[0,t)$ conditioned on $t_\infty = t$ is the law of the process $(T_\cdot, \rho, a_\cdot, S_\cdot)$ with the spine forgotten, when it is started at time $0$ as a singleton of age $0$.
   In particular, $\mathcal{C}_1(t_{\infty})$
    has the law of $\hat{H}^{(0)}$. 
\end{theorem}
We now see that the time-reversal of the steady-state cluster growth process can be described entirely in terms of a bi-infinite sequence of independent copies of the random infinite tree $\hat{H}^{(0)}$. Each copy contains the information necessary to describe one inter-explosion period of $\mathcal{C}(\cdot)$, whose lifetime is given by the age of the root. As usual, the inter-explosion period containing time $0$ must be size-biased. If we are given a sample of $\hat{H}^{(0)}$ \emph{with edge ages forgotten}, then running time backwards gives an interesting fragmentation or logging process in which there are immediately infinitely many fragments, all finite. In this logging process the edges are cut independently but not all at the same rate: an  edge whose end vertices have ages $a$ and $a'$ is cut at rate $1/(a \wedge a')$. The vertex ages all decrease at rate $1$, so each edge is almost surely cut strictly before either of its endpoints reaches age $0$ and disappears.

\begin{remark} The infinite rooted tree $C_1(t_\infty)$ is a \emph{unimodular} random tree. Unimodularity  is an analog of re-root invariance that applies to possibly infinite random networks. See for example Aldous and Lyons \cite{AldousLyons2007} and Benjamini, Lyons and Schramm \cite{BLS} for the definition and some important properties of unimodularity. $\hat{H}^{(0)}$ is unimodular because it is the Benjamini-Schramm limit, or local weak limit, of $H^{(x)}$ as $x \searrow 0$, and $H^{(x)}$ is re-root invariant for $x > 0$. In particular, the law of $\hat{H}^{(0)}$ is invariant under the re-rooting (and consequent change of spine) that is induced by any continuous-time random walk of the root whose jump rates are determined by (possibly random)  edge conductances: this property is called \emph{involution invariance} and is known from \cite{AldousLyons2007} to be equivalent to unimodularity. Note that although in $\hat{H}^{(x)}$ the spine joins two independent uniform random vertices, we cannot make sense of choosing two independent uniform vertices in an infinite random tree. The length of the spine of $\hat{H}^{(x)}$ diverges in probability as $x \searrow 0$.
\end{remark}

From Lemma~\ref{L: size-biasing calculation} we can deduce the Benjamini-Schramm limit of the steady-state cluster conditioned on $|\mathcal{C}| = k$, as $k \to \infty$. We conjecture that this is also the Benjamini-Schramm limit as $n \to \infty $ of the distribution of the fires seen in the long run by a tagged vertex in $\mf(n)$. Depending on the lightning rate $\lambda(n)$, the typical fires may in fact have cycles, but the length of the shortest cycle tends to infinity in probability as $n \to \infty$, so the cycles are not captured by the Benjamini-Schramm limit.
\begin{theorem}\label{T: BS limit}
 The Benjamini-Schramm (or local weak) limit of $\mathcal{C}$ conditioned on $|\mathcal{C}| = k$ as $k \to \infty$ exists and has the law of $\hat{H}^{(0)}$.
\end{theorem}
\begin{proof} Fix a radius $r \in \mathbb{N}$ and consider the ball $B(\rho,r)$ of radius $r$ about the root vertex in the graph distance. We have to show that restriction of $\mathcal{C}$ conditioned on $|\mathcal{C}(0)| = k$ to $B(\rho,r)$ converges in distribution as $k \to \infty$ to the restriction of $\hat{H}^{(0)}$ to $B(\rho,r)$. 
Note that $\hat{H}^{(x)}$ has the law of $\mathcal{C}(0)$ conditioned on $\theta_1 = x$, and that it converges in the local weak sense to $\hat{H}^{(0)}$ as $x \searrow 0$. 

To compare these two conditionings, we consider $\mathcal{C}(0)$ conditioned on \emph{both} $|\mathcal{C}(0)| = k$ and $\theta_1 = x$, where $ x \to 0$ and $k \to \infty$. (One could in fact take $x=k^{-1/2}$, so that $k$ is a typical value for $|\hat{H}^{(x)}|$, but this will not be necessary for our argument). We can perform this double conditioning in either order. 

On one hand, once we have conditioned on $|\mathcal{C}(0)| = k$, further conditioning on $\theta = x$ does not affect the distribution of $\mathcal{C}(0)$ at all, so the doubly-conditioned measure is the law of $\mathcal{C}$ conditioned on $|\mathcal{C}| =k$.

 On the other hand, we claim that conditioning the tree $\hat{H}^{(x)}$ to have exactly $k$ vertices hardly affects the distribution of the ball $B(\rho,r)$ in $\hat{H}^{(x)}$, in the strong sense of total variation distance. To make this precise, for each $r \in \mathbb{N}$ let $U_r$ be the subtree of $\hat{H}^{(x)}$ consisting of all vertices whose most recent spinal ancestor is at distance less than $r$ from $\rho$, let $s_r$ be the $r^{th}$ vertex up the spine, if it exists, let $a_r$ be the age of $s_r$ and let $V_r$ be the subtree rooted at $s_r$. Then $\hat{H}^{(x)} = U_r \cup V_r$ and $B(\rho,r) \subseteq U_r \cup \{s_r\}$.

Now fix $r$ and let $\epsilon > 0$. Choose $a_{min} > 0$, $M < \infty$ and $x_0 > 0$ such that for $x < x_0$ we have $$ \mathbb{P}(s_r \text{ exists, } a_r > a_{min}, \text{ and } |U_r| < M) > 1- \epsilon\,.$$
 We have $|\hat{H}^{(x)}| = |U_r| + |V_r|$. Conditional on the event that $s_r$ exists and  $a_r = a$,  $U_r$ and $V_r$ are independent,  and $|V_r|$ has the distribution of $|H^{(x)}|$ conditioned to have root age $a$. Further conditioning on $|U_r| + |V_r| = k$ affects the distribution of $U_r$ through a tilt that depends only on $|U_r|$ and $a_r$:
\begin{eqnarray*}\frac{\textup{density}\left(a_r = a, \;U_r = U\right) \,|\, |U_r| + |V_r| = k)}{\textup{density}(a_r = a, \;U_r = U)}  & = & \frac{\mathbb{P}(|\hat{H}^{(x)}| = k \,|\, a_r = a,\; U_r = U)}{\mathbb{P}(|\hat{H}^{(x)}| = k )}\\ & = &  \frac{\mathbb{P}(|\hat{H}^{(x)}| = k - |U| \,|\, a(\rho) = a)}{\mathbb{P}(|\hat{H}^{(x)}| = k )}\,. \end{eqnarray*}
 Since $\hat{H}^{(x)}$ has the distribution of $\mathcal{C}(0)$ conditioned on $\theta_1 = x$, we have from Corollary~\ref{C: Expectations} that
 $$\mathbb{P}(|\hat{H}^{(x)}| = k)  = 2k \sech^{2k-2}\tfrac{x}{2} \tanh \tfrac{x}{2} w_k  \sim \frac{\sech^{2k-2}\tfrac{x}{2} \tanh\tfrac{x}{2}}{\sqrt{\pi k}}\; \text{ as $k \to \infty$}\,, $$ uniformly in $x$.
Likewise equation~\eqref{E: size tail given explosion time} gives
$$\mathbb{P}(|\hat{H}^{(x)}| = k\,|\, a(\rho) = a) \sim \frac{\sech^{2k}\tfrac{x}{2} \tanh\tfrac{x}{2} \sech^2\tfrac{a}{2}\tanh\tfrac{a}{2}}{\sqrt{\pi k} \sech^2\tfrac{a+x}{2} \tanh\tfrac{a+x}{2}}\; \text{ as $k \to \infty$}\,, $$ uniformly in $x > 0$ and locally uniformly in $a > 0$.  Hence the density ratio computed above converges to $1$ as $ x \to 0$ and $k \to \infty$, locally uniformly in $a$ and $|U|$.

We have shown that for any $\epsilon > 0$, outside a set of measure at most $\epsilon$, conditioning on $| \hat{H}^{(x)}| = k$ tilts the joint distribution of $U_r$ and $a_r$ by a factor that tends uniformly to $1$ as $x \to 0$ and $k \to \infty$. Hence the total variation distance between the distributions of $B(\rho,r)$ in $\hat{H}^{(x)}$ before and after conditioning on $|\hat{H}^{(x)}| = k$ tends to $0$ as $x \to 0$ and $k \to \infty$ together. Since $\hat{H}^{(x)}$ converges to $\hat{H}^{(0)}$ in the Benjamini-Schramm sense as $x \to 0$, this implies that the distribution of $\mathcal{C}$ conditioned on $|\mathcal{C}| = k$ also converges in the Benjamini-Schramm sense to $\hat{H}^{(0)}$.
\end{proof}

\begin{remark}
Much is known about size-biased Galton-Watson trees. See Janson's comprehensive survey \cite{Janson}.  Kesten made sense of size-biasing a critical (single-type) Galton-Watson tree, where the total progeny is almost surely finite but has infinite expectation. The size-biased tree is an infinite rooted tree with one end, described by an infinite spine of vertices whose offspring have the size-biased version of the offspring distribution. See for instance Lyons and Peres \cite[Ch. 12]{LyonsPeres}. In the special case of the critical binary Galton-Watson tree, the corresponding Kesten tree consists of an infinite spine, each vertex along the spine other than the root has exactly one non-spinal child, which is identified with the root of an independent critical binary Galton-Watson tree. This is a description of $\mathcal{G}_1\left(t_\infty^{-}\right)$, the genealogical tree of $\mathcal{C}_1$ at its explosion time.

 A multi-type Kesten tree was introduced by Kurtz, Lyons, Pemantle and Peres \cite{KLPP}. Recently P\'{e}nisson \cite{Pen}, Stephenson \cite{St} and Abraham et al \cite{ADG} have studied the local weak convergence of critical multi-type Galton-Watson trees with finitely many types conditioned to have large total progeny to the corresponding multi-type Kesten tree. Theorem~\ref{T: BS limit} is similar to their results, though not a consequence since $H$ has a continuum of types. Our result is reasonably precise in that the conditioning is on $|\mathcal{H}| = k$, rather than on $|\mathcal{H}| > k$. In \cite{ADG} the authors condition on the multi-set of types, which we have not done here.
\end{remark}

\subsection{Diagonalization of the transfer operator for the multitype Galton-Watson process associated to $H$}\label{SS: transfer operator}
 In this section it will be convenient to work with a compact space of types, so we will make a change of variable, taking the \emph{type} of a vertex to be $\tanh(a/2)$ where $a$ is its age. The result of this transformation is that the space of types is $[0,1]$ and the distribution of the type of the root of $\mathcal{C}$ is the uniform distribution on $[0,1]$.

Now the $k^\text{th}$ generation of the multitype Galton-Watson process is described by a purely atomic measure on the type $[0,1]$, say $\mu_k$. In this section we study the operator $T$ that sends $\mu_k$ to the expected value of $\mu_{k+1}$ given $\mu_k$. Starting at generation $0$ with a measure $\mu_0 = \delta_x$, where $x$ is a random type with law $U([0,1])$, we have
 \[ \mathbb{E}(\mu_k \,|\, a(\rho) = 2 \arctanh(x) ) = T^k(\delta_x)\,.\]
 It is convenient to study the corresponding 
 \emph{transfer operator} $T^*$ acting on $C([0,1])$, defined by 
\[\int T^*f(s) d\mu(s) = \int f(t) d T(\mu)(t)\]
Explicitly, we have
\begin{eqnarray*} T^*f(s) & = & \mathbb{E}\left(\left.\int f(t) d\mu_1(t) \,\right|\, \mu_0 = \delta_s\right)\\
 & = & \int_0^1 2 \arctanh(s\wedge t) f(t) dt \end{eqnarray*}
 
 \begin{lemma}
 $T^*$ extends to a compact self-adjoint operator on $L^2([0,1])$, a contraction with simple spectrum $\{1/(n(2n-1)): \, n \in {1,2,3, \dots}\}$.
 \end{lemma}
  \begin{proof}
Note that $T^*$ extends to a bounded operator on $L^2([0,1])$  defined by the same formula. Indeed by Cauchy-Schwarz the operator norm is bounded by the Hilbert-Schmidt norm, which is finite:
\begin{eqnarray*} \int_0^1 |T^*f(s)|^2 ds & = & \int_0^1 \left(\int_0^1 2 \arctanh(s\wedge t) f(t) dt\right)^2 ds\\ & \le & \int_0^1 \left(\int_0^1 4 \arctanh(s \wedge t)^2 dt\right)\left(\int_0^1 |f(t)|^2 dt\right)   ds \\
 & = & \left(\frac{2\pi^2}{3} - 8 \log 2\right) \|f\|_2^2 \end{eqnarray*}

Since $T^*$ is bounded, self-adjointness follows from the symmetry of the kernel. Compactness follows from the finiteness of the Hilbert-Schmidt norm.
 
 Next we observe that the operator $T^*$ is diagonalized by the Legendre polynomials of odd degree. From self-adjointness we know that the eigenvectors are orthogonal in $L^2([0,1])$. It is simple to check that for $n \ge 1$ odd, $T^*(x^n)$ is an odd polynomial of degree $n$ with leading coefficient $1/(n(2n-1))$. This can be done by integration by parts and an induction. It follows that the eigenvectors are the odd Legendre polynomials $P_{2n-1}(x)$ for $n=1,2,3,\dots$, with eigenvalues $1/(n(2n-1))$. The odd Legendre polynomials form an orthogonal basis for $L^2([0,1])$, so we are done.
\end{proof}

Note that the (unique) invariant probability distribution $\mu^*$ of $T$ is given by the density $2x$. The invariant distribution $\mu^*$ of $T$ corresponds to the age distribution with cdf $\tanh(a/2)^2$, which is  is the distribution of $t_\infty$. For any choice of $\mu_0$, $\mathbb{E}(\mu_k)$ converges at exponential rate in the total variation norm to $\left(\frac{3}{2}\int_0^1 x \,d\mu_0(x)\right)\,\mu^*$. 
In fact this convergence is uniform over probability measures $\mu_0$, because the Legendre polyomials normalized by $P_n(1) = 1$ are uniformly bounded by 1 in the sup norm on $[-1,1]$. In this normalization we have
\[ \int_0^1 P_{2i-1}(x)\,dx = (-1)^{i+1} w_i \quad \text{and} \quad \int_0^1 P_{2i-1}(x) P_{2j-1}(x)\,dx  = \delta_{ij}/(4i-1)\,.\] Thus when $\mu_0$ is the Lebesgue measure on $[0,1]$, the distribution of the type of the root of $\mathcal{C}$, by expanding the density of $\mu_0$ in the odd Legendre polynomials we obtain the following expression for the density of $\mathbb{E}(\mu_k)$: 
\[ \frac{d \mathbb{E}(\mu_k)}{d \mu_0}(x) = \sum_{i=1}^\infty (4i-1) (-1)^{i+1} w_i \frac{P_{2i-1}(x)}{(i(2i-1))^k}\,. \]
Hence the expected number of vertices at distance exactly $k$ from the root in $\mathcal{C}$ is given by \[\mathbb{E}(|\mu_k|) = \sum_{i=1}^\infty \frac{(4i-1) w_i^2}{(i(2i-1))^k}\,,\] from which we find that $\tfrac34 <\mathbb{E}(|\mu_k|) \le \tfrac{3}{4} + \tfrac{1}{4} 6^{-k}$ for every $k \ge 0$.

The transfer operator for the spinal Markov chain of $\hat{H}^{(0)}$ is very closely related to $T^*$. In fact it is the conjugate of $T^*$ by the map $f \mapsto x.f$, meaning that it is given by $$f \mapsto x T^*(f/x)\,.$$ Hence the eigenvectors for $T^*$ are the polynomials $xP_{2n-1}(x)$, with corresponding eigenvalues $1/(n(2n-1))$ for $n=1, 2, 3, \dots$. The first of these gives us the invariant distribution for the spinal Markov chain, which in terms of the transformed type space $[0,1]$ has the density $3x^2$.

\section{Construction of an infinite forest fire model}\label{S: stationary forest fire}

Although the connection between the steady state cluster process and the R\'ath-T\'oth mean field forest fire model is currently only heuristic, it is possible to view the steady state cluster growth process as the process of the cluster of a tagged vertex in an \emph{infinite} forest fire model, which we will construct in this section. It will have the property that a cluster burns at the moment it becomes infinite; there is no need for a lightning process. The model is defined on an infinite tree. One can think of the fires as being ignited at the ideal boundary of this tree.

The R\'ath-T\'oth model is exchangeable, but that is too much to ask for in a countably infinite forest fire model. If such a model were exchangeable then the probability that any particular pair of vertices become joined by an edge during any fixed time interval must be zero, since the expected number of edges arriving at any vertex in that time interval is finite. But then since there are only countably many possible edges, almost surely no edges would arrive at all! Instead of being exchangeable, our infinite forest fire model can be thought of as a candidate for a local weak limit of $\mf(n)$ in its stationary state, as $n \to \infty$. As Aldous and Steele discuss in \cite{AldousSteele}, local weak limits of exchangeable models are involution invariant, or unimodular. We have already seen that $\hat{H}^{(0)}$ has this property, and the infinite forest fire model that we construct here will also possess it.

We will construct our model on a labelled infinite tree $\mathcal{Z}$, whose vertices are labelled by the set $\mathbb{Z}^*$ of finite strings of integers. This labelling by strings is a convenient scaffolding for the construction, which we will eventually erase in order to obtain a stationary process. If $v \in \mathbb{Z}^*$ and $n \in \mathbb{Z}$ then we denote by $v.n$ the string obtained by appending $n$ to the string $v$. The edges of $\mathcal{Z}$ are the pairs $(v,v.n)$ for $v \in \mathbb{Z}^*$ and $n \in \mathbb{Z}$. The tree $\mathcal{Z}$ is rooted at the empty string, which we denote $\emptyset$. We denote by $|v|$ the length of the string, so that $|v|$ is the height of $v$ above the root $\emptyset$. We denote the subtree of $\mathcal{Z}$ rooted at $v$ by $v.\mathcal{Z}$. This consists of all vertices whose strings have $v$ as a prefix. The induced subgraph $\mathcal{N}$ of $\mathcal{Z}$ on the vertex set $\mathbb{N}^*$ is called the Harris-Ulam-Neveu tree.

Our forest fire model will be a \cadlag process $\mathbf{FF}_t: t \in [0,\infty)$.   A state of $\mathbf{FF}_t$ consists of a spanning subgraph of $\mathcal{Z}$, (necessarily a forest), together with an assignment of an age $a(e) \in \mathbb{R}$ to every edge $e$, and an age $a(v) \in [0,\infty)$ to every vertex $v$. We let $\mathbf{Live}_t$ denote the spanning forest consisting of all the edges of $\mathbf{FF}_t$ that have non-negative age. We call these the \emph{live} edges. We let $\mathbf{Future}_t$ denote the spanning forest consisting of all the edges of $\mathbf{FF}_t$ that have negative age. We call these the \emph{future} edges. We insist that for every vertex $v$, the set of $n \in \mathbb{Z}$ such that $(v,v.n)$ belongs to $\mathbf{FF}_t$ must be an interval of $\mathbb{Z}$ of the form $\mathbb{Z}_{\ge m} = \{ n \in \mathbb{Z} \,| \, n \ge m\}$ for some $m$. Thus as a rooted plane tree $\mathbf{FF}_t$ is always isomorphic to $\mathcal{N}$. For each vertex $v$ the sequence of ages of the edges to its children must be strictly decreasing: if $m < n$ and $(v,v.m)$ is an edge of $\mathbf{FF}_t$ then $(v,v.n)$ is also an edge of $\mathbf{FF}_t$ and $a((v,v.m)) > a((v,v.n))$. For every $t \ge 0$, conditional on the subgraph $\mathbf{Live}_t$, for each vertex $v \in \mathbb{Z}^*$ the sequence of the ages of child edges $(v,v.n)$ in $\mathbf{Future}_t$ forms a Poisson point process of unit intensity on $(-\infty,0)$, and these processes are independent. At time $0$, the graph $\mathbf{Future}_t$ is defined to be $\mathcal{N}$.
The reader may recognise that $\mathbf{Future}_t$ is isomorphic to Aldous' Poisson-weighted infinite tree, where the edge weights play the r\^{o}le of negative ages. 

Next, we specify how to sample from the distribution of $\mathbf{Live}_0$. Sample independent copies $H_v$ of the multitype Galton-Watson tree $H$,  one for each of the countably many vertices $v$ of $\mathcal{Z}$. We will define inductively an embedding $\phi$ of the disjoint union of some of these trees to make a spanning forest of $\mathcal{Z}$. Suppose the components in $\mathbf{Live}_0$ of all the vertices up to height $h$ have been specified; when we begin, $h=-1$. Then for each vertex $v$ at height $h+1$ that does not belong to any of the components already specified, we extend $\phi$ by mapping the root of $H_v$ to $v$ and then for each vertex $x \in H_v$ with children $w_1, \dots, w_k$ ordered so that the edge ages in $H_v$ satisfy $a((x,w_1)) < \dots < a(x,w_k)$, we define $\phi(w_i) = \phi(x).(-i)$ for $i = 1, \dots, k$.  

\emph{Informally}, the process $\mathbf{FF}_t$ evolves as follows. The age of each edge of $\mathbf{FF}_t$ and the age of each vertex increases at rate $1$. Let $e$ be any edge of $\mathcal{Z}$. At some random time $\tau(e)$ such that $a_{\tau(e)}(e) \ge 0$, the component of $\mathbf{Live}_t$ that contains $e$ may explode. When this happens, all edges in this component are instantaneously deleted from $\mathbf{FF}_t$, and all vertices in the component have their ages reset to $0$. The vertices are not deleted. We refer to such a deletion event as a fire.

Warning: we have not yet defined a process! The previous paragraph apparently describes a \emph{deterministic} evolution - so all the randomness should be in the sampling of the initial state. However, it is not at all clear that the fire times $\tau(e)$ can all be defined as functions that are measurable with respect to the $\sigma$-algebra generated by $\mathbf{FF}_0$. If it were true that the initial state almost surely determines the fire times uniquely, then the process would be called \emph{endogenous}, and we would already have specified a stochastic process (though some work would be required to prove this). However, we are not currently able to resolve this \emph{endogeny problem}, so instead we will give a rigorous construction of a process that does satisfy the description above but is defined on a larger sigma algebra than the one generated by $\mathbf{FF}_0$. 
  
Although it is clear that any process  $\mathbf{FF}_t$ meeting the above informal description is not stationary, we will construct such a process with the property that it becomes stationary once the labelling of the vertices by $\mathbb{Z}^*$ is forgotten, and only the rooted plane graph structure and the age labelling are retained. After forgetting the labelling by strings, the process $\mathbf{Live}_t$ will be a candidate for the local weak limit of the R\'ath-T\'oth forest fire model $\mf(n)$ in its stationary state, as $n \to \infty$.

\subsection{Rigorous forward construction of a version of $\mathbf{FF}_t$}
In order to construct the process meeting the above description of $\mathbf{FF}_t$, we will use Kolmogorov's consistency theorem. For each $h \in \mathbb{N}$ we will construct a forest fire process $\mathbf{FF}^h_t$ on the truncation $\mathcal{Z}_h$ of $\mathcal{Z}$, which is the induced subgraph on strings of length at most $h$. We will show that for $h' < h$, the restriction of $\mathbf{FF}^h_t$ to $\mathcal{Z}_{h'}$ is identical in law to $\mathcal{Z}^{h'}_t$. It will follow that there exists a process $\mathbf{FF}^\infty_t$, taking values in the age-decorated spanning forests of $\mathcal{Z}$, whose restriction to $\mathcal{Z}_h$ is identical in law to $\mathbf{FF}^h_t$ for all $h$. Finally, we will check that $\mathbf{FF}^\infty_0$ has the law of $\mathbf{FF}_0$ and that the evolution of $\mathbf{FF}^\infty_t$ agrees with the informal description given above of the evolution of $\mathbf{FF}_t$.

To construct $\mathbf{FF}_t^h$, we define $\mathbf{FF}_0^h$ to be the restriction of $\mathbf{Live}_0 \cup \mathbf{Future}_0$ (as defined above) to the truncated tree $\mathcal{Z}_h$. We will denote the age labelling in $\mathbf{FF}_t^h$ by $a_t$. Thus $\mathbf{FF}_0^h$ is a spanning forest of $\mathcal{Z}_h$ whose vertices and edges are labelled with ages. Moreover, the initial states $\mathbf{FF}_0^h$ for different $h$ are coupled together, related by restriction. In addition, independently of $\mathbf{FF}_0^h$ we equip each leaf $v$ of $\mathcal{Z}_h$ with a Poisson point process $\mathcal{P}_0(v)$ on $[0,\infty)^2$ with intensity at $(t,y)$ given by $$\mathbf{1}(y < t + a_0(v))\,\tfrac{1}{2}\sech^2\tfrac{y}{2}\,dt\,dy\,. $$ Conditional on the ages of the leaves, the Poisson point processes $\mathcal{P}_0(v)$ are independent.  A point $(t,y)$ in $\mathcal{P}_0(v)$ represents the \emph{potential} ignition of a fire at the leaf $v$ at time $t$. It will also be possible for the points to be deleted before they ignite a fire. For the moment the reader may think of the point process $\mathcal{P}_0(v)$ simply as a device to encode a random ignition process at each leaf whose rate depends on the age of the leaf as part of the initial data. Later we will see that the $y$ co-ordinate plays the r\^ole of the arrival time of an edge from $v$ to a child of $v$. 

Given these initial data, the evolution of $\mathbf{FF}_t^h$ is deterministic. Each edge age $a_t(e)$ and each vertex age $a_t(v)$ increases at rate $1$. As above we partition the edges of $\mathbf{FF}_t^h$ into $\mathbf{Live}_t^h$ and $\mathbf{Future}_t^h$, according to the sign of $a_t(e)$. Thus, edges move from $\mathbf{Future}_t^h$ to $\mathbf{Live}_t^h$ when their age passes $0$. 

For each leaf $v$ of $\mathcal{Z}_h$, the first co-ordinate of each point of each Poisson point process $\mathcal{P}_t(v)$ decreases at rate $1$. When any of these points reaches the boundary so that it is a point $(0,y) \in\mathcal{P}_{t^-}(v)$, the vertex $v$ is \emph{ignited}. This means the following:
\begin{itemize}
\item $(0,y)$ is deleted so that it is not present in $\mathcal{P}_t(v)$ (recall we want to define a \cadlag process),
\item the edges in the connected component of $v$ in $\mathbf{Live}_{t^-}^h$ are deleted from $\mathbf{FF}_{t-}^h$ to obtain $\mathbf{FF}_t^h$,\item the vertex ages of all the vertices in that component (including $v$ itself) are reset to be $0$ at time $t$,
\item for each vertex $w$ at height $h$ in that component, the Poisson point process $\mathcal{P}_t(w)$ is obtained by removing all points $(t,y)$ such that $y \ge t$ from $\mathcal{P}_{t^-}(w)$.  
\end{itemize}

Now the burning times of the edges are well-defined measurable functions of the initial data. Indeed, even if we suppressed the ignitions, no infinite connected component could form in $\mathbf{Live}_t^h$, since this would require some vertex to have infinitely many live edges at some finite time, which almost surely does not occur. Therefore the event that any given edge $e$ burns before a given time $t$ almost surely depends on only finitely many of the ignition processes $\mathcal{P}_t(v)$, and hence on finitely many edge arrival times and potential ignitions. It follows that it is measurable. Almost surely each edge does eventually burn in $\mathbf{FF}^h_t$.

For the consistency argument, we have to prove that for each $h \ge 1$ the restriction of $\mathbf{FF}_t^h$ to $\mathcal{Z}_{h-1}$ is identically distributed to $\mathbf{FF}_t^{h-1}$. For each edge $e = (v,v.n)$ in $\mathbf{FF}_0^h$ where $|v| = h-1$, we define the \emph{arrival time} $\alpha(e)$ of $e$ to be $-a_0(e)$, and the \emph{a priori burning time} $\Theta(e)$ of $e = (v,v.n)$ to be 
$$\Theta(e) = \min\{t: (0,y) \in \mathcal{P}_{t^-}(v.n),\, t \ge \alpha(e) \}\,.$$ 
This is the unique time at which an ignition at the leaf $v.n$ could possibly cause a fire in which the edge $e$ is burned. However, it is also possible for the edge $e$ to be burned earlier than this time, if $v$ is burned in a fire ignited at a different child leaf of $v$ at some time in $[\alpha(e), \Theta(e))$. Either way, $e$ does not belong to $\mathbf{FF}_t^h$ for any time $t > \Theta(e)$. We define 
$\tilde{\mathcal{P}}_t(v)$ to be the point process 
$$ \tilde{\mathcal{P}}_t(v) = \{ (\Theta(e)-t,\,\Theta(e) - \alpha(e)):  \,e = (v,v.n) \in \mathbf{FF}_t^h,\, |v| = h-1\}\,.$$
The first co-ordinate of each points in $\tilde{\mathcal{P}}_0(v)$ represents the waiting time until a time at which the vertex $v$ could potentially burn due to an ignition at one of its children, given complete information about the arrival times of the edges $(v,v.n)$ and the  processes $\mathcal{P}_0(v.n)$ of potential ignitions at the children. The second co-ordinate in each point enables us to enforce the constraint on the age of $v$ at each of these potential burning times that comes from the fact that a live edge incident on $v$ could not have survived through the time at which $v$ previously burned.  Nevertheless, each of the points in $\tilde{\mathcal{P}}(v)$ still only represents a \emph{potential} burning time, because $v$ and the live edges to its children could be burned earlier by a fire transmitted to $v$ from the parent of $v$.
 
After conditioning on the ages of the vertices at height $h-1$ in $\mathbf{FF}_0^h$, (which are mutually dependent), the processes $\tilde{\mathcal{P}}_0(v): |v| = h-1$ are conditionally independent, as they are functions of disjoint sets of conditionally independent variables. We also need to compute the distribution of $\tilde{\mathcal{P}}_t(v)$ to complete the consistency argument.
 \begin{lemma}\label{L: ignition projection} $\tilde{\mathcal{P}}_0(v)$ is a Poisson point process on $[0,\infty)^2$ with intensity $$\mathbf{1}(y < a_0(v) + t)\,\tfrac{1}{2}\sech^2\tfrac{y}{2} \,dt\,dy\,.$$\end{lemma} 
\begin{proof}
 Consider an edge $e = (v,v.n)$ that belongs to $\mathbf{FF}_0^h$, where $v$ is at height $h-1$. There are two cases to consider: $e \in \mathbf{Live}_0^h$ and $e \in \mathbf{Future}_0^h$.
 
 Suppose $e \in \mathbf{Future}_0^h$. We compute the distribution of $a_{\alpha(e)}(v.n)$. Before the edge $e$ becomes live at time $\alpha(e) = -a_0(e)$ (which is positive), the burning times of $v.n$ follow a stationary renewal process with interarrival times distributed like $t_\infty$. This is because the initial age $a_0(v.n)$ is the root age of $H_{v.n}$, which has the distribution of $\theta_1$, and this is the stationary distribution of the spent time of the renewal process. The hazard function of the renewal process when the spent time is $a$ is $\tanh\tfrac{a}{2}$, which is $ \int_0^a \tfrac{y}{2}\sech^2\tfrac{y}{2}\,dy$, so the renewals are correctly described by the actual ignitions implied by $\mathcal{P}(v.n)$. Hence conditional on $\alpha(e)$, the age $a_{\alpha(e)}(v.n)$ is also distributed like $\theta_1$.
 
 Claim: conditional on all the initial data apart from $\mathcal{P}(v.n)$,  $\Theta(e)-\alpha(e)$ is distributed like $\theta_1$.  We have
 $$ \mathbb{P}(\Theta(e) > \alpha + s,\,|\, \alpha(e) = \alpha,\, a_{\alpha}(v.n) = a) = \mathbb{P}(\mathcal{P}(v,n) \cap \Delta_{\alpha,s} = \emptyset)\,,$$
  where 
 $$\Delta_{\alpha,s} = \{ (t,y) \,:\, \alpha \le t < \alpha + s,\, y < t-\alpha + a\}\,.$$ 
 This probability is 
 \begin{multline*} \exp\left( - \int_0^s \int_0^{u + a} \tfrac{1}{2}\sech^2\tfrac{y}{2}\,dy \,du\right) = \exp\left( - \int_0^s \tanh\tfrac{u+a}{2} \,du\right) \\ = \sech^2 \tfrac{a+s}{2}\,\cosh^2 \tfrac{a}{2}\,.\end{multline*}
 Here we used the fact that $a_t(v.n) \le t + a_0(v.n)$ to see that the indicator function $\mathbf{1}(y < t + a_0(v.n))$ in the density of $\mathcal{P}(v.n)$ is identically $1$ on $\Delta_{\alpha,s}$.  
Now we convolve with the distribution of $a_{\alpha(e)}(v.n)$ to obtain the distribution of $\Theta(e)$:
$$ \mathbb{P}(\Theta(e) > \alpha(e) + s) = \int_0^\infty \sech^2\tfrac{a+s}{2}\,\cosh^2\tfrac{a}{2}\,.\,\tfrac{1}{2}\sech^2\tfrac{a}{2}\,da  = 1 - \tanh\tfrac{s}{2}\,.$$
Hence $\Theta(e)-\alpha(e) = \Theta(e) + a_0(e)$ is distributed like $\theta_1$, as claimed. 
 
We must also consider the edges from $v$ to its children that are already present in $\mathbf{Live}_0^h$. We condition on $\mathbf{FF}_0^{h-1}$, and consider a  leaf $v$ of $\mathcal{Z}_{h-1}$, with initial age $a_0(v)$. The ages of the offspring of $v$ in the spanning forest $\mathbf{Live}_0^h$ and the ages of the corresponding edges are described by a PRM of pairs $(a_v,a_e) \in [0,\infty)^2$ with density 
$$ \mathbf{1}(0 \le a_e \le a_0(v) \wedge a_v)\,.\, \tfrac{1}{2}\sech^2\tfrac{a_v}{2}\,\,da_v\,da_e\,. $$
Each offspring of $v$ in $\mathbf{Live}_0^h$ is a leaf $v.n$ of $\mathcal{Z}_h$ with its own Poisson process $\mathcal{P}(v.n)$ of potential ignitions. Conditional on all the initial data apart from $\mathcal{P}(v.n)$, the a priori burning time of the edge $e = (v,v.n)$ is distributed as follows:
$$ \mathbb{P}(\Theta(e) > s \,|\, e \in \mathbf{Live}_0^h,\, a_{0}(v.n) = a) = \mathbb{P}(\mathcal{P}(v.n) \cap \Delta_{0,s} = \emptyset)\,,$$ which we computed above to be $\sech^2\tfrac{a+s}{2} \cosh^2\tfrac{a}{2}$.
Differentiating this with respect to $s$ gives the conditional density of $\Theta(e)$. Integrating out the leaf vertex age, we find that the pairs $(a_0(e), \Theta(e))$ for edges $e$ of the form $(v,v.n)$ in $\mathbf{Live}_0^h$ are the atoms of a PRM whose intensity is
$$ \left(\int_0^\infty \tfrac{1}{2}\sech^2\tfrac{a_v+\Theta_e}{2} \tanh\tfrac{a_v+\Theta_e}{2} \mathbf{1}(0 \le a_e \le a_v \wedge a_0(v)) da_v\,\right)\,da_e\,d\Theta_e\,.$$
This density simplifies to 
$$\mathbf{1}(0 \le a_e < a_0(v))\,.\, \tfrac{1}{2}\sech^2\tfrac{a_e+\Theta_e}{2} \,da_e\,d\Theta_e\,.$$

Combining the $(a_0(e),\Theta(e))$ pairs for the edges $e$ from $v$ to height $h$ in $\mathbf{Live}_0^h$ and $\mathbf{Future}_0^h$, we find that they form a Poisson point process on $\mathbb{R}^2$ with intensity
$$ \mathbf{1}(a_e < a_0(v),\,a_e + \Theta_e > 0,\,\Theta_e > 0) \,.\,\tfrac{1}{2} \sech^2(\tfrac{a_e + \Theta_e}{2})\,da_e\,d\Theta_e\,.$$
Changing variables to $(\Theta_e, \Theta_e + a_e)$ we find that $\tilde{\mathcal{P}}(v)$ is a PRM with intensity $\mathbf{1}(y < a_0(v) + t)\,\tfrac{1}{2}\sech^2 \tfrac{y}{2}\, dt\, dy$, as required. 
\end{proof}

Lemma~\ref{L: ignition projection} shows that the laws of $\mathbf{FF}_t^h$ are compatible as $h$ varies. In particular we obtain a sample of $\mathbf{FF}_0^{h-1}$ from a sample of $FF_0^h$ by restricting to $\mathcal{Z}_{h-1}$ and letting the process $\tilde{\mathcal{P}}(v)$ serve as $\mathcal{P}(v)$. The construction causes the (deterministic) evolutions to stay coupled in the sense that the restriction of $\mathbf{FF}_t^h$ to $\mathcal{Z}_{h-1}$ is $\mathbf{FF}_t^{h-1}$ for all $t \ge 0$.  Applying Kolmogorov's consistency theorem, we find that there exists an essentially unique random process $\mathbf{FF}^\infty_t$ on the whole tree $\mathcal{Z}$ whose truncation to $\mathbf{Z}_h$ has the law of the process $\mathbf{FF}_h^t$ for every $h$.  In $\mathbf{FF}^\infty$ there is no explicit ignition process analogous to the processes $\mathcal{P}(v)$, but we have not shown that the evolution of $\mathbf{FF}_t^\infty$ is deterministic, i.e. that $\mathbf{FF}_t^\infty$ is a measurable function of $\mathbf{FF}_0^\infty$.

\subsection{Steady state cluster growth in $\mathbf{FF}_t^\infty$}

\begin{theorem} The connected component of the root in $\mathbf{FF}_t^\infty$ is stationary up to order-preserving relabelling of the vertices by strings. The live cluster of the root in $\mathbf{FF}_t^\infty$ is a version of the stationary multitype Galton-Watson tree process $\mathcal{H}(t)$.\end{theorem}
By the connected component of the root in $\mathbf{FF}_t^\infty$, we mean the connected component using both live and future edges, but not burned edges. By \emph{stationary up to order-preserving relabelling} we mean that if we forget the labelling of the vertices by strings but retain the rooted plane tree structure of $\mathcal{Z}$ and the labelling of vertices and edges by ages, then we obtain a stationary process.  
\begin{proof}
$\mathbf{FF}_t^\infty$ is always isomorphic as a rooted plane tree to $\mathcal{N}$. The restriction of $\mathbf{Live}_t^\infty$ to $\mathbf{FF}_t^\infty$ is a spanning forest whose connected components are finite at all times. The restriction of $\mathbf{FF}_t^\infty$ to the connected component of the root is a forest fire process in which when an edge burns it is deleted along with the entire subtree of $\mathbf{FF}_t^\infty$ that it cuts off from the root. For stationarity, we have to show that the restriction of $\mathbf{Live}_t$ to the connected component of the root in $\mathbf{FF}_t^\infty$ is stationary. This suffices because after dropping the string-labelling the process $\mathbf{Future}_t$ is stationary and for each fixed time $t$, $\mathbf{Future}_t$ is independent of $\left(\mathbf{Live}_s\right)_{s \in [0,t]}$. 

For each $v \in \mathcal{Z}$ denote by $\mathcal{L}_t(v)$ the cluster of $v$ in $\mathbf{Live}_t$. Denote by $\mathrm{pa}(v)$ the parent of $v$ in $\mathcal{Z}$. From the construction (and in particular the uniqueness part of Kolmogorov's consistency theorem), we see that for each $v \in \mathcal{Z}$ other than the root, the conditional distribution of $\mathcal{L}_t(v)$ given $\alpha(\mathrm{pa}(v),v)) > t$ is the distribution of $\mathcal{L}_t(\emptyset)$. Hence $\mathcal{L}_t(\emptyset)$ is a geometric cluster growth process in a (possibly dynamic) environment of clusters identically distributed to $\mathcal{L}_t(\emptyset)$. We already know from \S\ref{S: geometric cluster growth} that the steady state cluster growth process is a steady state solution of this problem, and we are starting at time $0$ with $\mathcal{L}_0(\emptyset)$ distributed like the steady state cluster. But the question of whether this recursive property of $\mathcal{L}_t(\emptyset)$ forces it to be stationary is a chicken-and-egg problem which we have not solved. 

Instead we will prove only that $FF_t^\infty$ is stationary, by showing inductively that for each $h \ge 0$, the live cluster of the root in $\mathbf{FF}_t^h$ is a stationary process, having the distribution of the truncation of $\mathcal{H}(t)$ at height $h$. This will also imply the second statement of the theorem, by the uniqueness part of Kolmogorov's consistency theorem. Denote the live cluster of $v$ in $\mathbf{FF}_t^h$ by $\mathcal{L}_t^h(v)$.  By construction, for $\emptyset \neq v \in \mathcal{Z}_h$ the conditional distribution of $\mathcal{L}_t^h(v)$ given that $\alpha(\mathrm{pa}(v),v) > t$ is equal to the distribution of $\mathcal{L}_t^{h-|v|}(\emptyset)$. 

To begin the induction, $\mathbf{FF}_t^0$ consists of the root vertex, labelled by an age $a_t(\emptyset)$ and equipped with an ignition process $\mathcal{P}_t(\emptyset)$. As remarked in the construction, the resulting ignitions are the arrival times of a stationary renewal process whose interarrival time is distributed like $t_\infty$, and  $a_t(\emptyset)$ records the spent time since the last ignition; the conditional distribution of $\mathcal{P}_t(\emptyset)$ given $a_t(\emptyset)$ is a Poisson point process that is conditionally independent of the renewal process before the last ignition. Hence $\mathcal{L}_t^0(\emptyset)$ is stationary with the distribution of the truncation to height $0$ of $\mathcal{H}(t)$, as required.

The induction step is simple: consider the dynamics of the truncation $\mathcal{H}^h(t)$ of $\mathcal{H}(t)$ at height $h$, which is stationary by its construction. At time $0$ it is distributed like the truncation $H^h$ of $H$ at height $h$. At rate $1$ an edge arrives joining the root to a new child, whose subtree is distributed like $H^{h-1}$ and subsequently evolves like $\mathcal{H}^{h-1}(t)$. The first of these child subtrees to burn determines the burning time of the root, at which time $\mathcal{H}^h(t)$ becomes a singleton, and then continues to evolve in the same way. Likewise $\mathcal{L}_0^h(\emptyset)$ is distributed like $H^h$ and subsequently at rate $1$ an edge arrives joining the root to a new child, whose subtree is distributed like $\mathcal{L}_t^{h-1}$; the first of these child subtrees to burn determines the next burning time of the root, at which time $\mathcal{L}_t^h(\emptyset)$ becomes a singleton, and then continues to evolve in the same way. The induction hypothesis ensures that these descriptions agree, and we are done.
\end{proof}

\subsection{A stationary forest fire model on $\mathcal{Z}$}
Because of the relatively simple way in which we sampled $\mathbf{Live}_0$, the whole infinite forest fire model $\mathbf{FF}_t^\infty$ is not a stationary process. Even if we drop the labelling of the vertices by strings and think of $\mathcal{Z}$ as an infinite rooted plane tree, (meaning that we remember the ordering of the children of each vertex), we do not obtain a stationary process. Indeed, in $\mathbf{FF}_0^\infty$ almost surely all the vertices have distinct ages. However, in $\mathbf{FF}_t^\infty$ at any positive time there have already been fires, and infinitely many of the vertices involved in any one of those fires will have survived unburned afterwards until time $t$; so there will be infinite sets of vertices sharing the same age.

For a richer stationary forest fire model in which vertices are not deleted at burning times, we have to modify the construction of the initial state to fill in the past. We will only sketch the construction and will not give a proof of stationarity. 

In the following description, all samplings are independent unless otherwise stated. The goal is to construct a rooted plane tree isomorphic to $\mathcal{Z}$ by gluing countably infinitely many independent copies of $\hat{H}^{(0)}$ representing the fires. We glue by identifying vertices. We begin by constructing the past and future fires of the root. To do this, we sample a bi-infinite i.i.d. sequence of one-ended infinite rooted plane trees $\mathbf{Fire}_i(\emptyset): i \in \mathbb{Z}$, each distributed like $\hat{H}^{(0)}$, except that $\mathbf{Fire}_0(\emptyset)$ is size-biased in proportion to the age of the root. We then sample the initial root age $a_0(\emptyset)$ by sampling from $U([0,a])$ where $a$ is the age of the root in $\mathbf{Fire}_0(\emptyset)$. This determines the time of the first fire at the root after time $0$: it is $\theta_1(\emptyset):= a-a_0(\emptyset)$. The time of the last fire of the root before time $0$ is $\theta_0(\emptyset):=-a_0(\emptyset)$. Each tree $\mathbf{Fire}_i(\emptyset)$ has its own root age, and these ages allow us to construct the complete sequence of times $\theta_i(\emptyset): i \in \mathbb{Z}$ at which the root burns, so that $\theta_{i+1}(\emptyset)-\theta_{i}(\emptyset)$ is the age of the root in $\mathbf{Fire}_i(\emptyset)$. The sequence $\left(\theta_i(\emptyset)\right)_{i \in \mathbb{Z}}$ has the distribution of a stationary renewal process extending infinitely into the past and the future. We glue all of the rooted plane trees $\mathbf{Fire}_i(\emptyset)$ together, in order, at their roots to create an infinite rooted plane tree, whose root we label with the empty string $\emptyset$. The root has edges with the order type of the integers and every other vertex has only finite degree. Each edge $e$ in $\mathbf{Fire}_i(\emptyset)$ is equipped with an age and hence acquires an arrival time $\alpha(e)$ by subtracting this age from $\theta_{i+1}(\emptyset)$. Therefore we can label the children of the root by integers, in order of the arrival times of their edges to the root. We do this so that the first child whose edge to the root has non-negative arrival time has label $0$.

We have now constructed all of the edges incident on $\emptyset$ and all of the vertices at height $0$ or $1$.  We now continue the construction inductively with respect to height. Suppose we have constructed all of the edges and vertices up to height $h$. For each vertex $v$ at height $h$ we already have a unique fire in which the vertex $v$ was constructed: we let $\mathbf{Fire}_0(v)$ be this fire. This provides a pair of consecutive fire times of $v$: the burning time of $\mathbf{Fire}_0(v)$ and the preceding burning time of $v$. 
We sample a bi-infinite i.i.d. sequence of fires $\mathbf{Fire}_i(v): i \in \mathbb{Z}\setminus\{0\}$, each distributed like $\hat{H}^{(0)}$. 
We identify the root of each $\mathbf{Fire}_i(v)$, $(i \neq 0)$, with $v$. The root ages give the inter-fire times of $v$ in the past and the future, and therefore yield the complete sequence of burning times of $v$.
 Now we now have constructed all the children $v.n$ of $v$, ordered by the time of arrival of the edge $(v,v.n)$, with the children labelled so that the edge to $v.0$ is the first one to arrive after time $0$.  

Once the inductive construction is finished, we have constructed a copy of $\mathcal{Z}$ in which every vertex has a bi-infinite sequence of burning times distributed like a stationary renewal process with interarrival times distributed like $t_\infty$. Every edge $e$ has an arrival time $\alpha(e)$ and a burning time $\theta(e)$. The edge set of $\mathcal{Z}$ is partitioned into countably infinitely many fires; each of these is a one-ended infinite rooted plane tree consisting of edges all of which have the same burning time.

At any time $t$ we can partition the edge set into three spanning forests: \begin{itemize}\item $\mathbf{Future}_t$, the spanning forest consisting of all edges $e$ such that $\alpha(e) > t$, \item $\mathbf{Live}_t$, consisting of all edges $e$ such that $\alpha(e) \le t < \theta(t)$, and
\item $\mathbf{Burned}_t$, consisting of all edges $e$ such that $\theta(e) \le e$.
\end{itemize}
At each time $t$ we can label each vertex with an age,
$$ a_t(v) := t - \sup((-\infty,t] \cap \{\theta_i(v): i \in \mathbb{Z}\})\,.$$
Likewise we can label each edge of $\mathbf{Live}_t$ with an age,
$$ a_t(e) := t-\alpha(e)\,.$$
When we forget the labelling of vertices by strings, but keep the rooted plane tree structure of $\mathcal{Z}$, the spanning forest $\mathbf{Live}_t$ labelled with vertex and edge ages constitutes a stationary process. This follows easily from the fact that we started with a stationary renewal process of fires at the root.

Also by construction the connected component of the root $\emptyset$ in $\mathbf{Live}_t$ is a version of the dynamic multi-type Galton-Watson tree $\mathcal{H}(t)$, in other words a version of the steady state cluster growth process.

What is the dynamics of the stochastic process $\mathbf{Live}_t$? At each vertex $v$, the arrival times of the edges to the children of $v$ form a Poisson process of rate $1$; these arrival processes are independent at distinct vertices, and the arrivals are independent of $\mathbf{Live}_t$. Each edge burns just once, at the time when its connected component in $\mathbf{Live}_t$ explodes.

By considering only the connected component of $\emptyset$ in $\mathbf{Live}_t \cup \mathbf{Future}_t$, we obtain a stationary forest fire process in which both edges and vertices are deleted when they burn.

\end{document}